\documentclass[11pt]{article}
\usepackage{amsmath, amssymb, theorem, latexsym, epsfig}
 \usepackage{ifpdf}
\numberwithin{equation}{section}

\theoremstyle{plain}
\theorembodyfont{\itshape}
\newtheorem{theorem}{Theorem}[section]
\newtheorem{proposition}[theorem]{Proposition}
\newtheorem{lemma}[theorem]{Lemma}
\newtheorem{corollary}[theorem]{Corollary}

\theorembodyfont{\rmfamily}
\newtheorem{definition}[theorem]{Definition}

\newtheorem{example}[theorem]{Example}
\newtheorem{remark}[theorem]{Remark}

\newtheorem{conjecture}[theorem]{Conjecture}
\newenvironment{proof}{{\noindent \textbf{Proof}\,\,}}{\hspace*{\fill}$\Box$\medskip}

\newcommand{\Sing}{\operatorname{Sing}}

\def\cc{\mathbb C}
\def\mcf{\mathcal F}
\def\oc{\overline{\cc}}

\def\cp{\mathbb{CP}}
\def\wt#1{\widetilde#1}
\def\rr{\mathbb R}
\def\var{\varepsilon}

\def\mcr{\mathcal R}

\def\nn{\mathbb N}

 \def\mcb{\mathcal B}
 \def\zz{\mathbb Z}
 \def\rp{\mathbb{RP}}
 \def\mcb{\mathcal B}
 \def\Sing{\operatorname{Sing}}
\def\La{\Lambda}
\def\ii{\mathbb I}

  \def\diag{\operatorname{diag}}
 \def\la{\lambda}
\def\mcc{\mathcal C}

\def\mcm{\mathcal M}
\def\mcq{\mathcal Q}

\def\mcr{\mathcal R}
\def\mcp{\mathcal P}
\def\mcn{\mathcal N}

\def\mcd{\mathcal D}
\def\Xi{\mathcal Z}

\def\qq{\mathbb Q}
\def\d{\partial}
\def\dd{\partial^2}
\def\gpq{\gamma_{p,q}}

\def\er{\eta_{\rho}}
\def\pqr{(p,q;\rho)}
\def\tro{\theta_{\rho}}
\def\tr{\theta_r}

\def\mo{\operatorname{mod}}
\def\ii{\mathbb I}

\title{On rationally integrable planar dual and projective billiards}
\author{Alexey Glutsyuk\thanks{CNRS, France (UMR 5669 (UMPA, ENS de Lyon), UMI 2615 (ISC J.-V.Poncelet)). E-mail: 
aglutsyu@ens-lyon.fr} \thanks{HSE University, Moscow, Russia} \thanks{Kharkevich Institute for Information Transmission Problems (IITP RAS), Moscow} \thanks{Partially supported by Laboratory of Dynamical Systems and Applications, HSE University, of the Ministry of science and higher education of  RF grant  No 075-15-2019-1931} \thanks{Supported by part by RFBR grants 16-01-00748, 16-01-00766, 20-01-00420}
 \thanks{This material is partly based upon work supported by the National Science Foundation under Grant No. 1440140, while the author was in residence at the Mathematical Sciences Research Institute in Berkeley, California, during the period 
 August--September 2018.}}
\begin{document}
\maketitle
\begin{abstract}  
A {\it caustic} of a strictly  convex planar bounded billiard is a smooth curve 
whose tangent lines are reflected from the billiard boundary to its tangent lines. 
The famous Birkhoff Conjecture  states that if the billiard 
 boundary has an inner neighborhood foliated by closed caustics, then the billiard is an ellipse. It was studied by many mathematicians, including H.Poritsky, 
 M.Bialy, S.Bolotin, A.Mironov, V.Kaloshin, A.Sorrentino and others. In the paper we study  its following generalized {\it dual} version  stated by S.Tabachnikov. Consider a closed smooth strictly convex   curve $\gamma\subset\rp^2$ 
 equipped with a {\it dual  billiard structure:} a family of non-trivial 
 projective involutions acting on its projective tangent lines and fixing the tangency points. 
 {\it Suppose that its outer neighborhood admits a foliation by closed curves (including $\gamma$) such that 
 the involution of each tangent line permutes 
 its intersection points with every leaf. Then $\gamma$ and the leaves are conics forming a pencil.} 
 We prove positive answer in the case, when the curve $\gamma$ is $C^4$-smooth and the 
 foliation  admits a rational first integral. To this end, we show that each 
 $C^4$-smooth germ  of  curve carrying a rationally integrable dual  
 billiard structure is a conic and classify   rationally integrable dual  billiards on (punctured) conic. 
 They include  the dual  billiards  induced by pencils of conics,  two infinite series of exotic 
dual billiards and five more  exotic ones.
\end{abstract}
 \tableofcontents
\section{Introduction}
\subsection{Main results: classification of rationally integrable dual  planar 
billiards}
The famous Birkhoff Conjecture deals with a  billiard 
in a bounded planar domain $\Omega\subset\rr^2$ with smooth strictly convex boundary. 
 Recall that its {\it caustic} is a curve $S\subset\rr^2$ 
 such that each tangent line to $S$ is reflected from the boundary $\partial\Omega$ 
  to a line tangent to $S$. 
 A billiard $\Omega$  is called {\it Birkhoff caustic-integrable,} if a neighborhood of its boundary in $\Omega$ is foliated by closed caustics, and the boundary $\partial\Omega$ is a leaf of this foliation. 
 It is well-known that each elliptic billiard is integrable:  ellipses confocal to the boundary are caustics, 
 see \cite[section 4]{tab}. 
 The {\bf Birkhoff Conjecture}  states the converse: {\it the only Birkhoff caustic-integrable 
 convex bounded planar 
 billiards with smooth boundary are ellipses.}\footnote{This  conjecture,  attributed  to 
 G.Birkhoff,  was first mentioned   in  print    in the paper \cite{poritsky} by  H. Poritsky, who worked with Birkhoff as a post-doctoral fellow 
 in late 1920-ths.} See its brief survey  in  Subsection 1.5. 
 
 S.Tabachnikov suggested its generalization  to projective billiards introduced by himself in 1997 in \cite{tabpr}. 
 See the following definition and conjecture.

 \begin{definition} \cite{tabpr} A {\it projective billiard} is a smooth planar curve $C\subset\rr^2$ equipped with a transversal line field $\mcn$. 
 For every $Q\in C$ the {\it projective billiard reflection involution} at $Q$ acts on the space of lines through $Q$ as the affine involution 
 $\rr^2\to\rr^2$ that fixes the points of the tangent line to $C$ at $Q$, preserves the line $\mcn(Q)$ and acts on $\mcn(Q)$ as central symmetry 
 with respect to the point\footnote{In other words, two lines $a$, $b$ through $Q$ are permuted by reflection at $Q$, if  and only if 
 the quadruple of lines $T_QC$, $\mcn(Q)$, $a$, $b$ is harmonic: there exists a projective involution of the space $\rp^1$ of lines through 
 $Q$ that fixes $T_QC$, $\mcn(Q)$ and permutes $a$, $b$.} $Q$. 
 In the case, when $C$ is a strictly convex closed curve,  the {\it projective billiard map} acts on the {\it phase cylinder:} 
 the space of oriented lines intersecting $C$. It sends an oriented line to its image under the above reflection involution at its last point 
 of intersection with $C$ in the sense of orientation. See Fig. 1.
 \end{definition} 
 \begin{figure}[ht]
  \begin{center}
   \epsfig{file=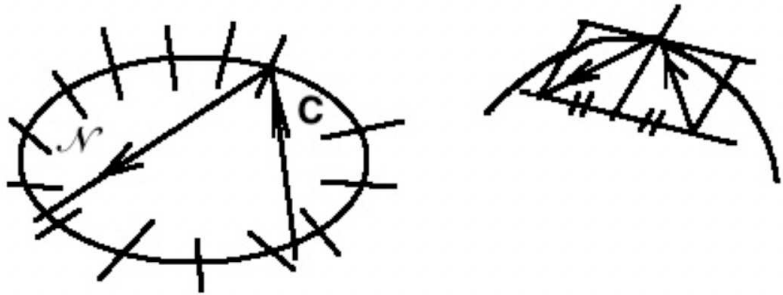, width=28em}
    \caption{The projective billiard reflection.}
    \label{fig:0}
  \end{center}
\end{figure}
\begin{example} \label{exconst} 
A usual Euclidean planar billiard is a projective billiard with transversal line field being 
normal line field. 
\end{example}
\begin{example} \label{exconst2} Each simply connected complete Riemannian surface of constant curvature is  isometric (up to  constant factor) to one of the  two-dimensional space forms:   the  Euclidean plane, the unit sphere, 
 the hyperbolic plane. {\it Any billiard in the hyperbolic plane (hemisphere) is  isomorphic to a projective billiard,} see \cite{tabpr}. 
Namely, each space form is represented by a  hypersurface $\Sigma$ in the space $\rr^3$ equipped with appropriate quadratic form 
$$<Ax,x>, \ \  <x,x>:=x_1^2+x_2^2+x_3^2,$$
$$A \text{ is a symmetric 3x3-matrix called {\it space form matrix}}:$$

Euclidean plane: $A=\diag(1,1,0)$,  $\Sigma=\{ x_3=1\}$.

Sphere: $A=Id$, $\Sigma=\{<x,x>=1\}$ is the unit sphere.

Hyperbolic plane: $A=\diag(1,1,-1)$, $\Sigma=\{ <Ax,x>=-1, \ x_3>0\}.$

\noindent The metric of the surface $\Sigma$ is induced by the quadratic form $<Ax,x>$. 
Its geodesics are the sections of the surface 
$\Sigma$ by two-dimensional vector subspaces in $\rr^3$. The billiard in a  domain $\Omega\subset\Sigma_+:=\Sigma\cap\{ x_3>0\}$  
is defined by reflection of geodesics from its boundary. The tautological projection $\pi:\rr^3\setminus\{0\}\to\rp^2$ sends $\Omega$ diffeomorphically 
to a domain in the affine chart $\{ x_3=1\}$. It sends billiard orbits in $\Omega$ to orbits of the projective billiard on $C=\pi(\partial\Omega)$ 
with the transversal line field $\mcn$ on $C$ being  
the image of the normal line field to $\partial\Omega$ under the differential $d\pi$. The projective billiard on $C$  is 
a space form billiard, see the next definition.
\end{example}
\begin{definition}  \label{pspform} Let $A$ be a space form matrix. Let $C$ be a curve in an affine chart in 
$\rp^2$. Let $\mcn$ be the transversal line field on $C$ defined as follows.

a) Case, when  $A=\diag(1,1,0)$. Then $\mcn$ is the normal line field to $C$ in the affine chart $\{ x_3\neq0\}$. 

b) Case, when $\det A\neq0$, i.e., $A=\diag(1,1,\pm1)$. Then for every $Q\in C$ 
the two-dimensional subspaces in $\rr^3$ projected to the lines tangent to $T_QC$ and $\mcn(Q)$ are orthogonal 
with respect to the scalar product $<Ax,x>$.

Then the projective billiard defined by $\mcn$ is called a {\it space form billiard\footnote{A space form projective billiard with matrix 
$A=\diag(1,1,-1)$ is not necessarily the projection of a billiard in the hyperbolic plane $\Sigma=\Sigma_+$. Some its part may lie in the projection to $\rp^2$ of 
the de Sitter cylinder $\{<Ax,x>=1\}$, where the quadratic form $<Ax,x>$ defines a pseudo-Riemannian metric of constant curvature.}.}
\end{definition}

 The definitions of  caustic and  Birkhoff integrability for 
 projective billiards repeat the above definitions given for classical billiards. 
 
 \begin{conjecture}\label{conjt} {\bf (S.Tabachnikov)} In every Birkhoff integrable projective billiard  its boundary and closed caustics forming a foliation 
 are ellipses whose projective-dual conics form a pencil. 
 \end{conjecture} 
 
 Below we state the dual version of  the Tabachnikov's Conjecture (2008, \cite{tab08}) and present partial positive results. 
To do this, 
consider  $\rr^2_{x_1,x_2}$ as the plane $\{ x_3=1\}\subset\rr^3_{x_1,x_2,x_3}$  
 identified with the corresponding affine chart in $\rp^2_{[x_1:x_2:x_3]}$. The {\it orthogonal polarity} 
 sends a two-dimensional vector subspace $W\subset\rr^2$ to its Euclidean-orthogonal  subspace 
 $W^\perp$. The corresponding {\it projective  duality} (also called orthogonal polarity) is the map 
 $\rp^{2*}\to\rp^2$ sending lines to points so that the tautological projection  of each punctured two-dimensional 
 subspace $W\setminus\{0\}\subset\rr^3$ (a line $L$)  is sent to the projection  of its punctured orthogonal complement $W^\perp\setminus\{0\}$ (called its dual point and denoted by $L^*$). The line dual to a point $P$ will be denoted by $P^*$. To each curve $C\subset\rr^2$ we associate the {\it dual curve} $\gamma=C^*\subset\rp^2$ 
 consisting of those points that are dual to the tangent lines to $C$. 
 
 Let now a planar curve $C$ be equipped with a projective billiard structure: a transversal  line field $\mcn$. For every point $Q\in C$ let $L_Q$ 
 denote the projective tangent line to $C$ at $Q$ in the ambient projective plane $\rp^2\supset\rr^2$. The projective duality sends the 
 space $\rp^1_Q$ of lines through $Q$ to the projective line $Q^*$ dual to $Q$. The  line $Q^*$ is tangent to $\gamma$ at the point $P=L_Q^*$ dual to 
 $L_Q$. {\it The duality "line $\mapsto$ point" conjugates the projective billiard involution acting on $\rp^1_Q$ with a non-trivial projective involution 
 $\sigma_P:L_P\to L_P$ fixing $P$ and the point  dual to $\mcn(Q)$.}
  Thus,  {\it the duality transforms a projective billiard on $C$ to a dual billiard on $\gamma=C^*$,} see the next definition. 

 \begin{definition} 
 A  {\it dual  billiard structure} on a smooth curve $\gamma\subset\rp^2$ is a family of non-trivial projective 
 involutions $\sigma_P:L_P\to L_P$ fixing $P$. 
 \end{definition}
 
 \begin{remark}  Let a projective billiard on $C$ have a strictly convex closed caustic $S$. Then its dual curve 
 $S^*$ is also strictly convex and closed, and  for every $P\in\gamma=C^*$ the  
 dual billiard 
 involution  $\sigma_P:L_P\to L_P$ permutes the two points of intersection  $L_P\cap S^*$. See Fig. 2. A curve $S^*$ satisfying the 
 latter statement is called an {\it invariant curve for the dual billiard.}
  \end{remark}
 
  \begin{figure}[ht]
  \begin{center}
   \epsfig{file=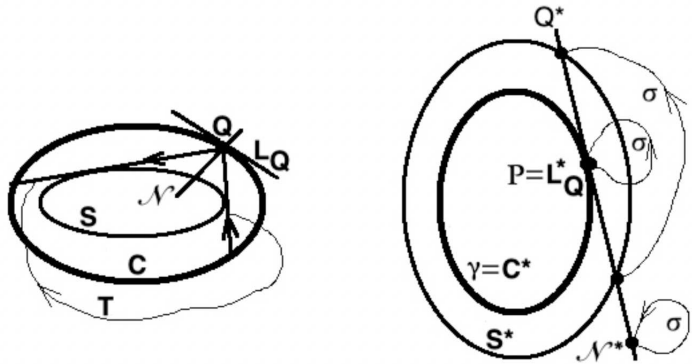, width=25em}
    \caption{The projective billiard reflection involution $T$ acting on lines through a point $Q\in C$ and the dual involution $\sigma=\sigma_{P}$ acting on  the dual line $Q^*$ tangent to the dual curve $\gamma=C^*$ at the point $P=L_Q^*$.}
    \label{fig:0}
  \end{center}
\end{figure}

\begin{definition} \label{dpint} 
 A dual  billiard on a strictly convex closed curve $\gamma$ is  {\it integrable,} if there exists a $C^0$-foliation by closed strictly convex 
 invariant curves on a neighborhood of $\gamma$ on its concave side, with $\gamma$ being a leaf. See Fig. 3. 
 \end{definition}
   \begin{figure}[ht]
  \begin{center}
   \epsfig{file=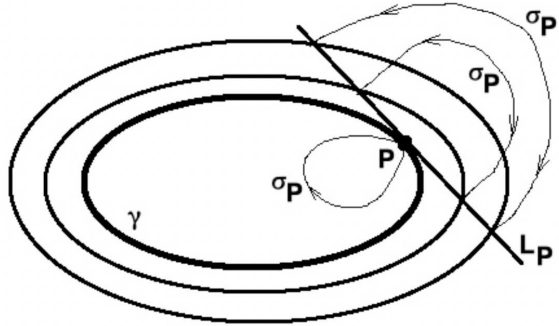, width=20em}
   \caption{An integrable dual  billiard structure}
        \label{fig:5}
  \end{center}
\end{figure}

\begin{conjecture} \label{conjtd} {\bf (S.Tabachnikov \cite{tab08}); dual to Conjecture \ref{conjt}).}  For every integrable dual billiard 
the underlying curve and the corresponding invariant curves forming foliation are conics forming a pencil.
\end{conjecture} 

\begin{remark} A projective billiard  on a  strictly convex closed curve is integrable, if and only if so is its dual billiard. 
The {\it outer} dual billiard in $\rr^2$,  
 with $\sigma_P:L_P\to L_P$ being the central symmetry with respect to the tangency point $P$, is dual to the centrally-projective billiard, 
 whose transversal  field consists of lines passing through the origin \cite{tabpr}. 
Thus, {\it Conjecture \ref{conjtd} would imply the Birkhoff Conjecture and its versions on surfaces of constant curvature and for 
outer billiards} (as observed in \cite{tab08}), see Examples \ref{exconst}, \ref{exconst2}. 
\end{remark}

One of the main results of the present paper is the following theorem.

\begin{theorem} \label{tclosed} Let $\gamma\subset\rp^2$ be a $C^4$-smooth strictly convex closed planar curve 
 equipped with an  integrable dual   billiard structure. Let the corresponding foliation by invariant curves admit a rational first 
 integral. Then its leaves, including $\gamma$, are conics forming a pencil.  
 \end{theorem}

Below we state a more general result for $\gamma$ being a germ. To do this, let us introduce the following definition.
 \begin{definition} A dual  billiard on a (germ of)   curve  $\gamma\subset\rp^2$ 
  given by involution family $\sigma_P:L_P\to L_P$ 
 is called {\it rationally integrable,}  if there exists a non-constant rational 
 function $R$ on $\rp^2$ whose restriction to $L_P$ is $\sigma_P$-invariant for every $P\in\gamma$: 
 $R\circ\sigma_P=R$ on $L_P$.  
  \end{definition}
  \begin{example} Let a dual billiard on $\gamma$ be {\it polynomially} integrable: the above integral $R$ is  polynomial in some 
  affine chart $\rr^2$. Then for every $P\in\gamma\cap\rr^2$ the involution $\sigma_P$ fixes the intersection point of the line $L_P$ with the 
  infinity line, and hence, is the central symmetry $L_P\to L_P$ with respect to the tangency point $P$. 
  Thus, the dual billiard in question is a polynomially integrable outer billiard. It is known that in this case the underlying curve is a conic:  
  stated as a conjecture and proved  in \cite{tab08} under some non-degeneracy assumption; proved in full generality in   \cite{gs}.  
  \end{example}

\begin{example} \label{tabobs} 
{\bf (S.Tabachnikov's observation).} Let $A$, $B$ be real symmetric $3\times3$-matrices, $B$ be non-degenerate. Consider 
the pencil of conics $\mcc_\la:=\{<(B-\la A)x,x>=0\}$; set $\gamma:=\mcc_0=\{<Bx,x>=0\}$. 
The set of those points in $\cp^2$ that lie in complexifications of all  $\mcc_\la$ simultaneously will be called the  {\it basic set} 
of the pencil  and denoted by $\mcb(\mcc)$.  
 {\it For every $P\in\gamma^o:=\gamma\setminus\mcb(\mcc)$ 
   the involution  permuting the two complex points of intersection $\mcc_\la\cap L_P$ for each $\la$ 
   is a well-defined real projective involution $\sigma_P:L_P\to L_P$.} 
    This   yields a dual  billiard  on $\gamma^o$, which will be called {\it dual billiard 
   of conical pencil type.} It is known to be rationally integrable with a quadratic  integral: the ratio of quadratic polynomials vanishing 
   on some two given conics of the pencil. 
   \end{example}

 \begin{definition} Two  dual   billiard structures on two (germs of) curves $\gamma_1$, 
 $\gamma_2$ in 
 $\rp^2$  are  {\it real-projective equivalent,} if there exists a 
 projective transformation $\rp^2\to\rp^2$ sending $\gamma_1$ to $\gamma_2$ and transforming one structure to the other one. 
 (Projective equivalence  preserves rational integrability.) Real-projective equivalence of projective 
 billiards is defined analogously. 
 \end{definition}
 
 The main result of the paper is the next theorem stating that for every rationally integrable dual billiard 
 the underlying curve is a conic, the dual billiard structure extends to 
 the  conic punctured in at most four points, and  classifying  rationally integrable
   dual billiards on punctured conic.  Unexpectedly,  {\it there are infinitely 
 many exotic, non-pencil rationally integrable dual billiards on punctured conic, with integrals of arbitrarily high degrees.} 
 
   \begin{theorem} \label{tgerm} Let $\gamma\subset\rr^2\subset\rp^2$ be a $C^4$-smooth non-linear  germ of curve  equipped with a rationally integrable dual   billiard structure. 
 Then $\gamma$ is a conic, and the dual   billiard structure has one of the 
 three following types (up to real-projective equivalence):
 
  1) The  dual billiard  is of conical pencil type and has a quadratic integral.
    
  2) There exists an affine chart $\rr^2_{z,w}\subset\rp^2$ 
  in which $\gamma=\{ w=z^2\}$ and such that for every $P=(z_0,w_0)\in\gamma$ the 
  involution $\sigma_P:L_P\to L_P$ is given by one of the following formulas: 
  
  a) In the coordinate 
  $$\zeta:=\frac z{z_0}$$
  $$\sigma_P:\zeta\mapsto\eta_\rho(\zeta):=\frac{(\rho-1)\zeta-(\rho-2)}{\rho\zeta-(\rho-1)},$$
\begin{equation} \rho=2-\frac 2{2N+1}, \ \text{ or } \ \rho=2-\frac1{N+1} \  \text{ for some 
  } N\in\nn.\label{rhoval}\end{equation} 
  
  b) In the coordinate 
  $$u:=z-z_0$$
  \begin{equation} \sigma_P: u\mapsto-\frac u{1+f(z_0)u},\label{sigmaef}\end{equation}
   \begin{equation} f=f_{b1}(z):=\frac{5z-3}{2z(z-1)} \text{ (type 2b1))}, \  \text{ or } \ 
   f=f_{b2}(z):=\frac{3z}{z^2+1} \text{ (type 2b2))}.\label{sigma2}\end{equation}

  c) In the above coordinate $u$ the involution $\sigma_P$ takes the form (\ref{sigmaef}) with 
   \begin{equation} f=f_{c1}(z):=\frac{4z^2}{z^3-1} \text{ (type 2c1))}, \  \text{ or } \ 
   f=f_{c2}(z):=\frac{8z-4}{3z(z-1)} \text{ (type 2c2))}.\label{sigma3}\end{equation}
   
   d) In the above coordinate $u$ the involution $\sigma_P$ takes the form (\ref{sigmaef}) with 
   \begin{equation} f=f_d(z)=\frac4{3z}+\frac1{z-1}=\frac{7z-4}{3z(z-1)} \ \ \text{ (type 2d)}.\end{equation}
  \end{theorem}

  \medskip

  {\bf Addendum to Theorem \ref{tgerm}.} {\it Every dual   billiard structure 
  on $\gamma$ of type 2a) has a rational first integral $R(z,w)$ of the form} 
  \begin{equation}R(z,w)=\frac{(w-z^2)^{2N+1}}{\prod_{j=1}^N(w-c_jz^2)^2}, \ \ 
  c_j=-\frac{4j(2N+1-j)}{(2N+1-2j)^2}, \ \text{ for } \rho=2-\frac2{2N+1};\label{exot1}
  \end{equation}
  \begin{equation}R(z,w)=\frac{(w-z^2)^{N+1}}{z\prod_{j=1}^N(w-c_jz^2)}, \ \ 
  c_j=-\frac{j(2N+2-j)}{(N+1-j)^2}, \ \text{ for } \rho=2-\frac1{N+1}.\label{exot2}\end{equation}
   {\it The dual billiards of types 2b1) and 2b2) have respectively the integrals}
  \begin{equation}R_{b1}(z,w)=\frac{(w-z^2)^2}{(w+3z^2)(z-1)(z-w)}, 
  \label{exo2bnew} \end{equation}
   \begin{equation}R_{b2}(z,w)=\frac{(w-z^2)^2}{(z^2+w^2+w+1)(z^2+1)}.
  \label{exo2bnew2} \end{equation}
  {\it The dual billiards of types 2c1), 2c2) have respectively the integrals}
  \begin{equation}R_{c1}(z,w)=\frac{(w-z^2)^3}{(1+w^3-2zw)^2},\label{exo2b}\end{equation} 
  \begin{equation}R_{c2}(z,w)=\frac{(w-z^2)^3}{(8z^3-8z^2w-8z^2-w^2-w+10zw)^2}.\label{exoc2}\end{equation}
{\it The dual billiard  of type 2d) has the integral}
 \begin{equation}R_{d}(z,w)=\frac{(w-z^2)^3}{(w+8z^2)(z-1)(w+8z^2+4w^2+5wz^2-14zw-4z^3)}.\label{exodd}\end{equation}

  \medskip

We prove the following theorem, which is a unifying complex version  of Theorems \ref{tclosed}, \ref{tgerm}. 
To state it, let us introduce the following definition.

\begin{definition} Consider a regular germ of holomorphic curve $\gamma\subset\cp^2$ at a point $O$. 
A {\it complex (holomorphic or not) dual   billiard} on $\gamma$ is a germ of (holomorphic or not) family 
of complex projective involutions $\sigma_P:L_P\to L_P$, $P\in\gamma$, acting on  
complex projective tangent lines $L_P$ to $\gamma$ at $P$ and fixing $P$.  A complex 
dual   billiard on $\gamma$ is said to be {\it rationally integrable,} if 
there exists a non-constant complex rational function $R$ on $\cp^2$ such that for every $P\in\gamma$ 
the restriction $R|_{L_P}$ is $\sigma_P$-invariant: $R\circ\sigma_P=R$ on 
$L_P$. The definition of complex-projective equivalent complex dual   billiards 
repeats the definition of real-projective equivalent ones with change of real projective 
transformations $\rp^2\to\rp^2$ to complex ones acting on $\cp^2$.
\end{definition}

 \begin{theorem} \label{tcompl} 
 Every regular  germ of holomorphic curve in $\cp^2$ (different from a straight line) 
 equipped with a 
 rationally integrable  complex dual   billiard structure is a conic. 
 Up to complex-projective equivalence, the corresponding billiard structure has one of the types 1), 2a), 2b1), 2c1), 2d) listed in Theorem \ref{tgerm}, with 
 a rational integral as in its addendum. (Here the coordinates $(z,w)$ 
 as in the addendum are complex affine coordinates.) 
  \end{theorem}
 {\bf Addendum to Theorem \ref{tcompl}} {\it The billiards of  types 2b1), 2b2), see (\ref{sigma2}), are 
 complex-projectively equivalent, and so are the billiards of types 2c1) and 2c2). For every 
 $g=b,c$ there exists a complex projective equivalence between the billiards 2g1), 2g2) that sends the integral $R_{g1}$ of the former, see (\ref{exo2bnew}), (\ref{exo2b}) (treated as a rational function 
 on $\cp^2_{[z:w:t]}\supset\cc^2_{z,w}=\{ t=1\}$), to the integral $R_{g2}$ of the latter, see  (\ref{exo2bnew2}), 
 (\ref{exoc2}), up to constant factor.} 
 
 \subsection{Classification of rationally $0$-homogeneously integrable 
 projective billiards with smooth connected boundary}

 Let $\Omega\subset\rr^2_{x_1,x_2}$ be a domain with smooth boundary 
 $\partial\Omega$ equipped with a projective billiard structure 
 (transverse line field).  The {\it projective billiard flow} (introduced in \cite{tabpr}) acts on  $T\rr^2|_\Omega$ analogously to the classical 
 case of Euclidean billiards. Given a point 
 $(Q,v)\in T\rr^2$, $Q\in\Omega$, $v=(v_1,v_2)\in T_Q\rr^2$, the flow moves the point $Q$ 
along the straight line directed by $v$ with the fixed uniform velocity $v$, until it  hits the boundary 
$\partial\Omega$ at some point $H$. Let $v^*\in T_H\rr^2$ denote the image of the velocity 
vector $v$ (translated to $H$) under the projective billiard 
reflection from the tangent line $T_H\partial\Omega$. Afterwards 
the flow moves the point $H$ with the new uniform velocity $v^*$ until its trajectory 
hits the boundary again etc. See Fig. 4 below.

 \begin{figure}[ht]
  \begin{center}
   \epsfig{file=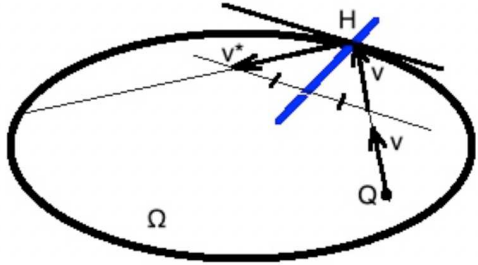, width=17em}
   \caption{Projective billiard flow}
        \label{fig:0}
  \end{center}
\end{figure}

\begin{remark} The flow in a Euclidean planar billiard always has a trivial first integral $||v||^2$. But it is not a first integral in  
a generic projective billiard. 
It is a well-known folklore fact that {\it Birkhoff integrability of a Euclidean planar billiard  with strictly convex closed boundary 
is equivalent to the existence of a non-trivial first integral of the billiard flow independent with $||v||^2$ on a neighborhood of the unit tangent bundle to 
$\partial\Omega$ in $T\rr^2|_\Omega$. }
\end{remark}

 Billiard flows in  space forms of constant curvature  
 and their integrability were studied by many mathematicians, including 
A.P.Veselov \cite{veselov, veselov2}, S.V.Bolotin \cite{bolotin, bolotin2}  (both in any dimension), M.Bialy and A.E.Mironov \cite{bm, bm2, bm3, bm4, bm5}, 
 the author \cite{gl, gl2} and others. A Euclidean planar billiard is called {\it polynomially integrable,} if its flow admits a first integral 
 that is polynomial in the velocity $v$ whose restriction to the unit velocity hypersurface $\{ ||v||=1\}$ is non-constant 
  \cite{bolotin, bolotin2, kozlov, bm}, \cite[definition 1.1]{gl2}.  
 S.V.Bolotin suggested the {\it polynomial version} of Birkhoff Conjecture stating that {\it if a billiard in a strictly 
 convex bounded planar domain with $C^2$-smooth boundary is polynomially integrable, then the billiard boundary is an ellipse,} 
 together with its versions on the sphere  and on the hyperbolic plane. 
 Now this is a theorem: a joint result of M.Bialy, A.E.Mironov and the author of the present paper \cite{bm, bm2, gl, gl2}.
 Here we present a version of this result for rationally integrable projective billiard flows, see the following definition. 
 
 All the results of this subsection will be proved in Section 9.

 \begin{definition} A planar projective billiard is {\it rationally 0-homogeneously 
 integrable,} if its flow admits a non-constant 
 first integral $I$ that is a rational homogeneous 
 function of the velocity with numerator and denominator having the same degrees 
 (called a {\it rational 0-homogeneous integral}): 
 $$I(Q,v)=\frac{I_{1,Q}(v)}{I_{2,Q}(v)}; \ \ \ I_{1,Q}(v), \ I_{2,Q}(v) 
 \text{ are homogeneous polynomials,}$$
$$\deg I_{1,Q}=\deg I_{2,Q}.$$ 
Here we consider that the degrees $\deg I_{j,Q}(v)$ are uniformly bounded.
 \end{definition}

 \begin{example}  It is known that for every polynomially integrable planar billiard 
 the polynomial integral $I_Q(v)$ can be chosen homogeneous  of even degree $2n$, see  
 \cite{bolotin}, \cite[p.118; proposition 2 and its proof on p.119]{bolotin2}, 
 \cite[chapter 5, section 3, proposition 5]{kozlov}. Then the rational function 
 $$\Psi(Q,v):=\frac{I_Q(v)}{||v||^{2n}}$$ 
 is a rational 0-homogeneous integral of the billiard. 
 Thus, {\it every polynomially integrable Euclidean planar billiard is 
 rationally 0-homogeneously integrable.} This also holds for billiards on the sphere 
 and the hyperbolic plane.
\end{example}

 \begin{theorem} \label{tclasspr} Let a projective billiard in a strictly convex bounded domain $\Omega\subset\rr^2$ 
 with  $C^4$-smooth boundary be defined by a continuous transversal line field on $\partial\Omega$ and be rationally 0-homogeneously integrable. 
   Then its boundary is a conic, and the projective billiard is a space form billiard (see Definition \ref{pspform}). 
 \end{theorem}
 
 Theorem \ref{tgermproj} stated below extends Theorem \ref{tclasspr} to germs of planar projective billiards. 
 Each of them is  a germ of 
$C^4$-smooth curve $C$ equipped with a transversal line field $\mcn$. Here $C$ is not necessarily 
 convex. We choose  a side from the curve $C$ and a simply connected domain $U$ adjacent to $C$ from the chosen side. 
Let $Q\in U$ and $v\in T_Q\rr^2$ be such that the ray issued from the point $Q$ in the direction of the vector $v$ 
 intersects $C$, and the distance of the point $Q$ to their first intersection point be equal to $\tau_0||v||$, $\tau_0>0$. 
 Then for  $t_0>\tau_0$ close enough to $\tau_0$ the projective billiard flow maps in times $\tau\in(0,t_0)$ are well-defined on $(Q,P)$. 
 As before, we say that a {\it germ of projective billiard} thus defined is {\it rationally $0$-homogeneously integrable,} if it admits a first integral 
 rational and $0$-homogeneous in $v$ on $T\rr^2|_U$ for some $U$ (small enough) whose degree  
is uniformly bounded in $Q\in U$. 
 
 Before  the statement of Theorem \ref{tgermproj} let us state two preparatory propositions: the first saying that integrability is independent on choice 
 of side; the second reducing classification of germs of integrable projective billiards to classification of germs of  integrable dual billiards given by 
 Theorem \ref{tgerm}. To do this, 
 following S.V.Bolotin \cite{bolotin2}, let us identify the ambient plane $\rr^2$ of a  projective billiard with the plane $\{ x_3=1\}$ in the Euclidean space $\rr^3_{x_1,x_2,x_3}$ and represent a point $x=(x_1,x_2)\in\rr^2$ and a vector $v=(v_1,v_2)\in T_x\rr^2$ by the vectors 
 $$r=(x_1,x_2,1), \ \ v=(v_1,v_2,0)\in\rr^3.$$  
 \begin{proposition} \label{ratmom} 1) Let a germ of projective billiard in $\rr^2_{x_1,x_2}$ with reflection from a $C^2$-smooth germ of curve $C$  
 (or a global planar projective billiard in a connected domain with  $C^2$-smooth boundary)
 be rationally $0$-homogeneously integrable. Then the rational $0$-homogeneous integral can be chosen as a rational $0$-homogeneous function of the moment vector $M$: 
 \begin{equation} M=M(r,v):=[r,v]=(-v_2, v_1, \Delta), \ \Delta=\Delta(x,v):=x_1v_2-x_2v_1.\label{momve}\end{equation}
 
 2) The property of a  projective billiard germ to be  rationally $0$-homogeneously 
 integrable  depends only on the germ of curve with transversal line field and does not depend on the choice of side.
 \end{proposition}
  \begin{proposition} \label{procriter} A planar projective billiard with  $C^2$-smooth boundary 
 is rationally 0-homogeneously 
 integrable, if and only if its dual billiard  is rationally integrable. If $R$ is a rational integral of 
 the dual billiard, written as a $0$-homogeneous rational function  in  homogeneous coordinates  on the 
 ambient projective plane,  then $R([r,v])$ is a $0$-homogeneous rational integral 
 of the projective billiard. 
 \end{proposition}
 \begin{remark} Versions of Propositions \ref{ratmom}, \ref{procriter}  for polynomially integrable billiards on surfaces of constant curvature 
 were earlier proved respectively in the paper \cite{bolotin2} by S.V.Bolotin (Proposition \ref{ratmom})  and in two joint papers by  M.Bialy and A.E.Mironov 
 \cite{bm, bm2} (Proposition \ref{procriter}, see also \cite[theorem 2.8]{gl2}). 
 \end{remark}

 \begin{theorem} \label{tgermproj} Let $C\subset\rr^2_{x_1,x_2}$ be a non-linear $C^4$-smooth germ of 
 curve equipped with a transversal line field $\mcn$. 
  Let the corresponding germ of projective billiard be $0$-homogeneously rationally integrable. 
 Then $C$ is a conic;  the line field $\mcn$ extends to a global analytic transversal line field   on the whole conic $C$  punctured in at most four points;  
the corresponding projective billiard has one of the following types up to projective equivalence.
 
 1) A space form  billiard whose matrix can be chosen $A=\diag(1,1,-1)$. 

2)  $C=\{x_2=x_1^2\}\subset\rr^2_{x_1,x_2}\subset\rp^2$, and 
the  line field $\mcn$  is directed by one of the following vector fields at points of the conic $C$: 
 \smallskip
 
2a) \ \ \ \ \ $(\dot x_1,\dot x_2)=(\rho,2(\rho-2)x_1)$, 
$$\rho=2-\frac2{2N+1} \text{ (case 2a1), \ or }  \ \ \rho=2-\frac1{N+1} \text{ (case 2a2), }  \ \ N\in\mathbb N,$$
 the  vector field 2a) has quadratic first integral $\mcq_{\rho}(x_1,x_2):=\rho x_2-(\rho-2)x_1^2.$

2b1) \ $(\dot x_1, \dot x_2)=(5x_1+3,
2(x_2-x_1))$, \ \ \ 2b2)  \ $(\dot x_1,\dot x_2)=(3x_1, 2x_2-4)$, 
\smallskip

2c1) \ $(\dot x_1, \dot x_2)=(x_2, 
x_1x_2-1)$, \ \ \ \ \ \ 2c2) \ $(\dot x_1, \dot x_2)=(2x_1+1, 
x_2-x_1)$.

2d) \ $(\dot x_1, \dot x_2)=(7x_1+4, 2x_2-4x_1)$.
\end{theorem}
{\bf Addendum to Theorem \ref{tgermproj}.} {\it The projective billiards from Theorem \ref{tgermproj} have the 
following $0$-homogeneous rational integrals:

Case 1): A ratio of two homogeneous quadratic polynomials in $(v_1,v_2,\Delta)$, 
$$\Delta:=x_1v_2-x_2v_1.$$
Case 2a1), $\rho=2-\frac2{2N+1}$:
\begin{equation}\Psi_{2a1}(x_1,x_2,v_1,v_2):=\frac{(4v_1\Delta-v_2^2)^{2N+1}}{v_1^2\prod_{j=1}^N(4v_1\Delta-c_jv_2^2)^2}. \label{r2a1v}\end{equation} 

Case 2a2), $\rho=2-\frac1{N+1}$:
\begin{equation}\Psi_{2a2}(x_1,x_2,v_1,v_2)=\frac{(4v_1\Delta-v_2^2)^{N+1}}{v_1v_2\prod_{j=1}^N(4v_1\Delta-c_jv_2^2)}. \label{r2a2v}\end{equation}
The $c_j$ in (\ref{r2a1v}), (\ref{r2a2v}) are the same, as in (\ref{exot1}) and (\ref{exot2}) respectively.

Case 2b1): 
 \begin{equation}\Psi_{2b1}(x_1,x_2,v_1,v_2)=\frac{(4v_1\Delta-v_2^2)^2}{(4v_1\Delta+3v_2^2)(2v_1+v_2)(2\Delta+v_2)}.\label{r2b1v}\end{equation}
 
 Case 2b2): 
 \begin{equation}\Psi_{2b2}(x_1,x_2,v_1,v_2)=\frac{(4v_1\Delta-v_2^2)^2}{(v_2^2+4\Delta^2+
4v_1\Delta+4v_1^2)(v_2^2+4v_1^2)}.
 \label{r3b2v}\end{equation}
 
 \begin{equation}\text{Case 2c1): } \ \ \ \ \ \ \  \ \ \ \ \Psi_{2c1}(x_1,x_2,v_1,v_2)=\frac{(4v_1\Delta-v_2^2)^3}{(v_1^3+\Delta^3+
 v_1v_2\Delta)^2}.\label{r2c1v}\end{equation}
 
 Case 2c2): 
  \begin{equation}\Psi_{2c2}(x_1,x_2,v_1,v_2)=\frac{(4v_1\Delta-v_2^2)^3}{(v_2^3+2v_2^2v_1+(v_1^2+2v_2^2+5v_1v_2)\Delta+v_1\Delta^2)^2}.
 \label{r2c2v}\end{equation}
 
 Case 2d):}  $\Psi_{2d}(x_1,x_2,v_1,v_2)$
 \begin{equation}=\frac{(4v_1\Delta-v_2^2)^3}{(v_1\Delta+2v_2^2)(2v_1+v_2)(8v_1v_2^2+2v_2^3+(4v_1^2+5v_2^2+28v_1v_2)\Delta+16v_1\Delta^2)}.
 \label{r2dv}\end{equation}
 
 In Subsection 9.3 we prove the following  characterization of space form billiards on conics as projective billiards on conics with  conical caustics.
 
 \begin{proposition} \label{procons} A transversal line field $\mcn$ on a punctured planar regular conic  $C$ defines a projective billiard that is projectively equivalent to 
 a space form billiard, if and only if there exists a regular conic $S\neq C$ such that for every $Q\in C$ the complexified 
 projective billiard reflection at $Q$ permutes the complex lines through $Q$ tangent to the complexified conic $S$.
  \end{proposition} 
 \begin{remark} The latter permutation condition   determines $\mcn$ by $S$ in a unique way. The corresponding 
 projective billiard has a family of conical caustics whose dual conics form a pencil, see \cite{ccs} and \cite[subsection 2.3]{fierobe-th}. 
 \end{remark}
\subsection{Applications to billiards with complex algebraic caustics}
\begin{definition} Let $C\subset\rr^2\subset\rp^2$ be a planar curve equipped with a projective 
billiard structure. For every $Q\in C$ consider the complexification of the billiard reflection involution acting on the space of complex lines through  
$Q$. 
Let $S\subset\cp^2$ be an algebraic curve in the complexified ambient projective plane that contains no straight line. 
We say that $S$ is a {\it complex caustic} of the {\it real billiard} on $C$, if for every $Q\in C$ 
each {\it complex} projective line tangent to $S$ and passing through $Q$ is reflected by the complexified 
 reflection at $Q$ to a line tangent to $S$.  
 \end{definition}
 \begin{remark} The  usual Euclidean  billiard on a strictly convex planar curve $C$ has $C$ as a real caustic: through each its point $Q$ passes the unique 
tangent line to $C$, and it is fixed by the reflection. If $C$ is a conic, then its complexification $C_\cc$ is a complex  
caustic for $C$ for the same reason. But  if $C$ is a higher degree algebraic curve, then 
a priori its complexification $C_\cc$ is not necessarily a complex caustic. 
In this case through a generic point $Q\in C$ passes at least one complex line tangent to $C_\cc$ 
that does not coincide with its tangent line at $Q$. To check, whether $C_\cc$ is a complex caustic, one has to check 
whether the collection of all the complex lines through $Q$ tangent to $C_\cc$ is invariant under the reflection at $Q$. 
This is a non-trivial condition on the algebraic curve $C$. 
\end{remark}

{\bf Open problem.}  Consider Euclidean billiard on a strictly convex closed planar curve $C$. Let $C$ be 
{\it contained in an algebraic curve} and  
 have a  {\it real  caustic contained in an algebraic curve.} Is it true that  $C$ is a conic?
\medskip

A positive answer would imply the particular case of the Birkhoff Conjecture, when the billiard boundary is contained in an algebraic curve.

 We prove the following theorem as an application of results of \cite{gl2}. 
\begin{theorem} \label{talg1} Let $C\subset\rr^2$ be a non-linear $C^2$-smooth connected embedded (not necessarily closed, convex or algebraic) 
curve 
equipped with the structure of standard Euclidean billiard (with the usual reflection law). Let the latter billiard 
have a {\bf complex} algebraic caustic. Then $C$ is a conic. 
\end{theorem}

The next theorem is its analogue for projective billiards. It will be deduced from main results of the present 
paper.
\begin{theorem} \label{talg2} Let $C$ be a non-linear $C^4$-smooth connected  embedded planar curve. Let  
$C$ be equipped with a projective billiard structure having at least two different complex algebraic caustics. 
Then $C$ is a conic.
\end{theorem}
Theorems \ref{talg1} and \ref{talg2} will be proved in Section 10.

 \subsection{Sketch of proof of Theorem \ref{tcompl} and plan of the paper}
 
 First we prove algebraicity of a rationally integrable dual   billiard.
 
 \begin{definition} A {\it singular} holomorphic dual   billiard on a  holomorphic curve $\gamma\subset\cp^2$ 
is a holomorphic 
 dual billiard  structure on the complement of the curve $\gamma$ to a discrete subset of points where the 
 corresponding family of involutions $\sigma_P:L_P\to L_P$ does not extend holomorphically. 
 \end{definition} 
  
 \begin{proposition} \label{proalg}  Let a regular non-linear germ  of holomorphic curve $\gamma\subset\cp^2$ 
 carry a complex (not necessarily holomorphic) rationally integrable dual   billiard structure with 
 a rational integral $R$. Then 1) $R|_\gamma\equiv const$,  and thus, $\gamma$ is contained in an irreducible 
 algebraic curve, which will be also denoted by $\gamma$; 2) the involution family $\sigma_P$ extends to a  singular holomorphic dual   billiard structure 
 on the algebraic curve $\gamma$ with the same rational integral $R$. 
 \end{proposition}
 \begin{proposition} \label{proalg2} Let  $\gamma\subset\rr^2$  be a regular non-linear $C^2$-smooth germ of curve 
 equipped with a dual   billiard structure having a rational integral $R$. Then 1) $R|_\gamma\equiv const$,  and thus, the complex Zariski closure 
 of the curve $\gamma$ is an 
 algebraic curve  in $\cp^2$; 2) the dual billiard structure   extends to 
 a singular holomorphic dual    billiard structure on each its non-linear 
irreducible component, with the same integral $R$.
\end{proposition}
 
 Parts 1), 2)  of these propositions will be proved in Subsections 2.1, 2.2. 
 
  Recall that each (may be singular) germ of  analytic curve in $\cp^2$ is a finite union of its irreducible components, which are 
  locally bijectively holomorphically parametrized germs 
 called {\it local branches.} 
 
 \begin{definition} \cite[definition 3.3]{gl2}
 Let $b$ be an irreducible (i.e., parametrized) non-linear 
 germ of analytic curve at a point $O\in\cp^2$. 
An affine chart $(z,w)$ centered at $O$ such that the $z$-axis is tangent 
 to $b$ at $O$ is called {\it adapted} to $b$.  In an adapted chart the germ $b$ 
 can be holomorphically bijectively parametrized by a complex parameter 
 $t$ from a disk centered at 0 as follows:
 $$t\mapsto(t^{q_b},c_bt^{p_b}(1+o(1))), \text{ as } t\to0; \ \  q_b,p_b\in\nn, \ 1\leq q_b<p_b, \ c_b\neq0,$$
 $$q_b=1, \text{ if and only if } b \text{ is a regular germ.}$$
 The {\it projective Puiseux exponent} \cite[p. 250, definition 2.9]{alg} of the germ $b$ 
 is the number 
 $$r=r_b:=\frac{p_b}{q_b}.$$
 The germ $b$ is called {\it quadratic,} if $r=2$ \cite[definition 3.5]{gs}. When $b$ is 
 a germ of line, it is parametrized by $t\mapsto(t,0)$: then we set $q_b=1$, $p_b=r_b=\infty$.
 \end{definition}
 
 The main part of the proof of Theorem \ref{tcompl} 
  is the proof of the following theorem on possible types  
 of singularities and local branches of the curve $\gamma$. 
 
 \begin{theorem} \label{typesing} Let an irreducible complex algebraic curve 
 $\gamma\subset\cp^2$ carry 
 a structure of rationally integrable singular holomorphic dual   billiard. Then the following 
 statements hold: 
 
 (i) the curve $\gamma$ has no inflection points, and at each its singular point  
 (if any) all its local branches are quadratic;
 
 (ii) there exists at most  unique singular  point of the curve $\gamma$ where 
 there exists at least one singular local branch.
 \end{theorem}
 
 \begin{theorem} \label{type-conic} Every complex irreducible projective planar 
 algebraic curve satisfying the above statements (i) and (ii) is a conic.
 \end{theorem}
 
 The proof of Theorem \ref{type-conic} will be given in Section 6. It is based on E.Shustin's generalized Plucker formula \cite{sh}, 
 dealing with intersection of an irreducible algebraic curve with its Hessian curve. It gives formula for the contributions of singular and inflection points to their intersection index.

 Theorems \ref{typesing}, \ref{type-conic} together with 
 Proposition \ref{proalg}  immediately imply that every germ of 
  holomorphic curve $\gamma$ 
 carrying a rationally integrable complex dual billiard  structure is a germ of a conic. 
Afterwards in Section 7 we classify the rationally integrable dual billiard structures on a conic. This will finish the 
proof of  Theorem \ref{tcompl}. Then in Section 8 we classify the real forms of the complex dual billiards 
from Theorem \ref{tcompl} and finish the proof of Theorems \ref{tgerm}, \ref{tclosed}.

 The  proof of 
 Theorem \ref{typesing} is based on studying the Hessian of 
appropriately normalized rational integral: the Hessian introduced by  S.Tabachnikov, who  used it to study polynomially integrable outer billiards \cite{tab08}. This idea was further elaborated and 
used by M.Bialy and A.Mironov in a series 
of papers on Bolotin's Polynomial Birkhoff Conjecture and its analogues for magnetic billiards \cite{bm, bm2, bm4}. 
It was also used in the previous paper by the author and E.Shustin on polynomially integrable outer billiards 
\cite{gs} and in the author's recent paper  on S.V.Bolotin's Polynomial Birkhoff Conjecture \cite{gl2}. 
 The rational integral  $R$ being constant along the curve $\gamma$ (Proposition \ref{proalg}), we normalize it 
 to vanish identically on $\gamma$. Let $f$ be the irreducible polynomial vanishing on $\gamma$ in an 
 affine chart $\cc^2_{x_1,x_2}\subset\cp^2$. Then 
 $$R=f^kg_1, \ \ g_1|_\gamma\not\equiv0.$$
 Replacing $R$ by its $k$-th root $G:=fg$, $g:=g_1^{\frac1k}$,  we consider its Hessian 
 $$H(G):= \frac{\dd G}{\d x_1^2}\left(\frac{\partial G}{\partial x_2}\right)^2-2\frac{\dd G}{\d x_1\d x_2}\frac{\partial G}{\d x_2}\frac{\partial G}{\d x_1}+
\frac{\dd G}{\d x_2^2}\left(\frac{\partial G}{\d x_1}\right)^2.$$
Its key property is that {\it $H(G)|_\gamma\neq0$ outside singular and inflection points of the curve $\gamma$ and zeros (poles) of the 
function $g_1|_\gamma$.}  

{\bf Plan of proof of Theorem \ref{typesing}.} 

Step 1 (Subsection 2.2): differential equation on $H(G)$. 
Given a point $P_0\in\gamma$, consider an affine chart $(z,w)$ in which  
the tangent line $L_{P_0}$ is not parallel to the $w$-axis.  
Then for every $P\in\gamma$ close enough to $P_0$ the line $L_P$ is parametrized by affine coordinate $z$. 
The involution $\sigma_P:L_P\to L_P$ is conjugated to the standard involution 
$\oc_{\theta}\to\oc_{\theta}$, $\theta\mapsto-\theta$, via a  mapping $\mcf_P:\theta\mapsto\mcf_P(\theta)$ 
that sends $\theta$ to the point 
in the tangent line $L_P$ with $z$-coordinate $z(P)+\frac{\theta}{1+\psi(P)\theta}$; $\psi(P)\in\cc$ is uniquely 
determined by $\sigma_P$.  Invariance of the  function $R|_{L_P}$ under the involution $\sigma_P$ 
is equivalent to statement that  the function $R\circ\mcf_P(\theta)$ is even. Writing the condition 
that its cubic Taylor coefficient vanishes (analogously to \cite{tab08, bm, bm2}), we get the differential equation 
\begin{equation}\frac{d H(G)|_{\gamma}}{dz}(P)=6\psi(P)H(G).\label{6pgh}\end{equation}

We prove equation (\ref{6pgh}) in a more general situation, for an irreducible germ of analytic curve $b$ at a point $B$ 
 equipped with a family of projective involutions  
 $\sigma_P:L_P\to L_P$, $P\in b\setminus\{ B\}$, admitting a germ $R$ of meromorphic (not necessarily rational) integral. 
 Here $f$ is a local defining function of the germ $b$, and $G$, $H(G)$ are  defined as above. 
 Relation between meromorphic and rational integrability will be explained in Subsection 2.4.

Step 2 (Subsection 2.3): formula relating asymptotics of $H(G)$ and $\sigma_P$.  We fix a potential singular (inflection) point $B\in\gamma$, 
  a local branch $b$ of the curve $\gamma$ at $B$ (whose quadraticity we have to prove) and affine coordinates $(z,w)$ centered at $B$ and 
  adapted to $b$. The function $H(G)$ being a linear combination of 
  products of rational powers of holomorphic functions at $B$, we get 
  that $H(G)|_b=\alpha z^{d}(1+o(1))$, as $z\to0$, for some $d\in\mathbb Q$ and $\alpha\neq0$. 
  This together with equation (\ref{6pgh}) implies that  
  $$\psi|_{b\setminus\{ B\}}=\frac1z\left(\frac d6+o(1)\right)=-\frac1z\left(\frac\rho2+o(1)\right), \text{ as } z\to0; 
   \ \rho=-\frac d3\in\mathbb Q.$$
    We then say that the involution  family $\sigma_P$ is {\it meromorphic} 
with pole  at $B$ of order at most one with {\it residue $\rho$}. This means exactly that  
$$\text{in the coordinate }\ \ \  \zeta:=\frac z{z(P)} \ \ \ \text{ on } \ L_P$$
$$ \text{ the involution } \ \sigma_P 
\text{ converges to } \ \er(\zeta):=\frac{(\rho-1)\zeta-(\rho-2)}{\rho z-(\rho-1)}, \ \text{ as } P\to B.$$
Thus, the above number $\rho$ in the limit and $H(G)$ are related by the formulas
\begin{equation} H(G)|_b= \alpha z^d(1+o(1)), \ \  \alpha\neq0, \ \ \rho=-\frac d3.\label{hgrho}\end{equation}

Step 3 (Section 3) Consider the projective Puiseux exponent $r$ of the local branch $b$, and let us represent it as an irreducible fraction: 
\begin{equation}r=\frac pq,  \   p,q\in\nn,   \ (p,q)=1:  \ \ p_b=ps_b, \ q_b=qs_b, \ s_b=G.C.D(p_b,q_b).\label{propuis}\end{equation}
For given $p$, $q$ as above and $\rho\in\cc$ we introduce the {\it $\pqr$-billiard:} the curve  
$$\gamma_{p,q}:=\{ w^q=z^p\}\subset\cc^2$$ 
equipped with the dual billiard structure given by the  family of involutions 
$$\sigma_P:L_P\to L_P, \ \ P\in\gamma_{p,q}; \ \ \sigma_P:\zeta\to\er(\zeta) \ \ \text{ in the coordinate } \ \zeta  \ \text{ for all } \ P.$$
We show (Theorem \ref{quasi-pqr}) that 
{\it if a germ of family of involutions $\sigma_P:L_P\to L_P$ 
defined on a punctured irreducible germ of holomorphic curve $b$ with Puiseux exponent $r=\frac pq$ 
admits a meromorphic first integral, then the corresponding $\pqr$-billiard (with $\rho$  given by (\ref{hgrho})) 
admits a $(p,q)$-quasihomogeneous 
rational first integral.} To do this, we consider the lower $(p,q)$-quasihomogeneous parts in the numerator 
and the denominator of the meromorphic integral.  We show that their ratio is an integral of the $\pqr$-billiard. If it is non-constant, 
then we get a non-trivial integral. The opposite case, when the latter lower 
$(p,q)$-quasihomogeneous parts are the same up to constant factor, will be reduced to the previous one by 
replacing the denominator by appropriate linear combination of the numerator and denominator.

Step 4 (Section 4). Classification of quasihomogeneously rationally integrable $\pqr$-billiards (Theorem \ref{classpqr}). 
Our first goal is to show that the underlying 
curve $\gamma_{p,q}$ is a conic: $p=2$, $q=1$. In Subsection 4.1 we treat the case of polynomial integral. To treat the 
case of non-polynomial rational integral, we first show (in Subsection 4.2) that one can normalize it to be a so-called 
$\er$-primitive quasihomogeneous rational function $R=\frac{\mcp_{1}^{m_1}}{\mcp_{2}^{m_2}}$ vanishing on $\gpq$. In particular, 
this means that each $\mcp_{j}$ is a product of  prime factors $w^q-c_jz^p$, $c_j\neq0$  
(and may be $z$, $w$)  in power 1, including the factor $w^q-z^p$. Then in Subsection 4.3 we prove two formulas (\ref{rhooo}), 
(\ref{for2}) expressing 
$\rho$ via the powers $m_1$, $m_2$, the  number of factors $w^q-c_jz^p$ and the powers of $z$, $w$ in  $\mcp_j$. 
The first formula (\ref{rhooo}) will be deduced from (\ref{hgrho}). The second formula (\ref{for2}) is obtained as follows. Restricting the 
polynomial $\mcp_{1}$ from the 
numerator  to the line $L_P$, $P=(1,1)$, and dividing it by appropriate power $(z-\frac{\rho-2}\rho)^d$
yields a $\er$-invariant rational function in the coordinate $z$ with numerator divisible by $(z-1)^2$. 
Existence of such a power $d$ will follow from  the fact that $\frac{\rho-2}\rho$ is 
a fixed point of the involution $\er$.  Afterwards we replace the numerator $\mcp_{1}$ by the difference 
$\mcp_{1}-\la(z-\frac{\rho-2}\rho)^d$ with small $\la$; we get  a family of $\er$-invariant rational functions 
depending on the parameter $\la$, which has a pair of roots $\zeta_{\pm}(\la)$ converging to 1, as $\la\to0$. 
Comparing the asymptotics of the roots $\zeta_{\pm}(\la)$ and taking into account that they should be permuted 
by the involution $\er$, we get   formula (\ref{for2}). 
The {\it main miracle} in the proof of Theorem \ref{classpqr} (Subsection 4.4) 
is that {\it combining First and Second Formulas (\ref{rhooo}), (\ref{for2}) yields that $p=2$, $q=1$ and 
the curve $\gamma_{p,q}$ is the conic $\{ w=z^2\}$,}  and it also yields the constraints on 
$\rho$ given by Theorem \ref{classpqr}: the necessary condition for quasihomogeneous integrability. 
Then we prove its sufficience by constructing  integrals (Subsection 4.5).

Steps 3 and 4 together imply Statement (i) of Theorem \ref{typesing}: each local branch of the curve $\gamma$  is quadratic. 
They also yield a list of a priori possible values of the residue $\rho$. 

Step 5. Proof of statement (ii) of Theorem \ref{typesing}: uniqueness of singular point of the curve $\gamma$ with a singular local branch. 
To do this,  we prove Theorem \ref{lalt} stating that if a quadratic local branch $b$ at a point $O$ 
 is singular, then the  integral $R$ is constant along its projective tangent line $L_O$, and 
 the punctured line $L_O\setminus\{ O\}$ is a regular leaf of the foliation $R=const$. This implies that if $\gamma$ had two 
 distinct points with singular local branches, then the corresponding tangent lines would intersect, and we get a 
 contradiction with regularity of foliation at the intersection point. The proof of Theorem \ref{lalt} given in Subsection 5.2 is partly based on  Theorem \ref{tufold} (Subsection 5.1), which implies that if there exists a singular 
 quadratic local branch, then its self-contact order  is expressed via the corresponding residue $\rho$ by an explicit formula (\ref{rhode}) implying 
 that $\rho>r=2$. Once having inequality $\rho>r$, we deduce the statements of Theorem \ref{lalt} 
 analogously to  \cite[subsection 4.6, proof of theorem 4.24]{gl2}. 
 Step 5 finishes the proof of Theorem \ref{typesing}.
 
 In Section 6 we prove Theorem \ref{type-conic}. 
 Theorems \ref{typesing} and \ref{type-conic} together imply that $\gamma$ is a conic. Afterwards in Section 7 
 we classify singular holomorphic rationally integrable dual billiards on a complex conic. The list of a priori possible residues $\rho$ 
 at singularities is given by Theorem \ref{classpqr}, Step 4. In Subsection 7.1 we show that the sum of residues should be equal to four 
 (a version of residue formula for singular holomorphic dual billiard structures). 
 Afterwards we consider all the a priori possible residue configurations given by 
 these constraints and show that all of them are indeed realized by rationally integrable dual   billiards.  
This will finish the proof of Theorem \ref{tcompl}. Then Theorems \ref{tgerm}, \ref{tclosed} are proved in Section 8 
 by describing different real forms of thus classified complex integrable dual   billiards. Theorems \ref{tclasspr} and \ref{tgermproj} 
 classifying integrable projective billiards  (which are dual to the latter real forms) will be proved in Section 9.  
 Theorems \ref{talg1} and \ref{talg2} on billiards with complex caustics will be proved in Section 10. 

\subsection{Historical remarks}
In 1973 V.Lazutkin \cite{laz} proved that every strictly convex bounded planar billiard with sufficiently smooth boundary has an infinite number (continuum) of closed caustics.  
The Birkhoff Conjecture was studied by many mathematicians. In 1950 H.Poritsky \cite{poritsky} (and later 
E.Amiran \cite{amiran} in 1988)
 proved it under the additional assumption that 
 the billiard in each closed caustic near the boundary has the same closed caustics, as the initial billiard. 
 In 1993 M.Bialy \cite{bialy} proved that if the phase cylinder of the billiard in a domain $\Omega$ is 
 foliated  by non-contractible continuous closed curves which are invariant under the billiard map,  then the boundary 
 $\partial\Omega$ is a circle. (Another proof of the same result was later obtained in \cite{wojt}.) 
   In 2012 Bialy proved a similar result  for billiards on the constant curvature surfaces \cite{bialy1} and also for magnetic billiards 
 \cite{bialy2}. In 1995 A.Delshams and R.Ramirez-Ros suggested an 
 approach to prove splitting of separatrices for generic perturbation of ellipse \cite{dr}.
 D.V.Treschev \cite {treshchev} made a numerical experience indicating that 
 there should exist analytic {\it locally integrable} billiards, with the billiard reflection map 
 having a  two-periodic point where the  germ of its second iterate is analytically conjugated to a disk rotation. See 
 also  \cite{tres2} for more detail and \cite{tres3} for a multi-dimensional version. 
 A similar effect for a ball rolling on a vertical 
 cylinder under the gravitation force was discovered in \cite{tad}. 
  Recently V.Kaloshin and A.Sorrentino have proved a {\it local version} of the Birkhoff Conjecture \cite{kalsor}: 
 {\it an integrable deformation of 
 an ellipse is an ellipse}. Very recently M.Bialy and A.E.Mironov \cite{bm6} 
 proved the Birkhoff Conjecture for centrally-symmetric 
 billiards having a family of closed caustics that extends up to a caustic tangent to four-periodic orbits. 
 For a dynamical entropic version of the Birkhoff Conjecture and related results see \cite{marco}. 
 For a survey on the Birkhoff Conjecture and  results see \cite{kalsor, KS18, bm6} and 
 references therein. 
 
 Recently it was shown by the author \cite{glcaust} that every strictly convex $C^{\infty}$-smooth non-closed planar 
 curve has an adjacent domain from the convex side that admits an infinite number (continuum) of 
 distinct $C^\infty$-smooth foliations by non-closed caustics (with the boundary being a leaf). 
 
 A.P.Veselov proved a series of complete integrability results for billiards bounded by confocal quadrics 
 in space forms of any dimension and described billiard orbits there in terms of 
 a shift of the Jacobi variety corresponding to an appropriate hyperelliptic curve \cite{veselov, veselov2}. 
Dynamics in (not necessarily convex) billiards of this type was also studied in  \cite{drag, dr2, dr3, dr4, dr5}. 

 The Polynomial Birkhoff Conjecture together with its generalization to surfaces of constant curvature was  
 stated by S.V.Bolotin and partially studied by himself, see  \cite{bolotin}, \cite[section 4]{bolotin2}, 
 and  by M.Bialy and A.E.Mironov  \cite{bm3}. Its complete  solution 
 is a joint result of M.Bialy, A.E.Mironov and the author given in  the series of papers 
  \cite{bm, bm2, gl, gl2}.

For a survey on the Polynomial Birkhoff  Conjecture, its version for magnetic billiards and related results 
see the above-mentioned papers \cite{bm, bm2} 
by M.Bialy and A.E.Mironov, \cite{bm4, bm5} and references therein. 

The analogues of the  Birkhoff Conjecture for outer and dual billiards was stated  by S.Tabachnikov 
\cite{tab08} in 2008. Its polynomial version for outer billiards  was stated by Tabachnikov and proved by himself under genericity assumptions in the same paper \cite{tab08}, and  solved completely in the joint work of the author of the present paper with E.I.Shustin \cite{gs}. 

In 1995 M.Berger have shown that in Euclidean space $\rr^n$ with $n\geq3$ the only hypersurfaces admitting caustics are quadrics \cite{berger}. 
In 2020 this result was extended to space forms of constant curvature of dimension greater than two by the author of the 
present paper \cite{commute}. 

In 1997 S.Tabachnikov \cite{tabpr} introduced projective billiards and proved a criterium and a necessary condition 
for a planar projective billiard to preserve an area form. He had shown that if a projective billiard on  circle 
preserves an area form that is smooth up to the boundary of the phase cylinder, then the billiard is integrable. 

A series of results on projective billiards with open sets of $n$-periodic orbits (classification for $n=3$ 
and new examples for higher $n$) were obtained by C.Fierobe \cite{fierobe-proj, fierobe-proj-refl, fierobe-th}.

 \section{Preliminaries}
 
 \subsection{Algebraicity of underlying curve. Proof of 
 Propositions \ref{proalg} and \ref{proalg2}, parts 1)}
 Let $\gamma\subset\cp^2$ be a regular germ of holomorphic curve. 
For every $P\in\gamma$ the restriction $R|_{L_P}$ is invariant under 
 an involution $\sigma_P$ fixing $P$. In appropriate affine coordinate $\theta$ on $L_P$ 
 centered at $P$ the latter involution takes the form $\theta\mapsto-\theta$. 
 Therefore, the restriction $R|_{L_P}$ has zero derivative at $P$, since 
an even function has zero derivative at the origin. Finally, the rational function $R$ 
has zero derivative along any vector tangent to $\gamma$. Hence, it is constant on 
$\gamma$, and  the germ of curve $\gamma$ is algebraic. This proof 
remains valid in the case, when $\gamma$ is a real germ. 
Parts 1) of Propositions \ref{proalg} and \ref{proalg2} are proved.

For completeness of presentation (to state some results in full generality), 
we will deal  with the following notion of meromorphically integrable dual   billiard structure and 
 meromorphic version of Proposition \ref{proalg}.

\begin{definition} Let $b$ be a non-linear (may be singular) irreducible germ of analytic curve in $\cc^2$ at a point $B$, 
and let $\sigma_P:L_P\to L_P$ be a family of projective involutions parametrized by $P\in b\setminus\{ B\}$. 
The family $\sigma_P$ is  called a  
{\it  meromorphically integrable (singular) dual   billiard structure,} if 
there exists a germ of  meromorphic 
function $R$ at $B$ (defined on a neighborhood of the point $B$ in $\cc^2$), $R\not\equiv const$, such that 
the restrictions $R|_{L_P}$ are $\sigma_P$-invariant: there exists a neighborhood 
$U=U(B)\subset\cc^2$ such that for every $P\in b\cap U$, $P\neq B$, and every 
$x,y\in L_P\cap U$ such that $\sigma_P(x)=y$ one has $R(x)=R(y)$. 
\end{definition} 

\begin{proposition} \label{proalgm} In the conditions of the above definition 1) $R|_b\equiv const$; 2) the family $\sigma_P$ is holomorphic in $P\in b\setminus\{ B\}$ close enough to $B$.
\end{proposition}
The proof of the first part of Proposition \ref{proalgm} repeats that of Proposition \ref{proalg}, part 1). Its second part will be proved in the next 
subsection.

Later on, in Subsection 2.4 we will show that in many cases meromorphic integrability implies rational integrability.

\subsection{The Hessian of  integral and its differential equation. Singular holomorphic extension of dual billiard structure}
Let $b$ be an irreducible germ of holomorphic curve in $\cc^2_{x_1,x_2}$ at a point $B$. Let $b\setminus\{ B\}$ be equipped with a 
germ of   dual billiard structure having a non-constant meromorphic integral $R$, see the above definition. 
Recall that $R|_b\equiv const$, by Proposition \ref{proalgm}. 
Without loss of generality we consider that 
$$R|_{b}\equiv0,$$
adding a constant to $R$ (if $R|_{b}\not\equiv\infty$), or replacing $R$ by 
$R^{-1}$ (if $R|_{b}\equiv\infty$). 

Let $f$ be an irreducible germ of holomorphic function defining  $b$:
$$b=\{ f=0\}.$$
 One has  
 $$R=g_1f^k, \ \ g_1 \ \ \text{ is meromorphic, } \ \ g_1|_{b}\not\equiv0, \ k\in\nn.$$
 From now on we will work  with  the $k$-th root 
  \begin{equation}G=R^{\frac1k}=gf, \ g=g_1^{\frac1k}.
\label{ggf}\end{equation}
For every $P\in b\setminus\{ B\}$ close enough to $B$   
 each holomorphic branch of the function $G$ on $L_P$  is $\sigma_P$-invariant, since any two its holomorphic 
 branches are obtained one from the other by multiplication by a root of unity. 

Recall that the {\it skew gradient} of the function $G$  is the vector field 
$$\nabla_{skew} G:=(\frac{\partial G}{\partial x_2},-\frac{\partial G}{\partial x_1}),$$
which is tangent to its level curves. 

The involution $\sigma_P$  is conjugated to the standard involution $\cc_{\tau}\to\cc_{\tau}$,  
$\tau\mapsto-\tau$, via a transformation 
\begin{equation}\Phi_P:\tau\mapsto P+\frac{\tau}{1+\phi(P)\tau}\nabla_{skew}G(P), \ \Phi_P(0)=P.\label{phip}\end{equation}
The conjugacy is unique up to its pre-composition with a multiplication by constant $\tau\to\la\tau$; 
we can normalize it to be of  type (\ref{phip}) in unique way. The $\sigma_P$-invariance of the function $G$ is equivalent to the statement that 
\begin{equation} \text{the function } \xi(\tau):=G(P+ \frac{\tau}{1+\phi(P)\tau}\nabla_{skew}G(P)) \text{ is even,}
\label{even}\end{equation}
which holds if and only if the function $\xi(\tau)$ has zero Taylor coefficients at odd powers. The first 
coefficient vanishes for trivial reason, being derivative of a function $G$ along a vector tangent to 
its zero level curve.

Recall that the {\it Hessian} of the function $G$ is the function 
\begin{equation}
H(G):=\frac{\dd G}{\d x_1^2}\left(\frac{\partial G}{\partial x_2}\right)^2-2\frac{\dd G}{\d x_1\d x_2}\frac{\partial G}{\d x_2}\frac{\partial G}{\d x_1}+
\frac{\dd G}{\d x_2^2}\left(\frac{\partial G}{\d x_1}\right)^2.\label{hessdef}\end{equation}
It coincides with the value of its Hessian quadratic form on its skew gradient 
and also with the second derivative $\xi''(0)$, see \cite{tab08, bm, bm2}. 

\begin{remark}
The  Hessian  function $H(G)$ was introduced by S.Tabachnikov   \cite{tab08} 
and used in \cite{tab08, gs, bm, bm2, gl2} to classify polynomially 
integrable  Birkhoff and outer planar billiards; see results  mentioned in Subsection 1.5. 
\end{remark}

\begin{theorem} \label{tdif1} For a given $P\in b\setminus\{ B\}$ 
the cubic Taylor coefficient of the function $\xi$ from (\ref{even}) at $0$ vanishes, if and only if 
\begin{equation}\frac{ dH(G)}{d\nabla_{skew}G}(P)=6\phi(P) H(G)(P).\label{difeq1}\end{equation}
\end{theorem}
\begin{remark} Theorems analogous to Theorem \ref{tdif1} were stated and proved  in \cite{tab08, bm, bm2}, and the proofs from loc. cit. remain valid in our case. The proof of Theorem \ref{tdif1} given below follows similar arguments.
\end{remark}
\begin{proof} {\bf of Theorem \ref{tdif1}.} The third derivative $\xi'''(\tau)$ is equal to the third derivative in $\tau$ of the function 
$$G(P+\omega(\tau)W), \ W:=\nabla_{skew}G(P), \ \omega(\tau):= \frac{\tau}{1+\phi(P)\tau}=\tau-\phi(P)\tau^2+O(\tau^3),$$
 as $\tau\to0$. The first derivative equals 
 $$\xi'(\tau)=\frac{dG}{dW}(P+\omega(\tau)W)(1-2\phi(P)\tau+O(\tau^2)).$$
 We extend the vector $W$ to a constant vector function (field) by translations. Here and in what follows 
 the derivative of a function along $W$ means its derivative along the latter constant vector field. 
 For simplicity, in what follows we omit the argument $P+\omega(\tau)W$ at the derivatives. One has  
 $$\xi''(\tau)=(-2\phi(P)+O(\tau))\frac{dG}{dW}+(1-2\phi(P)\tau+O(\tau^2))^2\frac{d^2G}{dW^2},$$
 $$\xi'''(\tau)=O(1)\frac{dG}{dW}-(2\phi(P)+O(\tau))(1-2\phi(P)\tau+O(\tau^2))\frac{d^2G}{dW^2}$$
 $$-(4\phi(P)+O(\tau))\frac{d^2G}{dW^2}+(1-2\phi(P)\tau+O(\tau^2))^3\frac{d^3G}{dW^3}.$$
The value of the third derivative at zero is thus equal to 
 \begin{equation}\xi'''(0)= \frac{d^3G}{dW^3}(P)-6\phi(P)\frac{d^2G}{dW^2}(P),\label{'''}\end{equation}
 since $\frac{dG}{dW}(P)=0$.  One has  $\frac{d^2G}{dW^2}(P)=H(G)(P)$,
 $$\frac{d^3G}{dW^3}(P)=\frac{dH(G)}{d\nabla_{skew}G}(P),$$
 by \cite[lemma 2, (i)]{tab08} and since $\frac{d^3G}{dW^3}(P)$ is the value at $P$ of the expression
 $$G_{x_1x_1x_1}G_{x_2}^3-3G_{x_1x_1x_2}G_{x_2}^2G_{x_1}+3G_{x_1x_2x_2}G_{x_2}G_{x_1}^2
-G_{x_2x_2x_2}G_{x_1}^3.$$
  This together with (\ref{'''}) implies the statement 
 of the theorem.
 \end{proof} 
 
 Let us consider  
affine coordinates $(z,w)$  such that the tangent line $T_Bb$ is not parallel to the $w$-axis. 
For every $P\in b$ 
close to $B$ the restriction to 
$L_P$ of the coordinate $z$ is an affine coordinate on the projective line $L_P$. 

 We will deal with the following normalizations of mapping (\ref{phip}) and equation (\ref{difeq1}) 
 with respect to the coordinate $z$. Set 
 \begin{equation} V(P)=(1,\beta(P)):=\text{ the vector in } T_Pb \text{ with unit } z- \text{component.} \label{tang1}\end{equation}
The vectors $V(P)$ form a holomorphic vector field $V$ on $b\setminus\{ B\}$. One has 
\begin{equation}\nabla_{skew}G=hV, \ h:b\setminus\{ B\}\to\cc \text{ is a non-zero  function};\label{nav}\end{equation}
a priori the  function $h$ is multivalued holomorphic on $b\setminus \{ B\}$ 
(with a priori possible branching at $B$).  Set 
$$\theta:=h(P)\tau, \ \psi(P):=\phi(P)h^{-1}(P).$$
Let $\Phi_P(\tau)$ be the mapping from (\ref{phip}). Set 
$$\mcf_P(\theta):=\Phi_P(h^{-1}(P)\theta)=P+\frac{h^{-1}(P)\theta}{1+\phi(P)h^{-1}(P)\theta}h(P)V(P)$$
\begin{equation}=
P+\frac{\theta}{1+\psi(P)\theta}V(P).\label{psit}\end{equation}
\begin{proposition} The map $\mcf_P(\theta)$ conjugates the involution $\sigma_P$ 
with the standard symmetry $\theta\mapsto-\theta$, and its differential 
at $0$ sends the unit vector $\frac{\partial}{\partial\theta}$ to $V(P)$. One has
\begin{equation}\frac{dH(G)}{dV}(P)=6\psi(P)H(G)(P).\label{difeq2}\end{equation}
\end{proposition}
\begin{proof} The statements on conjugacy and differential follow by construction. Equation (\ref{difeq2}) is obtained 
from (\ref{difeq1}) by multiplication  by $h^{-1}(P)$.
\end{proof}

We use the following formula for  Hessian of product, see \cite[theorem 6.1]{bm}, \cite[formulas (16) and (32)]{bm2}:
\begin{equation} H(fg)=g^3H(f) \text{ on the set } \{ f=0\}.\label{hess3}\end{equation}

\begin{proof} {\bf of parts 2) of Propositions \ref{proalg}, \ref{proalg2}, \ref{proalgm}.} Let us prove part 2) of Proposition \ref{proalg}: 
for the other propositions the proof is analogous.  Let $\gamma$ denote the irreducible algebraic curve containing the initial germ $\gamma$. 
Let $\phi(P)$ denote the function (\ref{phip}) defined by the involutions $\sigma_P$. 
Equation (\ref{difeq1}) extends the function $\phi(P)$ holomorphically along paths in  $\gamma$ avoiding 
a finite collection of points where  some branch of the multivalued function $H(G)$ either vanishes, or is not holomorphic, or 
 its derivative in the left-hand side in (\ref{difeq1}) is not holomorphic. 
 It defines a holomorphic extension of the involution family $\sigma_P$. The relation 
$R\circ\sigma_P|_{L_P}=R$ remains valid for the extended dual billiard structure, by uniqueness of analytic extension. Let us show that this 
yields a well-defined singular holomorphic dual billiard structure on $\gamma$. Suppose the contrary: 
thus extended family 
$\sigma_P$ is multivalued, i.e.,  its extensions along two different paths arriving to one and the same point $A$ are 
two different involutions $\sigma_A$ and $\wt\sigma_A$. Then  
their composition $\sigma_A\circ\wt\sigma_A:L_A\to L_A$ is a parabolic transformation 
 with unique fixed point $A$, leaving invariant the restriction $R|_{L_A}$. 
Its orbits (except for the fixed point $A$) being infinite and accumulating to $A$, one has $R|_{L_A}\equiv const$. 
The  involutions $\sigma_A$, $\wt\sigma_A$ are well-defined and satisfy the above statements on an open subset  
of points $A$ in  $\gamma$, by local analyticity. Therefore, $R|_{L_A}\equiv const$ 
for an open subset of points $A\in\gamma$, which is impossible. The contradiction thus obtained implies that the extended 
 dual billiard structure is singular holomorphic.
 \end{proof}
\subsection{Asymptotics of degenerating involutions}

Here we deal with an irreducible germ $b$  at a point $B$ of analytic curve in 
$\cc^2$ equipped with a meromorphically integrable singular holomorphic dual   
billiard structure. We study asymptotics of involutions $\sigma_P$, as $P\to B$. 

For every $\rho\in\cc$ we denote by $\eta_\rho\in\operatorname{PSL}_2(\cc)$ the projective involution
\begin{equation}\eta_{\rho}:\oc_\zeta\to\oc_\zeta, \ \ \ \eta_\rho(\zeta):=\frac{(\rho-1)\zeta-(\rho-2)}{\rho \zeta-(\rho-1)}.\label{defeta}
\end{equation}
\begin{remark} Every projective involution $\oc\to\oc$  fixing  $1$ coincides with $\er$ for some $\rho\in\cc$ and vice versa.
\end{remark}
\begin{definition} Let $b$ be an irreducible germ    of holomorphic curve at a point  
$B\in \cc^2$. Let $(z,w)$ be an affine chart adapted to $b$. A germ of singular holomorphic 
dual   billiard structure on $b$ given by  a holomorphic family of involutions 
$\sigma_P:L_P\to L_P$, $P\in b\setminus\{ B\}$   is  said to be {\it meromorphic with pole of order at most one} at $B$, 
 if the involutions $\sigma_P$ written in the coordinate
$$\zeta:=\frac z{z(P)}$$
on $L_P$  converge in $\operatorname{PSL}_2(\cc)$ to some involution $\oc_\zeta\to\oc_\zeta$. 
Then the limit involution is equal to $\eta_\rho$ for some $\rho$, by the above remark. The latter number $\rho$ 
is called the {\it residue} of the billiard structure at $B$. In the case, when $\rho\neq0$, we say that $\sigma_P$ 
has {\it pole of order exacly one} at $B$.
\end{definition}

\begin{remark} For every meromorphic billiard structure of order at most one the above  residue is 
independent on choice of adapted chart. 
\end{remark}
\begin{example} 1) In the case,  when $\sigma_P$ limits to a 
well-defined projective involution $L_B\to L_B$, as $P\to B$ (e.g., if $\sigma_P$ extends holomorphically to $P=B$), we say that 
the dual billiard structure is {\it regular} at $B$.
In this case  the involutions $\sigma_P$ written in the above coordinate $\zeta$ converge to the symmetry 
$\eta_0:\zeta\mapsto2-\zeta$, and the billiard structure has residue $\rho=0$ at $B$. 

2) Consider now the case, when  there exists a  conic $\Gamma$  passing through $B$ 
such that each involution $\sigma_P$ permutes the points of 
 intersection $L_P\cap\Gamma$. Let $\Gamma$ be transversal to $b$ at $B$. Then $\sigma_P$ converges to the 
  unique involution $\eta_1:\zeta\mapsto\frac1\zeta$ fixing 1 and 
 permuting the origin and the infinity: $\rho=1$.
 \end{example}
One of the key statements used in the proof of main results is the following proposition. 
\begin{proposition} \label{prhod} Let $b\subset\cp^2$ be an irreducible germ of holomorphic curve at 
a point $B$ equipped with a singular holomorphic dual   billiard structure admitting a meromorphic integral $R$. Let $f$ be an irreducible germ of holomorphic function defining $b$, i.e., $b=\{ f=0\}$, 
and let $k$ and  $G=R^{\frac1k}$ be the same, 
as in (\ref{ggf}).  Let $(z,w)$ be affine coordinates centered at $B$ and adapted to $b$. 
Let  us equip the germ $b$ with the coordinate $z$. Consider the restriction $H(G)|_b$ 
as a multivalued function of $z$. Let  $d\in\mathbb Q$ be the minimal number such that 
the monomial $z^d$ is contained in its Laurent Puiseux series:
$$H(G)|_b=\alpha z^d(1+o(1)), \ \ \alpha\neq0.$$
Then the involution family $\sigma_P:L_P\to L_P$ defining the dual   billiard  is meromorphic with pole $B$ of order at most one, and its  residue $\rho$ is equal to 
\begin{equation}\rho=-\frac d3.\label{rhod}\end{equation}
\end{proposition}
\begin{remark}  \label{ashessrem} The above asymptotic exponent $d$   is well-defined, 
since $H(G)$ is a finite sum of products of rational powers of holomorphic functions, see (\ref{hessdef}). It is independent on the affine chart containing $B$ 
chosen to define  $\nabla_{skew}G$ and $H(G)$, see the statement  after formula (\ref{hessdef}) above and the discussion in 
\cite[p. 1022, proof of proposition 3.6]{gl2}. Therefore, we can calculate the exponent $d$ writing $H(G)$ in the adapted 
coordinates $(z,w)$.
\end{remark}

\begin{proof} {\bf of Proposition \ref{prhod}.} Consider a  
line $L_P$ equipped with the coordinate $\zeta=\frac z{z(P)}$ and its parametrization 
 by the parameter $\theta$: 
 $$\zeta =1+\frac{\theta}{z(P)(1+\psi(P)\theta)}.$$
 Set $\rho=-\frac d3$, see (\ref{rhod}).  One has 
 $$\psi(P)=\frac 1{z(P)}\left(\frac d6+o(1)\right)=\frac1{z(P)}\left(-\frac{\rho}{2}+o(1)\right), \text{ as } P\to B,$$ 
by equation (\ref{difeq2}). Therefore, 
\begin{equation}\zeta=1+\frac{2\theta}{2z(P)(1+o(1))-\rho\theta}=\frac{2z(P)(1+o(1))-(\rho-2)\theta}{2z(P)(1+o(1))-\rho\theta}.\label{zepsi}\end{equation}
In the coordinate $\theta$ the involution $\sigma_P$ is standard: $\theta\mapsto-\theta$. Therefore, its matrix in the coordinate $\zeta$ treated as an element in 
$\operatorname{PSL}_2(\cc)$ is the conjugate of the matrix $\diag(1,-1)$ by the matrix 
of transformation (\ref{zepsi}). Up to a scalar factor, this is the matrix 
$$\left(\begin{matrix} 2-\rho & 2z(P)(1+o(1))\\ -\rho & 2z(P)(1+o(1))\end{matrix}\right)
\left(\begin{matrix} 1 & 0\\ 0 & -1\end{matrix}\right)\left(\begin{matrix} 2z(P)(1+o(1)) 
& -2z(P)(1+o(1))\\ \rho & 2-\rho\end{matrix}\right)$$
$$=
-4z(P)\left(\left(\begin{matrix}\rho-1 & -(\rho-2)\\ \rho & -(\rho-1)\end{matrix}\right)+o(1)\right).$$
Hence, $\sigma_P\to\eta_\rho$ in the coordinate $\zeta$. This proves the proposition.
\end{proof}

The  number $\rho$ is called "residue" due to the following proposition.
\begin{proposition} \label{propress}  Let $b$ be a regular germ at $B$ equipped with a meromorphic 
 dual   billiard structure with pole of order at most one with  residue $\rho$. 
Then in the coordinate 
$$u:=z-z(P)$$
the family of involutions $\sigma_P:L_P\to L_P$, $P\in b\setminus\{ B\}$, takes the form 
\begin{equation}\sigma_P:u\mapsto-\frac u{1+f(z(P))u}, \label{sigres}\end{equation}
$$f(z)=\frac{\rho}z+g(z), \ g(z) \text{ is a holomorphic function at } 0.$$
Conversely, an involution family  holomorphic in $P\in b\setminus\{ B\}$ and having form (\ref{sigres}) 
is meromorphic with pole of order at most one at $B$ with residue $\rho$. 
In particular,  $\sigma_P$ is regular at $B$, if and only if it has zero residue at $B$.
\end{proposition}
\begin{proof} The family $\sigma_P$ is meromorphic of order at most one at $B$ with residue $\rho$, if and only if in the coordinate 
$$\wt u:=\zeta-1$$ 
the involutions $\sigma_P$ take the form 
\begin{equation}\sigma_P:\wt u\mapsto-\frac{\wt u}{1+(\rho+o(1))\wt u}, \text{ as } P\to B,\label{sigresv}\end{equation}
 by definition and since  
$\er$ sends $\wt u$ to $-\frac{\wt u}{1+\rho\wt u}$. Rescaling $\wt u$ to $u=\wt uz(P)$ yields (\ref{sigres}) with 
$f(z)=\frac{\rho}z+g(z)$, $g(z)=o(\frac1z)$. Conversely, rescaling $u$ to $\wt u$ transforms (\ref{sigres}) 
to (\ref{sigresv}). The family of involution $\sigma_P$ depends holomorphically on 
$P\in b\setminus\{ B\}$, and hence, on $z=z(P)$, by regularity of the germ $b$. Therefore, if (\ref{sigresv}) 
holds, then the function $zf(z)$, and hence, $h(z):=zg(z)$ extends holomorphically to $0$.  One has $h(0)=0$, since 
$g(z)=o(\frac1z)$. Hence, $g(z)=\frac{h(z)}z$ is holomorphic at $0$. Statement (\ref{sigres}) is proved, and it 
 immediately implies the last statement of the proposition.
\end{proof}
\subsection{Meromorphic integrability versus rational}
Here we prove the following proposition.
\begin{proposition} \label{merrat} 
Let $b$ be a non-linear irreducible germ of holomorphic curve at $O\in\cc^2$ equipped with a meromorphically integrable 
singular dual   billiard structure with integral $R$. Let $\rho$ be its residue  at $O$ (see Proposition \ref{prhod}). If $\rho\neq0$, then 
 $R$ is rational, and $b$ lies in an algebraic curve.
\end{proposition}

\begin{proof} Let $(z,w)$ be coordinates adapted to $b$. Let  $U=U_z\times U_w$, 
$U_z=\{|z|<\var\}$, $U_w=\{|w|<\delta\}$, be 
a polydisk such that the meromorphic integral $R$ is well-defined on a  bigger polydisk containing its closure. For every $P\in b$ let 
$P_{\rho}$ denote the point in $L_P$ with the $\zeta$-coordinate $\theta_\rho:=\frac{\rho-1}{\rho}=\er(\infty)$; 
$\zeta=\frac z{z(P)}$. 
The involution $\sigma_P:L_P\to L_P$ sends the neighborhood of infinity $V_P(\var):=L_P\cap\{|z|>\frac\var2\}$ to a $o(z(P))$-neighborhood 
of the point  $P_\rho$ (thus, contained in $U$, if $P$ is close enough to $O$), 
 since $\sigma_P\to \er$ in the coordinate $\zeta$. The pullback of the integral $R$ 
under the map $\sigma_P|_{V_P(\var)}$ is a meromorphic function on $V_P(\var)$ whose restriction to the open subset $L_P\cap\{ \frac\var2<|z|<\var
\}\subset V_P(\var)$ coincides with $R$, by  $\sigma_P$-invariance. This extends $R$ to a meromorphic 
(and hence, rational) function on all of $L_P$ for every $P\in b\setminus\{ O\}$ close enough to $O$. 
The domains $V_P(\frac\var2)\subset L_P$ corresponding to $P$ close enough to $O$ foliate a neighborhood of 
the complement $L_O\setminus\{ |z|<\frac\var2\}$ in $\cp^2$. The function $R$ thus extended is meromorphic on 
the union of the latter neighborhood and the bidisk $U$, which covers a neighborhood of the line $L_O$ 
in $\cp^2$. 

\begin{proposition} \label{2ratbis} A function meromorphic on a neighborhood of a projective line in $\cp^2$ is rational.
\end{proposition}

\begin{proof} Take an affine chart $\cc^2_{z,w}$ on the complement of the projective line in question. We choose 
the center of coordinates close to the infinity line and the axes also  close to the infinity line. 
The function in question is rational in $z$ with fixed small $w$ and vice versa. Each function rational in 
two separate variables is rational  (Proposition \ref{2rat}). This proves Proposition \ref{2ratbis}.
\end{proof}

Proposition \ref{merrat} follows from 
Proposition \ref{2ratbis} and the above discussion.
\end{proof}

\section{Reduction to quasihomogeneously integrable 
$(p,q;\rho)$-billiards}

\begin{definition} \label{dquasint} Let $p,q\in\nn$, $1\leq q<p$, be coprime numbers. The curve 
$$\gpq:=\{ w^q=z^p\}\subset\cc^2\subset\cp^2$$
 will be called the {\it $(p,q)$-curve.} (It is injectively holomorphically parametrized by 
 $\cc^*$ via the mapping $t\mapsto(t^q,t^p)$.) 
Let $\rho\in\cc$. The {\it $(p,q;\rho)$-billiard} is the structure of singular holomorphic 
dual   billiard on the $(p,q)$-curve $\gpq$ 
defined by the family of involutions $\sigma_P:L_P\to L_P$, 
$P\in\gpq\setminus\{(0,0)\}$, all of them 
acting as the involution $\eta_{\rho}$ in the coordinate $\zeta=\frac z{z(P)}$ on $L_P$. 
\end{definition} 
\begin{definition} Recall that a polynomial $P(z,w)$ is {\it $(p,q)$-quasihomogeneous,} 
if it contains only monomials $z^kw^m$ with $(k,m)$ lying on the same line 
parallel to the segment $[(p,0),(0,q)]$. That is, a polynomial that becomes 
homogeneous after the substitution $z=t^q$, $w=t^p$, i.e., after restriction to the curve 
$\gpq$. A ratio of two $(p,q)$-quasihomogeneous polynomials will be called a 
{\it $(p,q)$-quasihomogeneous rational function.}
A $\pqr$-billiard is said to be {\it quasihomogeneously integrable,} if 
it admits a non-constant $(p,q)$-quasihomogeneous rational integral.
\end{definition}
The main result of the present section is the following theorem.
\begin{theorem} \label{quasi-pqr} Let $b$ be a non-linear irreducible germ of analytic curve 
at a point $B\in\cc^2$. Let $r=\frac pq$ be its projective Puiseux exponent, $(p,q)=1$, see (\ref{propuis}).  
Let $b$ admit a structure of meromorphically integrable singular 
dual   billiard,  $\rho$ be its residue at $B$.  
Then the $(p,q;\rho)$-billiard is quasihomogeneously integrable.
\end{theorem}

\subsection{Preparatory material. Newton diagrams}

Let us recall the well-known notion of Newton diagram of a germ of holomorphic function $f(z,w)$ at the origin. 
We consider that $f(0,0)=0$.  
To each  monomial $z^mw^n$ entering its Taylor series we put into correspondence 
the quadrant $K_{m,n}:=(m,n)+(\rr_{\geq0})^2$. Let $K(f)$ denote the convex hull of the union of 
the quadrants $K_{m,n}$ through all the Taylor monomials of the function $f$; it is an unbounded polygon 
with a finite number of sides. The {\it Newton diagram} 
$\bold{N}_f$ is the union of those edges of  the boundary $\partial K(f)$ that do not lie in the coordinate axes.

Fix a coprime pair of numbers $p,q\in\nn$, $(p,q)=1$. For every monomial $z^kw^m$  
 define its {\it $(p,q)$-quasihomogeneous degree:} 
$$\deg_{p,q}z^kw^m:=kq+mp.$$
Let $M_{p,q}(f)$ denote the minimal $(p,q)$-quasihomogeneous degree of a Taylor monomial of the function $f$. 
 The sum of its monomials $f_{km}z^kw^m$ with $\deg_{p,q}=M_{p,q}(f)$ is a $(p,q)$-quasihomogeneous polynomial called the 
{\it lower $(p,q)$-quasihomogeneous part} of the function $f$; it will be denoted by $\wt f_{p,q}(z.w)$. 
\begin{remark} In the case, when the Newton diagram $\bold{N}_f$ contains an edge parallel to the segment 
$[(p,0),(0,q)]$, the collection of  bidegrees of monomials entering the lower $(p,q)$-quasihomogeneous part $\wt f_{p,q}$ lies in the latter edge and contains its vertices. In the opposite case $\wt f_{p,q}$ is a 
monomial whose bidegree is the unique vertex $V$ of the Newton diagram such that the line through $V$ parallel to 
the above segment intersects $K(f)$ only at $V$. 
One has 
\begin{equation} \var^{-M_{p,q}(f)}f(\var^q z, \var^p w)=\wt f_{p,q}(z,w)+o(1), \text{ as } \var\to0,\label{newresc}\end{equation}
uniformly on compact subsets in $\cc^2$. 
\end{remark}

\begin{example} \label{exqpart} Let a germ of holomorphic function $f$ at the origin be irreducible (not a product of holomorphic 
germs vanishing at $0$). If $f(z,0)\equiv0$, then the Newton diagram $\bold{N}_f$ consists of just one, horizontal edge of height one. 
Let now $f(z,0), f(0,w)\not\equiv0$.   It is well-known that then 
the Newton diagram of the germ $f$ consists of one edge $[(ps_b,0),(0,qs_b)]$ with some $s_b\in\nn$ and 
coprime $p,q\in\nn$.
Let  $b$ be its zero locus. Then 
$b$ is a germ of curve injectively parametrized by a germ at $0$ of holomorphic map of the type 
\begin{equation}t\mapsto(t^{qs_b}, c_bt^{ps_b}(1+O(t))), \ \ c_b\neq0;\label{conscb}\end{equation} 
\begin{equation}\wt f_{p,q}(z,w)=(w^q-C_bz^p)^{s_b}, \ \ C_b=c_b^q,\label{pqpart}\end{equation}
up to constant factor. The proof of formula (\ref{pqpart}) repeats the proof of \cite[proposition 3.5]{gl2} 
with minor changes.
\end{example}

\begin{proposition} \label{pinters} 
Let $a$, $b$ be two irreducible germs of holomorphic curves at $O$. Let $r=r_b=\frac pq$ be 
the projective Puiseux exponent of the germ $b$, $(p,q)=1$, $r_a$ be that of the germ $a$. Let $(z,w)$ be affine coordinates centered at $O$ 
adapted to $b$; the coordinate $w$ being rescaled  so that $c_b=1$ in (\ref{conscb}).  

1) Let $\wt f_a(z,w)$ 
be the lower $(p,q)$-quasihomogeneous part of the germ of function $f_a$ defining $a$. Up to constant factor, 
the polynomial $\wt f_a(z,w)$ has one of the following types: 

a) $z^m$, if either $a$ is transversal to $b$, or $a$, $b$ are tangent and $r_a<r_b$; 

b) $w^m$, if $a$, $b$ are tangent and $r_a>r_b$; 

c) $(w^q-C_az^p)^m$, if $a$, $b$ are tangent and $r_a=r_b$;  $C_a$ is given by (\ref{pqpart}). 

2)  For every $P\in b\setminus\{ O\}$ consider the coordinate $\zeta:=\frac z{z(P)}$ on  the line $L_P$. 
Let $L=L_{(1,1)}$ denote the tangent line to 
$\gpq:=\{ y^q-\zeta^p=0\}\subset\cc^2_{\zeta,y}$ at the point $(1,1)$.  As $P\to O$,  the 
 $\zeta$-coordinates of 
 points of the intersection $a\cap L_P$ tend to some (finite or infinite) limits in $\oc_\zeta$. 
 The set of their finite limits coincides with the set of  zeros of the restriction to $L$ of the polynomial 
 $\wt f_a(\zeta,y)$. In the above cases a), b), c) it coincides respectively with the sets $\{0\}$, $\{\frac{r-1}r\}$ and 
 the collection of roots of the polynomial 
 $$\mcr_{p,q,C_a}(\zeta):=(1-r+r\zeta)^q-C_a\zeta^p.$$ 
\end{proposition} 
\begin{proof} Cases a) and b) correspond exactly to the cases, when the unique edge of the 
 Newton diagram of the function $f_a$ is not parallel to the segment $[(p,0), (0,q)]$; then the polynomial 
 $\wt f_a$ corresponds to one of its two vertices, and hence, is a power of either $z$, or $w$. 
 In Case c) Statement 1) of the proposition follows from (\ref{pqpart}). Statement 2) follows from 
 \cite[p.268, Proposition 2.50]{alg} and can be proved directly as follows. 
 Let $P\in b\setminus\{ O\}$, $z_0:=z(P)$. 
Consider the variable change $(z,w)=(z_0\zeta,z_0^ry)$ (for some chosen value of fractional power 
$z_0^r$). As $P\to O$, i.e., as $z_0\to0$, the curve $b$ written in the coordinates 
$(\zeta,y)$ tends to the curve $\gamma_{p,q}$, $(\zeta(P),y(P))\to(1,1)$, and $L_P\to L$. 
 The function $z_0^{-\frac1qM_{p,q}(f_a)}f_a(z_0\zeta,z_0^ry)$ tends to $\wt f_a(\zeta,y)$, by (\ref{newresc}). 
 This implies that each point of intersection $a\cap L_P$, whose $\zeta$-coordinate converges to a finite 
 limit after passing to a subsequence, does converge to a zero of the restriction 
 $\wt f_a|_L$, and each zero is realized as a limit. The $\zeta$-coordinates of the other intersection points 
 (if any) converge to infinity, by construction. The  polynomial $\wt f_a(\zeta,y)$  is a power of the polynomial 
 $\zeta$, $y$, $y^q-C_a\zeta^p$ respectively up to constant factor, by Statement 1). 
 The restrictions of the latter polynomials to the line $L$  are equal  respectively to $\zeta$, $1-r+r\zeta$ and $\mcr_{p,q,C_a}$. This together with the 
 above convergence  implies Statement 2). 
  \end{proof}

\subsection{Proof of Theorem \ref{quasi-pqr}.} 
 Let $R(z,w)=\frac{f(z,w)}{g(z,w)}$ be a non-constant meromorphic first integral of the 
dual   billiard on $b$. Here $f$ and $g$ are coprime germs of holomorphic functions at $B$ written in 
affine coordinates $(z,w)$ adapted to $b$. Let $r=r_b=\frac pq$ be the irreducible fraction representation of the 
projective Puiseux exponent $r$ of the germ $b$. Without loss of generality we consider that the corresponding  
constant $c_b$ in (\ref{conscb}) is equal to one, rescaling the coordinate $w$. Then the function $f_b(z,w)$
defining the curve $b$ is equal to $(w^q-z^p)^{s_b}$ plus higher $(p,q)$-quasihomogeneous terms,  
 by  (\ref{pqpart}). For a point $P\in b\setminus\{ B\}$ set $z_0=z(P)$. In the above rescaled 
 coordinates $(\zeta,y)=(z_0^{-1}z,z_0^{-r}w)$ one has 
$P\to(1,1)$, $L_P\to L$ ($L=L_{(1,1)}$ is the same, as Proposition \ref{pinters}),   
and the functions $z_0^{-\frac1qM_{p,q}(f)}f(z_0\zeta,z_0^ry)$, $z_0^{-\frac1qM_{p,q}(g)}g(z_0\zeta,z_0^ry)$ 
tends to $\wt f_{p,q}(\zeta,y)$ and $\wt g_{p,q}(\zeta,y)$ respectively, by (\ref{newresc}). 
The restriction $R|_{L_P}$ is $\sigma_P$-invariant, and $\sigma_P\to\er$ in the 
coordinate $\zeta$ on $L_P$. Therefore,  the restriction to the line $L$ of the ratio 
$$\wt R(\zeta,y):=\frac{\wt f_{p,q}(\zeta,y)}{\wt g_{p,q}(\zeta,y)}$$
is $\er$-invariant. Consider the action of group $\cc^*$  on $\cc^2$ by rescalings $(\zeta,y)\mapsto(\tau^q\zeta, \tau^py)$. It preserves the curve $\gamma_{p,q}$ punctured 
at the origin and at infinity and acts transitively on it. These rescalings multiply the 
quasihomogeneous rational function 
$\wt R$ by constants. This together with $\er$-invariance of its restriction to the tangent line $L$ implies 
invariance of its restriction to tangent line at any other point $Q\in\gamma_{p,q}$ under the involution $\er$ 
acting in the coordinate $\frac{\zeta}{\zeta(Q)}$. 
Therefore, $\wt R$ is a quasihomogeneous integral of the $\pqr$-billiard.
A priori it may be constant. This occurs exactly in the case, 
when $\wt g_{p,q}\equiv\la\wt f_{p,q}$, $\la\in\cc$.  But then replacing $g$ by $g-\la f$ cancels $\wt g_{p,q}$, and the lower 
$(p,q)$-quasihomogeneous part of the function $g-\la f$ is not constant-proportional to $\wt f_{p,q}$. The ratio 
$\frac f{g-\la f}$ being a meromorphic integral of the billiard on $b$, the above construction applied to it yields 
a non-constant quasihomogeneous integral of the $\pqr$-billiard. Theorem \ref{quasi-pqr} is proved.

\section{Classification of quasihomogeneously integrable $\pqr$-billiards}

The main result of the present section is the following theorem. 

\begin{theorem} \label{classpqr} A  $\pqr$-billiard is quasihomogeneously integrable, if and only if $p=2$, $q=1$ 
(i.e., the underlying curve $\gpq$ is a conic) and 
\begin{equation}\rho\in\mcm:=\{0, 1, 2, 3, 4\}\cup_{k\in\nn_{\geq3}}\{ 2\pm\frac2k\}.\label{rhov}\end{equation}
Then  the following  quasihomogeneous functions $R_{\rho}(z,w)$ are integrals.
 \medskip
 
  \begin{tabular}{|l|l|l|l|l|r|}
  \hline $\rho=0$ & $\rho=1$ & $\rho=2$ & $\rho=3$ & $\rho=4$  \\
  \hline \ & \ & \ & \ & \ \\
  $R_0=w-z^2$ & $R_1=\frac{w-z^2}z$ & $R_2=\frac{w-z^2}w$ & $R_3=\frac{w-z^2}{zw}$ 
  & $R_4=\frac{w-z^2}{w^2}$\\
  \ & \ & \ & \ & \ \\
  \hline
  \end{tabular}
  
  \medskip
  
  \begin{tabular}{|l|r|}
  \hline \ & \  \\
     $\rho=2-\frac2{2N+1}$ & $R_{\rho}(z,w)=\frac{(w-z^2)^{2N+1}}{\prod_{j=1}^N(w-c_jz^2)^2}$, \  
  $c_j=-\frac{4j(2N+1-j)}{(2N+1-2j)^2}$\\
  \ & \  \\
  \hline \ & \  \\
   $\rho=2+\frac2{2N+1}$ & $R_{\rho}(z,w)=\frac{(w-z^2)^{2N+1}}{w^2\prod_{j=1}^N(w-c_jz^2)^2}$, \  
  $c_j=-\frac{4j(2N+1-j)}{(2N+1-2j)^2}$\\
  \ & \  \\
  \hline \ & \  \\
  $\rho=2-\frac1{N+1}$ & $R_{\rho}(z,w)=\frac{(w-z^2)^{N+1}}{z\prod_{j=1}^N(w-c_jz^2)}$,  \ 
  $c_j=-\frac{j(2N+2-j)}{(N+1-j)^2}$\\
  \ & \  \\
  \hline \ & \  \\
  $\rho=2+\frac1{N+1}$ & $R_{\rho}(z,w)=\frac{(w-z^2)^{N+1}}{zw\prod_{j=1}^N(w-c_jz^2)}$,  \ 
  $c_j=-\frac{j(2N+2-j)}{(N+1-j)^2}$\\
  \ & \  \\
  \hline
\end{tabular}
   \end{theorem}
   {\bf Addendum to Theorem \ref{classpqr}.} {\it The variable change $(\wt z,\wt w)=(\frac{z}{w},\frac1{w})$ 
   transforms a $(2,1,\rho)$-billiard to a $(2,1,4-\rho)$-billiard. It interchanges 
 the integrals $R_\rho(\wt z,\wt w)$ and $R_{4-\rho}(z,w)$  given by the above formulas for every $\rho\in\mcm$.}
 
\medskip
Everywhere below by $L=L_{(1,1)}$ we denote the projective tangent line to  the curve $\gpq$ at the point $(1,1)$. 
\begin{remark} \label{remquasint} A $(p,q)$-quasihomogeneous rational function is an integral of the $\pqr$-billiard, if and only 
if its restriction to $L$ written in the coordinate $z$ is $\er$-invariant, see the above proof of Theorem \ref{quasi-pqr}.
\end{remark}
\begin{remark} \label{remquas}
It is well-known that each $(p,q)$-quasihomogeneous polynomial is a product of powers of {\it prime} 
quasihomogeneous polynomials $z$, $w$, $w^q-c_jz^p$ with $c_j\in\cc\setminus\{0\}$. 
\end{remark}
The proof of Theorem \ref{quasi-pqr} is based on the following  formula for the Hessian (calculated in the coordinates 
$(z,w)$) of a product 
\begin{equation} G(z,w)=(w^q-z^p)z^{\alpha}w^{\beta}\prod_{j=2}^M(w^q-c_jz^p)^{\mu_j}, \ \alpha,\beta,\mu_j\in\rr, \ c_j\neq0,1.\label{gzwmu}\end{equation}
\begin{proposition}\label{hesshom} Let $G$ be the same, as in (\ref{gzwmu}), with $c_j\in\cc\setminus\{0,1\}$. Set 
$$N:=1+\sum_{j=2}^M\mu_j, \ \rho_0=\frac23(r+1), \ r=\frac pq.$$
 There exists a $c\in\cc\setminus\{0\}$ such that 
\begin{equation} H(G)|_{\gpq}= c z^d;  \ d=3(pN+\alpha+\beta r-\rho_0).\label{ashes}\end{equation}
Formula (\ref{ashes}) holds for $c=qp(q-p)\left(\prod_{j=2}^p(1-c_j)\right)^3$. 
\end{proposition}
\begin{proof} The   Hessian of the defining polynomial $w^q-z^p$ of the curve $\gpq$ calculated in the coordinates $(z,w)$ 
is equal to 
$$H(w^q-z^p)=q(q-1)p^2w^{q-2}z^{2(p-1)}-p(p-1)q^2z^{p-2}w^{2q-2}.$$
Its restriction to  $\gpq$ is equal to the same expression with $w$ replaced by $z^r$, 
which yields $qp(q-p)z^{3p-2(r+1)}$. Each polynomial $w^q-c_jz^p$ being restricted to $\gpq$ is equal 
to $(1-c_j)z^p$. This together with (\ref{hess3})  implies (\ref{ashes}). 
\end{proof}
 
 \subsection{Case of $(p,q)$-quasihomogeneous polynomial integral}
 
 \begin{proposition} \label{pro2cases} Let a $\pqr$-billiard admit a $(p,q)$-quasihomogeneous 
 polynomial integral. Then  $\rho=0$,  $p=2$, $q=1$, and the polynomial $w-z^2$ is an integral. 
 \end{proposition}
\begin{proof} The  restriction to $L$ of a  $(p,q)$-quasihomogeneous 
 polynomial integral $\mcp$ is $\er$-invariant and has one pole, at infinity. Hence, $\er(\infty)=\infty$, thus, $\rho=0$. 
 Its restriction to $\gpq$ 
 should be constant, see Proposition \ref{proalg}. On the other hand, the latter restriction written in the coordinate 
 $z$ is a monomial $cz^\phi$. Therefore, $c=0$ and $\mcp|_{\gpq}\equiv0$. 
  Hence, 
  $$\mcp(z,w)=z^\alpha w^\beta\prod_{j=1}^M(w^q-c_jz^p)^{n_j}, \ c_j\neq0,  \  c_1=1, \ c_j\text{ are distinct, } \alpha,\beta\in\zz_{\geq0},$$ 
  see Remark \ref{remquas}. Set $k=n_1$, 
  $$G(z,w):=\mcp^{\frac1{k}}(z,w)=z^{\wt\alpha}w^{\wt\beta}(w^q-z^p)\prod_{j=2}^M(w^q-c_jz^p)^{\mu_j},$$ 
  $$\mu_j:=\frac{n_j}k, \ \wt\alpha:=\frac\alpha{k}, \ \wt\beta:=\frac{\beta}k, \ N:=1+\sum_{j=2}^M\mu_j.$$
 The restriction to $\gpq$ of the Hessian $H(G)$ is given by formula (\ref{ashes}). 
  Hence, 
$$\rho=-\frac{d}3=\rho_0-(pN+\alpha+\beta r)=\frac23(r+1)-qNr-(\alpha+\beta r),$$
by (\ref{rhod}). The latter right-hand side should vanish, since $\rho=0$. Therefore, 
$\frac23(r+1)\geq qNr$, hence $r\leq\frac2{3Nq-2}$. But $r>1$, and $N,q\geq1$. Therefore, $q=1$, and $r\leq2$. 
Hence, $p=r=2$. 
 
 Let us now show that the polynomial $w-z^2$ is an integral of the $(2,1;0)$-billiard. 
 Indeed,   its restriction to the tangent line $L=L_{(1,1)}$  is the polynomial $-1+2z-z^2=-(z-1)^2$; 
 here $\zeta=z$. 
 The latter polynomial is clearly invariant under the involution $\eta_0:z\mapsto2-z$. Hence, $w-z^2$ is an integral, by Remark 
 \ref{remquasint}. Proposition \ref{pro2cases} is proved.
 \end{proof}
 \subsection{Normalization of rational integral to primitive one} 
 
 Here we consider a quasihomogeneously integrable $\pqr$-billiard. We prove that its  integral (if it cannot be reduced to a polynomial) 
 can be  normalized to a primitive integral, see definitions and Lemma \ref{lchichi} below. 
 
 The restriction of a $(p,q)$-quasihomogeneous polynomial  $\mcp$ to 
 the tangent line $L=L_{(1,1)}$ to the curve $\gpq$ at the point $(1,1)$ is a polynomial in the coordinate $z$ on $L$. One has 
 \begin{equation}w|_L=1-r+rz, \ (w^q-cz^p)|_L=\mcr_{p,q,c}(z):=(1-r+rz)^q-cz^p.\label{prest}\end{equation}
\begin{definition} The roots of the restriction $\mcp|_L$ of a $(p,q)$-quasihomogeneous polynomial $\mcp$ 
will be called its {\it tangent line roots}. The 
linear combination of points representing roots with coefficients equal to their multiplicities is a divisor 
on $L\simeq\cc_z$. It will be called the {\it root divisor} of the polynomial $\mcp$ and denoted by 
$\chi(\mcp)$. Sometimes we will deal with $\chi(\mcp)$ as with a collection of roots, e.g., when we write inclusions 
that some points belongs to $\chi(\mcp)$. 
\end{definition}

\begin{definition}
Recall that the {\it complement of a divisor $\chi$ to a point $\theta$} is 
the divisor $\chi$  with the term corresponding to the point $\theta$  deleted. Set 
$$\theta_\rho:=\frac{\rho-1}\rho=\er(\infty).$$
 A $(p,q)$-quasihomogeneous polynomial $\mcp$ is called {\it $\er$-quasi-invariant,} if 
the complement $\chi(\mcp)\setminus\{\theta_\rho\}$  is $\er$-invariant. A $\er$-quasi-invariant polynomial 
$\mcp$ is  {\it $\er$-primitive}, if it is not a product of two $\er$-quasi-invariant 
polynomials. 
\end{definition}
\begin{proposition} \label{proprim} 1) A primitive $\er$-quasi-invariant polynomial $\mcp$ 
is (up to constant factor) a product $\prod_jQ_j$ of  some distinct prime 
$(p,q)$-quasihomogeneous polynomials $Q_j$ equal to $w^q-c_jz^p$,  $z$ or $w$. 

2) Any two prime factors 
$Q_k$, $Q_\ell$ are {\bf equivalent} in the following sense: there exists a finite sequence $k=j_1,j_2,\dots,j_m=\ell$ 
such that for every $s=1,\dots,m-1$ there exist tangent line roots $z_s$, $z_{s+1}$ of the polynomials $Q_{j_s}$ 
and $Q_{j_{s+1}}$ respectively such that $z_{s+1}=\er(z_s)$. 

3) If $\rho=r$, then either $\mcp=cw$, $c\in\cc\setminus\{0\}$, or 
$\mcp$ contains no $w$-factor. 

4) For any two distinct primitive $\er$-quasi-invariant polynomials  their  tangent line root collections do not intersect. 

5) Every $\er$-quasi-invariant polynomial is a product of powers of primitive ones.
\end{proposition}
The proposition follows from definition and the fact that the polynomial $w|_L=1-r+rz$ has one root $\frac{r-1}r$. 
\begin{definition} A  $(p,q)$-quasihomogeneous {\it  rational integral} of the $\pqr$-billiard is {\it $\er$-primitive,} if it is 
a ratio of nonzero powers of two non-trivial primitive $\er$-quasi-invariant $(p,q)$-quasihomogeneous polynomials.
\end{definition}
\begin{lemma} \label{lchichi} Let a $\pqr$-billiard be quasihomogeneously integrable and admit no 
polynomial $(p,q)$-quasihomogeneous integral. Then it admits a $\er$-primitive 
rational integral vanishing identically on $\gpq$, and $\rho\neq0$. 
\end{lemma}
\begin{proof} Let $R$ be a quasihomogeneous integral of the $\pqr$-billiard represented as an irreducible ratio 
of two quasihomogeneous polynomials:  numerator and denominator, both being non-constant 
(absence of polynomial integral). Its restriction to the curve $\gpq$ 
is constant, by Proposition \ref{proalg}. If the latter constant is finite non-zero, then the numerator and the denominator have equal $(p,q)$-quasihomogeneous degrees. Therefore, replacing the numerator by its linear combination 
with denominator one can get  another quasihomogeneous integral that vanishes identically on $\gpq$. If the above constant is 
infinity, we replace $R$ by $R^{-1}$ and get an integral vanishing on $\gpq$. Thus, we can and will consider 
that $R\equiv0$ on $\gpq$. Both numerator and denominator are $\er$-quasi-invariant, which follows from 
$\er$-invariance of the restriction of the integral to the tangent line $L=L_{(1,1)}$. Therefore, 
they are products of powers of primitive $\er$-quasi-invariant polynomials. Among all the 
$\er$-quasi-invariant primitive factors in the numerator and the denominator there are at least two distinct ones, 
by irreducibility and non-polynomiality of the ratio $R$. Take one of them $\mcp_1$, vanishing identically on $\gpq$ (hence, divisible 
by $w^q-z^p$) and another one $\mcp_2$. For every $i=1,2$ one has 
\begin{equation}\mcp_i(z,w)=z^{\alpha_i}w^{\beta_i}\prod_{j=1}^{N_i}(w^q-c_{ij}z^p), \ \ c_{ij}\neq0; \label{newmcp}\end{equation} 
$$\alpha_i,\beta_i\in\{0,1\}, \ 
\alpha_1\alpha_2=\beta_1\beta_2=0, \ c_{11}=1, \ \text{ all } c_{ij} \text{ are distinct,}$$
 by Proposition \ref{proprim}. Set now 
\begin{equation}R(z,w):=\frac{\mcp_1^{m_1}}{\mcp_2^{m_2}}, \ \ d_i:=\deg\mcp_i=N_ip+\alpha_i+\beta_i.
\label{rprim}\end{equation}
\begin{proposition} \label{tchichi} 
 The ratio (\ref{rprim}) of powers of two non-trivial primitive $\er$-quasi-invariant polynomials $\mcp_1$, $\mcp_2$ 
is an integral of the $\pqr$-billiard, if the following relation holds:

\begin{equation}\text{Case 1), } \theta_{\rho}\notin\chi(\mcp_1)\cup\chi(\mcp_2): \ \ \ \ \ \ \ \ \ 
 d_1m_1=d_2m_2.\label{dm1}\end{equation}
\begin{equation}\text{Case 2), } \theta_\rho\in\chi(\mcp_1): \ \ \ \ \ \ \  \ \ \ \ \ \ 
(d_1+1)m_1=d_2m_2.\label{dm2}\end{equation}
\begin{equation}\text{Case 3), } \theta_\rho\in\chi(\mcp_2): \ \ \ \ \  \ \ \ \ \ \ \ \ 
 d_1m_1=(d_2+1)m_2.\label{dm3}\end{equation}
\end{proposition} 
As is shown below, Proposition \ref{tchichi} is implied by the following obvious
\begin{proposition} \label{invzero} Let a rational function $R(\zeta)$ either do not vanish at 1, or 
have 1 as a root of even degree. Then it is $\er$-invariant, if and only if its 
zero locus and its pole locus are both $\er$-invariant.
\end{proposition}

\begin{proof} A rational function  $R$ is uniquely determined by its zero and pole loci 
up to constant factor. Therefore, if the latter loci 
 are invariant under a conformal involution $\er$, then  $R\circ\er=\pm R$. 
 The sign $\pm$ is in fact $+$, taking into account the condition at the point 1, which is fixed by $\er$.
 \end{proof}

\begin{proof} {\bf of Proposition \ref{tchichi}.} Let us show, case by case, that if the corresponding relation (\ref{dm1}), 
(\ref{dm2}) or (\ref{dm3}) holds, then the zero and pole divizors of the restriction $R|_L$ are $\er$-invariant. 
This together with Proposition \ref{invzero} implies that $R|_L$ is $\er$-invariant, and hence, $R$ is an integral 
(Remark \ref{remquasint}). 

Case 1): $\theta_{\rho}\notin\chi(\mcp_1)\cup\chi(\mcp_2)$ and $m_1d_1=m_2d_2$. Then the infinity in $L$ is not 
a pole of the restriction $R|_L$. Therefore, its zeros (poles)  are zeros of the polynomial 
$\mcp_1|_L$ (respectively, $\mcp_2|_L$). Their divisors are $\er$-invariant, by $\er$-quasi-invariance of the 
polynomials $\mcp_i$, and since the root collections of their restrictions to $L$ do not 
contain $\theta_\rho=\er(\infty)$. Hence, $R|_L$ is $\er$-invariant. 

Case 2):  $\theta_{\rho}\in\chi(\mcp_1)$ and $m_1(d_1+1)=m_2d_2$. Then the infinity in $L$ is 
a zero of multiplicity $m_1$ of the restriction $R|_L$. The point $\theta_\rho$ is a simple root of  the polynomial 
$\mcp_1|_L$, by assumption and primitivity. This together with its $\er$-quasi-invariance 
implies that the zero divizor of the function $R|_L$ is $\er$-invariant. Its pole divisor, i.e., the zero divisor of the 
function $\mcp_2^{m_2}|_L$ is also $\er$-invariant, as in the above discussion.

Case 3) is treated analogously to Case 2). 
\end{proof}

Proposition \ref{tchichi} immediately implies the statement of Lemma \ref{lchichi}, except for the statement 
that $\rho\neq0$. Suppose the contrary: $\rho=0$. Then $\er(\infty)=\infty\notin\chi(\mcp_1)$. 
Therefore, the restriction to $L$ of the $\er$-quasi-invariant polynomial $\mcp_1$ is $\er$-invariant (Proposition \ref{invzero}). 
Hence, $\mcp_1$ is a polynomial integral of the $\pqr$-billiard. The contradiction thus obtained proves 
that $\rho\neq0$ and finishes the proof of Lemma \ref{lchichi}. 
\end{proof}

 \subsection{Case of rational integral. Two formulas for $\rho$} 
Here we treat the case, when the $\pqr$-billiard in question 
 admits a  rational quasihomogeous integral and does not admit a polynomial one: thus, $\rho\neq0$ 
 (Lemma \ref{lchichi}). 
 Everywhere below we consider that the integral $R(z,w)$ is $\er$-primitive, vanishes on $\gpq$ (Lemma 
 \ref{lchichi}) and is given by formula (\ref{rprim}) with $m_1, m_2\neq0$ satisfying some of relations 
 (\ref{dm1})-(\ref{dm3}). Set 
\begin{equation}
G(z,w):=(R(z,w))^\frac1{m_1}=\frac{\mcp_{1}}{\mcp_{2}^\nu}, \ \nu:=\frac{m_2}{m_1}.
\label{newg}\end{equation}
  We prove two different formulas for the 
 residue $\rho$, deduced  
 
 - on one hand,  from formula (\ref{ashes}) for the Hessian $H(G)$  and 
 formula (\ref{rhod}) expressing $\rho$ via the asymptotic exponent $d$; 
 
 -  on the other hand, 
 by applying a similar argument to a special $\er$-invariant function on $L_P\simeq\cc_z$: the ratio of the 
 numerator of the integral and a power of $z-\frac{\rho-2}\rho$. Combining the two formulas 
 for $\rho$ thus obtained, we will show in Subsection 4.4 that $p=2$, $q=1$ and $\rho\in\mcm$. 

In our case formulas (\ref{ashes}) and (\ref{rhod}) yield 
\begin{equation}H(G)|_{w^q=z^p}=cz^d, \ d=3((N_1-\nu N_2)p+\alpha_1-\nu\alpha_2+r(\beta_1-\nu\beta_2)-\rho_0),\label{hgnew}\end{equation} 
$$r=\frac pq, \ \ \rho_0=\frac23(r+1),$$
\begin{equation}\rho=-\frac d3=\rho_0-(d_1-\nu d_2)-(r-1)(\beta_1-\nu\beta_2), \ d_i=N_ip+\alpha_i+\beta_i.
\label{rhooo}\end{equation}
This is the First Formula for $\rho$. Substituting to (\ref{rhooo}) the relations between 
the degrees $d_1$ and $d_2$ given by Proposition \ref{tchichi} and taking into account that 
$\beta_j=1$ if and only if $\tr:=\frac{r-1}r\in\chi(\mcp_j)$, we get
\begin{proposition} \label{pfor1} Let  $d_1$, $d_2$, $G$ be as above. Then one has the following 
formulas for the residue $\rho$ dependently on whether or not some of the numbers $\theta_\rho=
\er(\infty)=\frac{\rho-1}{\rho}$, 
$\theta_r=\frac{r-1}r$ lie in some of  $\chi(\mcp_{1,2})$:
\begin{center}
\begin{tabular}{|l|l|l|l|r|}
\hline \ & $\tro\notin\chi(\mcp_1)\cup\chi(\mcp_2)$ & $\tro\in\chi(\mcp_1)$ & $\tro\in\chi(\mcp_2)$ \\
\hline $\tr\notin\chi(\mcp_1)\cup\chi(\mcp_2)$ & $\rho=\rho_0$ & $\rho=\rho_0+1$ & $\rho=\rho_0-\nu$ \\
\hline $\tr\in\chi(\mcp_1)$ & $\rho=\rho_1:=\rho_0+1-r$ & $\rho=\rho_1+1$ & $\rho=\rho_1-\nu$ \\
\hline $\tr\in\chi(\mcp_2)$ & $\rho=\rho_2:=\rho_0-\nu(1-r)$ & $\rho=\rho_2+1$ & $\rho=\rho_2-\nu$\\
\hline
\end{tabular}
\end{center}
\end{proposition}

The Second Formula for the residue $\rho$ is given by the next lemma.
\begin{lemma} \label{lfor2} Let $\mcp(z,w)=z^\alpha w^\beta\prod_{j=1}^N(w^q-c_jz^p)$ 
 be a primitive $\er$-quasi-invariant $(p,q)$-quasihomogeneous polynomial vanishing on $\gpq$: $c_1=1$. Set 
\begin{equation}d_\mcp:=\deg\mcp=Np+\alpha+\beta, \  \ 
\hat d_{\mcp}:=\left[\begin{array}{lr}d_\mcp, \text{ if } \theta_\rho\notin\chi(\mcp)\\
d_\mcp+1, \text{ if } \theta_\rho\in\chi(\mcp).\end{array}\right.\label{mchi}\end{equation}
Then the residue $\rho$ is expressed by the  formula 
\begin{equation}\rho(\hat d_\mcp-2)=2(Np+\alpha+\beta r-\rho_0); \ \text{ here } \rho_0=\frac23(r+1).\label{for2}
\end{equation}
\end{lemma}
\begin{proof} The restriction of the polynomial $\mcp$ to the tangent line $L=\cc_z$ 
 is 
$$H(z):=z^\alpha(1-r+rz)^\beta\prod_{i=1}^N((1-r+rz)^q-c_iz^p), \ \deg H=d_{\mcp}=Np+\alpha+\beta.$$
The roots of the latter polynomial are exactly points of $\chi(\mcp)$. The involution $\er:L\to L$  
has two fixed points: those with $z$-coordinates 1 and  $\frac{\rho-2}\rho$. 
We consider the following auxiliary rational function 
\begin{equation} G(z):=\frac{H(z)}{(z-\frac{\rho-2}\rho)^{\hat d_\mcp}}.\label{gx}\end{equation}

{\bf Claim 1.} {\it The rational function $G$ is $\er$-invariant.}

\begin{proof} The zero divisor of the function $G|_L$ is $\er$-invariant. Indeed, the complement of the root divisor 
$\chi(\mcp)$ of the polynomial $H$ to $\theta_\rho$ is $\er$-invariant ($\er$-quasi-invariance of the polynomial 
$\mcp$). In the case, when  $H(\theta_\rho)=0$, one has $\hat d_\mcp=\deg H+1$, and hence, $\infty$ is a simple zero of the function $G$. The pole divisor of the function $H$ is the fixed point $\frac{\rho-2}{\rho}$ of the involution $\er$. 
This together with Proposition \ref{invzero} implies that $G$ is $\er$-invariant.
\end{proof}
\begin{corollary} For every $\la\in\cc$ the polynomial 
$$H_{\la}(z):=H(z)-\la\left(z-\frac{\rho-2}\rho\right)^{\hat d_\mcp}$$
has exactly two roots $\zeta_{\pm}(\lambda)$ converging to 1, as $\la\to0$. These roots are permuted by the 
involution $\er$. 
\end{corollary}

We  will deduce  formula (\ref{for2}) by  comparing asymptotics of the numbers $\zeta_{\pm}(\la)$ as roots 
of the polynomial $H_{\la}$ and writing the condition  that they should be permuted by the involution $\er$ 
with known Taylor series. To this end, we write the polynomials $H_{\la}$ and their roots in the new 
coordinate
$$u:=z-1; \ \ u_{\pm}:=u(\zeta_{\pm}(\la))=\zeta_{\pm}(\la)-1.$$
{\bf Claim 2.} {\it There exists a constant $A\in\cc^*$ such that as $u\to0$, one has}
\begin{equation} H(z)=H(1+u)=A(1+(Np+\alpha+\beta r-\rho_0)u+O(u^2))u^2.\label{ashz}\end{equation}
\begin{proof} One has  $p=qr$, 
\begin{equation}z^\alpha(1-r+rz)^\beta=(1+u)^\alpha(1+ru)^\beta=1+(\alpha+r\beta)u+O(u^2),\label{zalphab}\end{equation}
$$(1-r+rz)^q-z^p=(1+ru)^q-(1+u)^{p}$$
$$=\frac{q(q-1)r^2-p(p-1)}2u^2+\frac{q(q-1)(q-2)r^3-p(p-1)(p-2)}6u^3+O(u^4)$$
$$=\frac{p(1-r)}2u^2+\frac p6((p-r)(p-2r)-(p-1)(p-2))u^3+O(u^4)$$
\begin{equation}=\frac{p(1-r)}2u^2(1+(p-\rho_0)u+O(u^2)).\label{p1-r}\end{equation}
For $c_i\neq1$ one has the equality
$$(1-r+rz)^q-c_iz^p=((1-r+rz)^q-z^p)+(1-c_i)(1+u)^p=(1-c_i)(1+pu)+O(u^2).$$
Multiplying it with (\ref{zalphab}), (\ref{p1-r}) yields (\ref{ashz}) with $A=\frac{p(1-r)}2\prod_{i\geq2}(1-c_i)$. 
\end{proof}
\begin{corollary} One has $u_-=-u_+(1+o(1))$, as $\la\to0$, and 
\begin{equation}A(1+(Np+\alpha+\beta r-\rho_0)u_{\pm}+O(u_+^2))u_{\pm}^2=B(1+\frac{\hat d_\mcp\rho}2u_{\pm}+O(u_{\pm}^2)),\label{compab}\end{equation} 
$$B=B(\la)=\la\left(\frac2\rho\right)^{\hat d_\mcp}.$$
\end{corollary}
\begin{proof} One has $H(\zeta_{\pm})=\la(\zeta_{\pm}-\frac{\rho-2}\rho)^{\hat d_\mcp}$, by definition. 
Substituting (\ref{ashz})  to the latter formula yields (\ref{compab}).
\end{proof}

The involution $\er(z)$ written in the coordinate $u=z-1$ takes the form 
\begin{equation}\er:u\mapsto-\frac u{1+\rho u}.\label{sigmaup}\end{equation}
Therefore,  $u_-=-u_++(\rho+O(u_+))u_+^2$, since $u_{\pm}$ converge to $0$ and are permuted by $\er$. 
Dividing equations (\ref{compab}) for $u_+$ and 
for $u_-$ and substituting the latter asymptotic formula for $u_-$ yields 
$$\frac{1+(Np+\alpha+\beta r-\rho_0)u_++O(u_+^2)}{(1-(Np+\alpha+\beta r-\rho_0)u_++O(u_+^2))
(1-\rho u_+)^2}=1+\hat d_\mcp\rho u_++O(u_+^2),$$
$$2(Np+\alpha+\beta r-\rho_0)+2\rho=\hat d_\mcp\rho.$$
This proves (\ref{for2}).
\end{proof}
\subsection{Proof of the main part of Theorem \ref{classpqr}: necessity} 
The main part of Theorem \ref{classpqr} is given by the next lemma. 
\begin{lemma}  \label{lclass} 
Let a $\pqr$-billiard be quasihomogeneously integrable. Then $p=2$, $q=1$ and 
$\rho\in\mcm$. 
\end{lemma}
\begin{proof} The case of polynomial integrability, was already treated 
in Subsection 4.1. Let us  treat the case, when there are no polynomial integral. Then there exists 
a primitive integral $R$ vanishing on $\gpq$; let us fix it. One has 
$$R(z,w)=\frac{\mcp_{1}^{m_1}}{\mcp_{2}^{m_2}}(z,w), \ \mcp_{i}=z^{\alpha_i}w^{\beta_i}\prod_{j=1}^{N_i}(w^q-c_{ij}z^p),$$
$$ \alpha_j,\beta_j\in\{0,1\}, \ \alpha_1\alpha_2=\beta_1\beta_2=0, \ c_{11}=1, \ \ \text{ all }  \ c_{ij}  \ \text{ are distinct}.$$ 
The statement of the lemma  
 will be deduced by equating the two formulas 
 for the residue $\rho$ given by Proposition \ref{pfor1} and Lemma \ref{lfor2} (applied to  $\mcp_1$).

Case 1): $\theta_\rho=\er(\infty)\notin\chi(\mcp_1)\cup\chi(\mcp_2)$. Then 
\begin{equation}
\rho=\rho_0+(1-r)(\beta_1-\nu\beta_2)=\frac{2(N_1p+\alpha_1+\beta_1r-\rho_0)}{N_1p+\alpha_1+\beta_1-2},
\label{rhoo0}\end{equation}
by Proposition \ref{pfor1} and Lemma \ref{lfor2} applied to $\mcp_1$. 

Subcase 1a): $\beta_1=\beta_2=0$. Then (\ref{rhoo0}) yields 
$$(\rho_0-2)(N_1p+\alpha_1)=0, \ \ \rho_0=\frac23(r+1).$$
Hence, $\rho_0=r=2$, since $N_1p>0$, $\alpha_1\geq0$. This together with (\ref{rhoo0}) yields 
$$p=2, \ q=1, \  \rho=\rho_0=2.$$

Subcase 1b): $\beta_1=1$, $\beta_2=0$.  Then  (\ref{rhoo0}) yields 
$$2N_1p+2\alpha_1+2r=\rho_0(N_1p+\alpha_1+1)+(1-r)(N_1p+\alpha_1-1)$$
$$=\frac23(r+1)(N_1p+\alpha_1+1)+(1-r)
(N_1p+\alpha_1-1).$$
Substituting $r=\frac pq$ and multiplying the latter equation by $3q$ yields
\begin{equation} 6N_1pq+6\alpha_1q+6p=2(p+q)(N_1p+\alpha_1+1)+3(q-p)(N_1p+\alpha_1-1).\label{nro}
\end{equation}
Writing equation (\ref{nro}) modulo $p$ and dividing it by $q(\mo p)$ yields 
$$6\alpha_1=2(\alpha_1+1)+3(\alpha_1-1)=5\alpha_1-1(\mo p), \ \ \alpha_1\equiv-1(\mo p).$$
Thus, $\alpha_1\in\{0,1\}$ and $\alpha_1\equiv-1(\mo p)$, $p\in\nn$, $p\geq2$. Therefore, 
$$p=2, \ q=1, \ \alpha_1=1, \ \rho_0=2,$$
 $$\rho=\rho_0+(1-2)=1=\frac{2(2N_1+1+2-\rho_0)}{2N_1},$$
 by (\ref{rhoo0}). Hence, $2(2N_1+1)=2N_1$ and $N_1<0$. The contradiction thus obtained shows that Subcase 1b) is impossible. 
 
 Subcase 1c): $\beta_1=0$, $\beta_2=1$. Then 
 $$\rho=\frac{2(N_1p+\alpha_1-\rho_0)}{N_1p+\alpha_1-2}$$
 \begin{equation}=\rho_0+(r-1)\nu=\rho_0+(r-1)\frac{N_1p+\alpha_1}{N_2p+\alpha_2+1},\label{rhob2}
 \end{equation}
 by (\ref{rhoo0}) and since in our case $\nu=\frac{m_2}{m_1}=\frac{d_1}{d_2}$, see (\ref{dm1}). 
 Moving $\rho_0$ from the right- to the left-hand side, dividing both sides by $N_1p+\alpha_1$ and 
 multiplying them by the product of denominators in (\ref{rhob2}) yields
 $$(2-\rho_0)(N_2p+\alpha_2+1)=(r-1)(N_1p+\alpha_1-2).$$
 Substituting the value of $\rho_0$ and multiplying the latter equation by $3q$ yields
$$(4q-2p)(N_2p+\alpha_2+1)=3(p-q)(N_1p+\alpha_1-2).$$
 Reducing the latter equation modulo $p$ and dividing it by $q(\mo p)$ yields
 \begin{equation}4\alpha_2+3\alpha_1\equiv 2(\mo p).\label{4qp}\end{equation}
 In the case, when $\alpha_1=1$, one has $\alpha_2=0$ and $3\equiv 2(\mo p)$, which is  
 impossible, since $p\geq2$. Hence, $\alpha_1=0$. In this case $\alpha_2\in\{0,1\}$, $4\alpha_2-2=\pm2\equiv0(\mo p)$. 
 Hence, $p=2$, $q=1$, $\rho_0=2$. This together with the first equality in (\ref{rhob2}) implies that either $\rho=2$, 
 or $N_1p+\alpha_1-2=0$. If $\rho=2$, then $\rho=\rho_0$, which  contradicts the second equality in (\ref{rhob2}), 
 Therefore, 
 $$N_1p+\alpha_1-2=2N_1+\alpha_1-2=0, \ N_1=1, \ \alpha_1=0,$$
 \begin{equation}\rho=\rho_0+(r-1)\frac{N_1p+\alpha_1}{N_2p+\alpha_2+1}=2+\frac2{2N_2+\alpha_2+1}\in\mcm.
 \label{rhorho}\end{equation}

 Case 2): $\theta_\rho=\er(\infty)\in\chi(\mcp_1)$. Then
 $$\rho=\rho_0+1+(1-r)(\beta_1-\nu\beta_2)=
 \frac{2(N_1p+\alpha_1+\beta_1r-\rho_0)}{N_1p+\alpha_1+\beta_1-1},$$ 
 by (\ref{for2}) (applied to $\mcp_1$) and Proposition \ref{pfor1}. Multiplying by the denominator yields 
 \begin{equation}((r-1)\nu\beta_2-s)(N_1p+\alpha_1+\beta_1-1)=2s, \ s:=1+\beta_1(r-1)-\rho_0.\label{rho3}\end{equation}
  
 Subcase 2a): $\beta_2=0$. Then (\ref{rho3}) yields $(N_1p+\alpha_1+\beta_1+1)s=0$, hence 
 $s=0$ and $\rho_0=\frac23(r+1)=\beta_1(r-1)+1$. Thus, $\beta_1=1$, since $\rho_0>\frac43>1$, 
 $$\frac23(r+1)=r, \ r=p=2, \ q=1, \ \rho=\rho_0+1+1-r=\rho_0=2.$$
 
 Subcase 2b): $\beta_2=1$. Then $\beta_1=0$, and (\ref{rho3}) yields 
 $$(-s+\nu(r-1))(N_1p+\alpha_1-1)=2s, \ s=1-\rho_0<0.$$
 The right-hand side of the latter equation is negative, while the first factor in the left-hand side 
 is positive. Therefore, the second factor should be negative, which is obviously impossible. Hence, Subcase 2b) 
 is impossible.
 
 Case 3): $\theta_\rho=\er(\infty)\in\chi(\mcp_2)$. Then 
 \begin{equation}\rho=\rho_0-\nu+(1-r)(\beta_1-\nu\beta_2)=\frac{2(N_1p+\alpha_1+\beta_1r-\rho_0)}{N_1p+\alpha_1+\beta_1-2},\label{rhoo3}\end{equation}
 \begin{equation}\nu=\frac{d_1}{d_2+1}=\frac{N_1p+\alpha_1+\beta_1}{N_2p+\alpha_2+\beta_2+1},
 \label{fornu3}\end{equation}
  by (\ref{for2}),  Proposition \ref{pfor1} and (\ref{dm3}). 
  
  \medskip
  
  {\bf Claim 3.} {\it If $\beta_2=0$, then one has $\beta_1=\alpha_1=0$, $p=2$, $N_1=1$,}
  \begin{equation}\nu=\frac2{2N_2+\alpha_2+1}, \ \ \rho=\rho_0-\nu=2-\frac2{2N_2+\alpha_2+1}\in\mcm.
  \label{mcm3}\end{equation}
  \begin{proof} Multiplying (\ref{rhoo3}) with $\beta_2=0$ by its denominator yields
  $$(s+2-\nu)t=2t-2s, \ s:=\rho_0+\beta_1(1-r)-2, \ t=N_1p+\alpha_1+\beta_1-2,$$
  \begin{equation}(s-\nu)t=-2s.\label{-2s}\end{equation}
Note that one has always $t\geq0$. 
  
  Case $t>0$. Then $s-\nu$ and $s$ either  have different signs, or both vanish (which is impossible, 
  since $\nu>0$). Thus, $s-\nu<0<s$. But 
  $s=\frac13(2(r+1)+3\beta_1(1-r)-6)$. If $\beta_1=1$, then the latter expression in the brackets is 
  $2(r+1)-3r+3-6=-1-r<0$, hence $s<0$. The contradiction thus obtained shows that $\beta_1=0$. 
  Hence, $s=\rho_0-2$, 
  $$(\rho_0-2-\nu)(N_1p+\alpha_1-2)=-2(\rho_0-2), \ \ \nu=\frac{N_1p+\alpha_1}{N_2p+\alpha_2+1},$$
  $$(\rho_0-2)(N_1p+\alpha_1)=\nu(N_1p+\alpha_1-2).$$
  Substituting the above formula for the number $\nu$, multiplying by its denominator  and 
  by $3q$ and dividing by $N_1p+\alpha_1$ yields 
    \begin{equation}(2p-4q)(N_2p+\alpha_2+1)=3q(N_1p+\alpha_1-2).\label{3qpa}\end{equation}
    Reducing (\ref{3qpa}) modulo $p$ and dividing by $q(\mo p)$ yields 
    $-4(\alpha_2+1)\equiv 3(\alpha_1-2)(\mo p)$,  
    $$2-4\alpha_2-3\alpha_1\equiv0(\mo p).$$
    
    In the case, when $\alpha_1=1$, one has $\alpha_2=0$. Hence,  $-1\equiv0(\mo p)$, which is impossible. 
    Thus, $\alpha_1=0$. Then $2-4\alpha_2=\pm2\equiv0(\mo p)$. Hence, $p=2$, $q=1$, $\rho_0=2$, $s=0$. This 
    together with (\ref{-2s}) implies that $t=N_1p+\alpha_1-2=0$. Hence, 
    $N_1=1$, $\alpha_1=0$. The first statement of the claim is proved. Together with 
     (\ref{rhoo3}) and (\ref{fornu3}), it implies (\ref{mcm3}).
    \end{proof}
        
   {\bf Claim 4.} {\it In the case, when $\beta_2=1$, one has $p=r=\rho=2$, $q=1$.}
   
   \begin{proof} In this case $\beta_1=0$, and (\ref{rhoo3}) yields
   $$\rho=\rho_0+\nu(r-2)=\frac{2(N_1+\alpha_1-\rho_0)}{N_1p+\alpha_1-2},$$
   $$(s+2+\nu(r-2))t=2t-2s, \ s:=\rho_0-2=\frac23(r-2), \ t:=N_1p+\alpha_1-2,$$
   $$(s+\nu(r-2))t=-2s, \ \ (\frac23+\nu)(r-2)t=-\frac43(r-2), \ \ t\geq0, \ \nu>0.$$
The latter equality implies that $r=2$. Thus, $p=r=\rho_0=\rho=2$, $q=1$.
\end{proof}

Claims 3, 4 together with the previous discussion imply the statement of Lemma \ref{lclass}.
\end{proof}
\subsection{Sufficience and integrals. End of proof of Theorem \ref{classpqr}. Proof of the addendum}
Lemma \ref{lclass} reduces Theorem \ref{classpqr} to the following lemma. 
\begin{lemma} \label{lsuff} {\bf A)} The following statements are equivalent:

1) The  $(2,1;\rho)$-billiard is quasihomogeneously integrable.

2) One has $\rho\in\mcm$.

3) Either the mapping $T:=\eta_{2}\circ\er$ is the identity, or the points $\infty$ and $\frac12$ lie in the same $T$-orbit, i.e., $T^m(\frac12)=\infty$ for some $m\in\zz$. 

{\bf B)} For every $\rho\in\mcm$  the corresponding function $R_\rho(z,w)$ from the table in Theorem \ref{classpqr} 
is an integral of the $(2,1;\rho)$-billiard. 
\end{lemma}
\begin{proof} The implication 1) $=>$ 2) is given by Lemma \ref{lclass}. 

\begin{proof} {\bf of the equivalence 2) $<=>$ 3).} The map $T$ is identity, if and only if $\rho=2\in\mcm$. 
Let us consider the case, when $\rho\neq2$. Let us write the map $\er:L\to L$, $L=\cc_z$, in  the chart 
$$y:=\frac1{z-1}; \ \   y(1)=\infty, \ y(\infty)=0, \ y\left(\frac12\right)=-2, \ y(0)=-1.$$
The involution $\er$ fixes $1$, $\frac{\rho-2}\rho$, and $y(\frac{\rho-2}\rho)=-\frac\rho2$. 
Therefore, in the chart $y$  
\begin{equation}\er:y\mapsto-y-\rho, \ T=\eta_{2}\circ\er: y\mapsto y+\rho-2.\label{ur2}\end{equation}
The condition that $\rho\in\mcm\setminus\{2\}$ is equivalent to the condition saying that $\rho-2\in\{\frac2m \ | \ m\in\zz\setminus\{0\}\}$. 
The latter in its turn is equivalent to the condition that in the chart $y$ the point $y(\infty)=0$ is the $T^m$-image of the point $y(\frac12)=-2$. 
This proves equivalence of Statements 2) and 3). \end{proof}

\begin{proof} {\bf of the implication 3) $=>$ 1).} Case $\rho=0$ was already treated in Subsection 4.1; in this case the polynomial 
$R_0(z,w)=w-z^2$ is an integral. 

Case 1): $\rho=2$. Then  $\er(z)=\frac{z}{2z-1}$ fixes 1 and permutes 
$\infty$, $\frac12$. The restriction to $L=L_{(1,1)}$ of the function $R_2(z,w)=\frac{w-z^2}w$ written in the 
 coordinate $z$ is 
$\frac{(z-1)^2}{z-\frac12}$ up to constant factor. It is $\er$-invariant, by Proposition \ref{invzero} and invariance 
of its zero and pole divisors: double zero $1$ and the pair of simple poles $\frac12$, $\infty$. 
Hence, $R_2$ is an integral of the $(2,1;2)$-billiard. 

Case 2): $\rho=1$. Then the involution $\er(z)=\frac1z$ fixes 
$1$ and permutes $0$, $\infty$.   The 
 restriction to  $L$ of the  function $R_1(z,w)=\frac{w-z^2}z$  is equal to 
 $\frac{(z-1)^2}z$ up to constant factor. It is $\er$-invariant, by Proposition \ref{invzero} and 
 invariance of its zero and pole divisors, and $R_1$ is an integral, as in the above case.

Case 3): $\rho=3$. Then $\er(z)=\frac{2z-1}{3z-2}$ fixes 1 and  permutes $0$, $\frac12$.  
The restriction to $L$ of the  function $R_3(z,w)=\frac{w-z^2}{zw}$ has double zero $1$ and simple poles 
 $0$, $\frac12$. Hence, it is $\er$-invariant, and $R_3$ is an integral, as above.

Case 4): $\rho=4$. Then $\er(z)=\frac{3z-2}{4z-3}$ fixes the points 1 and $\frac12$. The latter points 
taken twice are respectively 
zero and pole divisors of the function $R_4|_{L}$. Hence, the latter function is invariant, and $R_4$ 
is an integral. 

Case 5): $\rho-2=\frac 2m$, $m\in\zz\setminus\{0\}$, $|m|\geq3$. Note that the integer number $m$ has the same sign, as 
the number $\rho-2$. Set 
$$\zeta_0=\frac12, \zeta_j=T^j(\zeta_0), \ j=0,\dots,m; \ \ \zeta_m=\infty,$$
$$ \chi:=\{\zeta_0,\dots,\zeta_{m-1}\}, \text{ if } \rho >2, \ \text{ i.e., } m>0,$$
$$ \chi:=\{\zeta_{m+1},\dots,\zeta_{-1}\}, \text{ if } \rho <2 \ \text{ i.e., } m<0.$$

{\bf Claim 5.} {\it The set $\chi$ is a collection $\chi(\mcp)$ of  roots of restriction to $L$ of a primitive $\er$-quasi-invariant $(2,1)$-quasihomogeneous 
polynomial $\mcp$. The polynomial $\mcp$ does not vanish identically on $\gamma=\gamma_{2,1}=\{ w=z^2\}$.}

\begin{proof} The restriction to $L$ of a prime quasihomogeneous polynomial $w-cz^2$ is $\mcr_c(z):=-cz^2+2z-1$. 
 The map $\eta_{2}$ permutes roots of the polynomial $\mcr_c$ for every $c$, 
since the sum of  inverses of roots is equal to 2 and $\eta_{2}$ acts as $v\mapsto2-v$ in the chart 
$v=\frac1z$. It permutes $\zeta_j$ and $\zeta_{m-j}$ for every $j=0,\dots,m$, since 
$y(\zeta_j)$ form an arithmetic progression, see (\ref{ur2}), $y(\zeta_0)=-2$, $y(\zeta_m)=0$, and $\eta_2:y\mapsto-y-2$.  
Therefore, the numbers $\zeta_j$ and $\zeta_{m-j}$ are roots of a quadratic polynomial $\mcr_{c_j}(z)$, unless $\zeta_j\in\{0,\frac12,\infty\}$.  The middle point $-1$ of the segment $[-2,0]\in\rr_y$ 
corresponds to $z=0$. One has $\zeta_j=0$ for some $j$, if and only if $m\in2\zz$, and then $j=\frac m2$. Therefore, $\chi$  consists of the union of roots 
of the  quadratic polynomials $\mcr_{c_j}$, $|j|=1,\dots,[\frac{|m|-1}2]$,  the point $\frac12$ (if $\rho>2$)  and  zero (if  $m\in2\zz$). The complement $\chi\setminus\{\frac12\}$ is $\eta_{2}$-invariant, by construction. 
One has $m\geq0$, if and only if $\rho-2\geq0$. 
The map $\er=\eta_{2}\circ T$ sends each $\zeta_j\in\chi$ to $\zeta_{m-j-1}$, 
since $T(\zeta_j)=\zeta_{j+1}$, by 
construction.  The latter image $\zeta_{m-j-1}=\er(\zeta_j)$  lies in $\chi$, except for the  case, when  
$$\rho<2, \ m<0,  \ \zeta_{m-j-1}=\zeta_m=\infty, \ j=-1, \ \zeta_{-1}=\er(\infty)=\theta_{\rho}\in\chi.$$ 
Indeed, if $\rho>2$, then $\zeta_j\in\chi$ exactly for $j\in[0,m-1]$, and in this case $m-1-j$ lies there as well. 
If $\rho<2$, then $\zeta_j\in\chi$ exactly for  $j\in[m+1,-1]$, and in this case $\zeta_{m-1-j}\in\chi$, 
unless $j=-1$. Thus, the complement $\chi\setminus\{\theta_{\rho}\}$ is $\er$-invariant. 
Any two points $\zeta_j,\zeta_k\in\chi$ can be obtained one from the other by 
the map $T^{j-k}=(\eta_{2}\circ\er)^{j-k}$ so that the latter map considered as a composition of $2|j-k|$ 
involutions is well-defined at $\zeta_k$ and $T^{j-k}(\zeta_k)=\zeta_j$. This follows from the above discussion. 
Thus, $\chi=\chi(\mcp)$ for some primitive 
$\er$-quasi-invariant $(2,1)$-quasihomogeneous polynomial $\mcp$ that is the product of the polynomials $w-c_jz^2$ 
and may be some of the monomials $z$, $w$. One has  
$1\notin\chi$, since $y(1)=\infty$. Hence, $\mcp|_\gamma\not\equiv0$.  Claim 5 is proved.
\end{proof}

Thus, if Statement 3) of the lemma holds, then there exist at least two distinct primitive quasihomogeneous 
$\er$-quasi-invariant polynomials:   $w-z^2$ and the above polynomial $\mcp$. 
Hence, the $(2,1;\rho)$-billiard admits a quasihomogeneous rational first integral
$$R=\frac{(w-z^2)^{m_1}}{(\mcp(z,w))^{m_2}},$$
 by Proposition \ref{tchichi}. Implication 3) $=>$ 1) is proved.
 \end{proof} 
 
Equivalence of Statements 1)--3) is proved. Now for the proof of Lemma \ref{lsuff} it remains to 
calculate the integrals. To do this, let us calculate the $c_j$ from the proof of Claim 5 for $j\neq\frac m2$. 
One has 
$$\rho-2=\frac2m, \ y_j:=y(\zeta_j)=-2+j(\rho-2)=-2+\frac{2j}m, \ \zeta_j=\frac1{y_j}+1=\frac{2j-m}{2(j-m)},$$
\begin{equation}c_j=(\zeta_j\zeta_{m-j})^{-1}=-\frac{4j(m-j)}{(2j-m)^2}, \ m=\frac2{\rho-2}\in\zz\setminus\{0\}.
\label{forcj}\end{equation}
Therefore, in the case, when $\rho>2$, the polynomial $\mcp(z,w)$ is the product of the quadratic polynomials 
$w-c_jz^2$, $j=1,\dots,[\frac{m-1}2]$, the polynomial $w$, and also the polynomial $z$ (which enters 
$\mcp$ if and only if $m\in2\zz$). In the case, when $\rho<2$, the polynomial $\mcp(z,w)$ is the product of the polynomials 
$w-c_jz^2$, $j=1,\dots, [-\frac{m+1}2]$, and also the polynomial $z$ (if $m\in2\zz$). 

Subcase 5a): $\rho=2+\frac2m$, $m=2N+1$, $N\in\nn$. Then $\mcp$ is the product of the 
polynomial $w$ and  $N$ above polynomials $w-c_jz^2$. Its degree is equal to $2N+1=m$. 
The divisor $\chi(\mcp)$ contains $\frac12$ and does not contain $\er(\infty)$, by construction and 
the above discussion. Substituting  $m=2N+1$ to formula (\ref{forcj})  
yields the formula for the coefficients $c_j$ given by the  table in Theorem \ref{classpqr}. 
This together with Proposition \ref{tchichi} implies that the corresponding function $R_{\rho}$ from the same table  is an integral of the $(2,1;\rho)$-billiard. 

Subcase 5b): $\rho=2-\frac2{2N+1}$.  It is treated analogously. In this case $\mcp$ is just the product of the 
above $N$ polynomials $w-c_jz^2$. The divisor $\chi(\mcp)$ contains $\theta_\rho=\er(\infty)$. This together with 
Proposition \ref{tchichi}  implies 
that the corresponding function $R_{\rho}$ from the  table  is an integral. 

Subcase 5c): $\rho=2+\frac1{N+1}$, $m=2N+2$. Treated analogously to Subcase 5a). But now the polynomial $\mcp$ 
contains the additional factor $z$, and substituting $m=2N+2$ to (\ref{forcj}) yields the formula for $c_j$ from the 
corresponding line of the table  in Theorem \ref{classpqr}.

Subcase 5d): $\rho=2-\frac1{N+1}$. Treated analogously. Lemma \ref{lsuff} is proved. The proof of 
Theorem \ref{classpqr} is complete.
\end{proof}

\begin{proof} {\bf of the addendum to Theorem \ref{classpqr}.} The equation for the curve 
$\gamma=\{ w=z^2\}$ in the new coordinates $(\wt z,\wt w)$ is the same: $\wt w=\wt z^2$. 
The variable change $(z,w)\mapsto(\wt z,\wt w)$ preserves the point $(1,1)$, and hence, 
the corresponding tangent line $L=L_{(1,1)}$ to $\gamma$, along which one has 
$$w=2z-1, \ \ \ \wt z=\frac z{2z-1}=\eta_2(z).$$
Therefore, in the coordinate $\wt z$ the involution $\er$ takes the form 
\begin{equation}\eta_2\circ\er\circ\eta_2=\eta_{\wt\rho}, \ \  \ \wt\rho:=4-\rho;\label{eta2rho}\end{equation} 
the latter formula follows from (\ref{ur2}). 
Thus, the variable change in question transforms the  $(2,1;\rho)$-billiard to the 
$(2,1;4-\rho)$-billiard. 

In new coordinates one has $R_{\rho}(z,w)=R_{4-\rho}(\wt z,\wt w)$. Indeed, 
$$R_0(z,w)=w-z^2=\frac1{\wt w^2}(\wt w-\wt z^2)=R_4(\wt z,\wt w),$$
$$ R_1(z,w)=\frac1{\wt w\wt z}(\wt w-\wt z^2)
=R_3(\wt z,\wt w), \ R_2(z,w)=R_2(\wt z,\wt w).$$
For the other integrals $R_{\rho}$ from the table in Theorem \ref{classpqr} the proof is analogous. 
The addendum is proved.
\end{proof}

\section{Local branches. Proof of Theorem \ref{typesing}}
The main result of this section is the following theorem.
\begin{theorem} \label{lalt} Let a non-linear irreducible germ of analytic curve $b\subset\cp^2$ at a point $O$ 
 admit a germ of  singular holomorphic dual   billiard structure with a meromophic integral $R$. 
 Then the germ $b$ is quadratic and  one of the two following statements holds: 
 
 a) either $b$ is regular; 

b) or $b$ is singular,  the  integral $R(z,w)$ is a rational function that is   
constant along the projective tangent line $L_O$ to $b$ at $O$, and  
the punctured line $L_O\setminus\{ O\}$ is a regular leaf of the foliation $R=const$ on $\cp^2$.
\end{theorem}
Theorem \ref{lalt} will be proved in Subsection 5.2. Theorem \ref{typesing} will be deduced from it in Subsection 5.3. 
Quadraticity of germ in Theorem \ref{lalt} follows from Theorems \ref{quasi-pqr} and \ref{classpqr}. The proof of Statement b) of Theorem \ref{lalt} for 
a singular germ is based on Theorem \ref{tufold} (stated and proved in Subsection 5.1), which yields 
a formula for residue of a meromorphically integrable singular dual billiard on a singular quadratic germ in terms of 
its self-contact order. 

\subsection{Meromorphically integrable dual billiard structure on singular germ: formula for residue}
Recall that each non-linear irreducible germ $b$ of analytic curve at $O\in\cc^2$ in adapted coordinates centered at $O$ 
admits an injective holomorphic parametrization 
\begin{equation}t\mapsto(t^{qs},\phi(t)), \ \ \phi(t)=ct^{ps}(1+O(t)), \ 1\leq q<p, \ s, q, p\in\nn,\label{paragerm}\end{equation} 
$$ c\in\cc\setminus\{0\}; \ \ (q,p)=1, \   \ r=\frac pq \text{ is the projective Puiseux exponent.}$$

\begin{definition}  A non-linear irreducible germ $b$ will be called {\it primitive},  if $s=1$ in (\ref{paragerm}) (following a suggestion of E.Shustin). 
\end{definition}

\begin{remark} A quadratic germ is primitive, if and only if it is regular. It is well-known that  if $s\geq2$, then the function 
$\phi(t)$ in (\ref{paragerm}) satisfies one of the   two following statements:

a) either there exist a $\delta\in\qq_{>0}$ and $\nu_1,\dots,\nu_{s-1}\in\cc\setminus\{0\}$, 
such that
\begin{equation} \phi(te^{\frac{2\pi i j}s})-\phi(t)=\nu_j t^{qs(r+\delta)}(1+o(1)), \text{ for 
 } j=1,\dots,s-1, \text{ as } 
t\to0;\label{deltaunif}\end{equation}

b) or there exit at least two distinct $\delta_1, \delta_2\in\qq_{>0}$ 
and distinct  $j_1, j_2\in\{1,\dots,s-1\}$ for which (\ref{deltaunif}) holds 
with $\delta$ replaced by $\delta_{1}$ and $\delta_{2}$ respectively.  
\end{remark}

\begin{definition} If $b$ is  not primitive and (\ref{deltaunif}) holds for a unique $\delta$, then $b$ will be 
called {\it uniformly ($\delta$-) folded}. 
\end{definition}
\begin{remark} This definition is equivalent  to the well-known definition 
of a germ having {\it two Puiseux pairs}.  
 \end{remark}
 
 \begin{theorem} \label{tufold} 
 Let an irreducible quadratic germ $b$ of analytic curve at $O\in\cc^2$ admit a structure of meromorphically 
 integrable singular dual   billiard with residue $\rho$ at $O$. Then $b$ is either regular, or uniformly 
 $\delta$-folded with $\delta$ related to $\rho$ by the formula
 \begin{equation} \rho=2+\delta.\label{rhode}
 \end{equation}
 \end{theorem}
 
 \begin{proof}
 Below we prove a more general theorem. To state it, let us introduce the following 
 definition.
 
 \begin{definition} Let $a$ and $b$ be two irreducible  germs of analytic curves 
 at $O\in\cc^2$ tangent to each other. Let $b$ be non-linear; let $r=r_b$ be its projective Puiseux exponent. 
Let $\delta\in\qq_{>0}$. We say 
 that $a$ is a {\it $\delta$-satellite for $b$,} if  $a$, $b$ are graphs of two multivalued functions 
 $\{ w=g_a(z)\}$, $\{ w=g_b(z)\}$ (represented by Puiseux series in $z$) satisfying 
 the following statement: there exist a sector $S$ with vertex at $0$, a 
 $c\in\cc\setminus\{0\}$, and holomorphic  branches of the functions $g_a(z)$, 
 $g_b(z)$  over $S$ (near $0$) for which 
 \begin{equation}g_a(z)-g_b(z)=cz^{r+\delta}(1+o(1)), \text{ as } z\to0, \ z\in S.
 \label{gab}\end{equation} 
 \end{definition}
 \begin{remark} \label{exsp}  Any two satellites have the same Puiseux exponent. 
  A uniformly $\delta$-folded  germ $b$ is a $\delta$-satellite for itself. 
 \end{remark}
 \begin{remark} \label{remsp} If $a$ and $b$ are $\delta$-satellites, then the above sector $S$ can be chosen with angle 
 arbitrarily large and containing an arbitrary given ray. 
 A priori it may happen that  $a$ and $b$ are $\delta_1$- 
 and $\delta_2$-satellites  with different $\delta_1,\delta_2>0$ 
 corresponding to two different pairs of holomorphic branches. (This holds, e.g.,  for $a=b$, if $b$ is neither primitive, 
 nor uniformly folded.) Two germs that 
 are $\delta$-satellites for a unique $\delta>0$ are called 
 {\it pure $\delta$-satellites.} 
 \end{remark}
 
 \begin{theorem} \label{tsputnik} Let an irreducible quadratic germ $b$  at $O\in\cc^2$ 
 admit a structure of meromorphically 
 integrable singular dual   billiard. Let the corresponding involution family  
 have residue $\rho$ at $O$. Let $a$ be an irreducible germ with the 
 same base point $O$ that lies in a level curve of the meromorphic integral. 
 Let $a$ be a $\delta$-satellite of the germ $b$. Then they are pure $\delta$-satellites, and the 
  corresponding number $\delta$ is  given by formula (\ref{rhode}). 
 \end{theorem}
\begin{proof} Let $S$, $g_a$ and $g_b$ be the same, as in (\ref{gab}). Fix a smaller sector $S'$, $\overline{S'}\setminus\{0\}\subset S$. 
The graphs of the functions $g_a$, $g_b$ over the sector $S$ will be denoted by $\Gamma_a$, 
$\Gamma_b$ respectively. Fix a 
 $z_0\in S'$, set $P=(z_0, g_b(z_0))\in\Gamma_b\subset b$. Let $L_P$ denote the line tangent to $b$ at $P$. 
 We introduce the coordinate 
 $$u:=\zeta-1=\frac z{z_0}-1$$
 on the tangent line $L_P$. 
 Let us find asymptotics of those $u$-coordinates of points of the intersection $\Gamma_a\cap L_P$, that tend to zero, as $z_0\in S'$ 
 tends to $0$. 
 
\begin{proposition} \label{rro} Let  $a$ and $b$ be  non-linear irreducible $\delta$-satellite germs at a point $O\in\cc^2$ with Puiseux exponent $r$. Let 
$S$, $g_a$, $g_b$  be the same, as in (\ref{gab}). Let $S'$ and 
$\Gamma_{a,b}$ be as above. 
As $P=(z_0,g_b(z_0))\to O$, $z_0\in S'$,  the intersection $\Gamma_a\cap L_P$ 
 contains exactly two points whose $u$-coordinates converge to zero. Their 
 $u$-coordinates $u_{\pm}$ are related by the asymptotic formula
 \begin{equation} u_-=-u_++(\rho_0+\delta)u_+^2+o(u_+^2), \ \ \rho_0=\frac23(r+1).
 \label{upmas}\end{equation}
 \end{proposition}
 
 \begin{proof} Without loss of generality we consider that $g_b(z)\simeq z^r(1+o(1))$, as $z\to0$, 
 rescaling the coordinate $w$.  Let us work in the further  rescaled coordinates $(\zeta,y)$, $(z,w)=(z_0\zeta, z_0^ry)$, in which 
 $\Gamma_{a,b}$ are graphs of functions $h_{a,b}(\zeta)$, $h_{a,b}(\zeta)$ converging 
 to $\zeta^r$ together with derivatives uniformly on compact subsets in $S$, as $z_0\to0$. In the new coordinates 
 $$\zeta(P)=1, \ h_a(\zeta)-h_b(\zeta)=cz_0^\delta(1+u)^{r+\delta}(1+\theta(z_0,u)),$$
 $\theta(z_0,u)\to 0$, as $z_0\to0$,  uniformly with derivatives in $u$ lying in a disk centered at $0$, by (\ref{gab}). Hence,  
 $$ \theta(z_0,u)=\chi(z_0)+o(u), \ \chi(z_0)\to0, \text{ as } z_0\to0.$$
The $u$-coordinates of points of the intersection $\Gamma_a\cap L_P$ 
  are found from the equation 
 $$ h_b(1)+h_b'(1)u=h_a(1+u)=h_b(1+u)+ cz_0^\delta(1+u)^{r+\delta}(1+\chi(z_0)+o(u)).$$
Moving $h_b(1+u)$ to the left-hand side and expressing the new 
 left-hand side by  Taylor formula with base point $1$ we get: 
 \begin{equation}-\frac12h_b''(1)u^2-\frac16h_b'''(1)u^3+o(u^3)= 
 cz_0^\delta(1+u)^{1+\delta}(1+\chi(z_0)+o(u)).\label{gb''}\end{equation}
 Since $h_b(\zeta)\to\zeta^r$, we get  $h_b''(1)=r(r-1)(1+\phi(z_0))$, $\phi(z_0)\to0$, $h_b'''(1)\to r(r-1)(r-2)$, as $z_0\to0$. 
 Substituting these expressions to (\ref{gb''}) yields 
 \begin{equation}-u^2(1+\frac{r-2}3u+o(u)+\phi(z_0))=\frac{2c}{r(r-1)}z_0^{\delta}(1+u)^{r+\delta}(1+\chi(z_0)+o(u)),\label{equu}\end{equation}
$\phi(z_0), \chi(z_0)\to0$. Equation (\ref{equu}) has exactly two solutions $u_\pm$ that tend to zero, as 
 $z_0\to0$: they have asymptotics 
 \begin{equation}u_+\simeq -u_-\simeq \sqrt{-\frac{2c}{r(r-1)}}z_0^{\frac{\delta}2}(1+o(1)).\label{u+d2}\end{equation}
Let us now prove (\ref{upmas}). Dividing equations (\ref{equu}) written for $u+$ 
and $u_-$ and taking into account that $u_+\simeq -u_-$, we get 
$$\left(\frac{u_+}{u_-}\right)^2\frac{1+\phi(z_0)+\frac{r-2}3u_++o(u_+)}{1+\phi(z_0)-\frac{r-2}3u_++o(u_+)}=\left(\frac{1+u_+}{1-u_+}\right)^{r+\delta}(1+\chi(z_0)+o(u_+)),$$
$$\frac{u_+}{u_-}=-(1+(r+\delta-\frac{r-2}3)u_++o(u_+)=-(1+(\rho_0+\delta)u_++o(u_+)).$$
The latter formula implies (\ref{upmas}). Proposition \ref{rro} is proved.
\end{proof}

Let now $b$ be  a singular quadratic germ equipped with a meromorphically integrable singular dual billiard structure. 
Fix two graphs $\Gamma_a\subset a$, $\Gamma_b\subset b$ satisfying (\ref{gab}) over $S$ with some 
$\delta>0$. For the proof of Theorem \ref{tsputnik} it suffices to show that $\rho=2+\delta$. Suppose the 
contrary: 
\begin{equation} \Theta:=\rho-2-\delta\neq0.\label{sneq}\end{equation}
 The curve $a$ should lie in a level curve $\alpha$ of the meromorphic integral, which is   
a one-dimensional analytic subset in a neighborhood of the origin $O$ and hence, has 
finite intersection index with the tangent line $L_O$ at $O$. Therefore, only finite 
and uniformly bounded number of points of   intersection $\alpha\cap L_P$ converge to $O$, as $P\to O$, i.e., 
as $z_0\to0$. Using the next proposition, we show that for every $N\in\nn$ and $z_0$ small enough (dependently on $N$) 
there are at least $N$ above intersection points that converge to $O$. The contradiction thus obtained will 
prove Theorem \ref{tsputnik}. 

We use the following characterization of satellite germs. 

\begin{proposition} \label{asintsp} Let $a$, $b$ be two  irreducible  germs of holomorphic curves at $O\in\cc^2$. Let $b$ be non-linear, 
 $r=\frac pq$ be  its Puiseux exponent, $(p,q)=1$. 

1) The germs $a$, $b$ are satellites, if and only the germ $a$ has the same Puiseux exponent $r$, 
and the lower $(p,q)$-quasihomogeneous parts of their defining functions are powers of one and the 
same prime $(p,q)$-quasihomogeneous polynomial $w^q-cz^p$.  

2) Let $a$ and $b$ be quadratic germs. They are satellites if and only if all 
the points of intersection $a\cap L_P$ have $\zeta$-coordinates, $\zeta=\frac z{z(P)}$, that tend to one. 
This holds if and only if {\bf some} point of the above intersection has $\zeta$-coordinate that tends to one. 
\end{proposition}

\begin{proof} Clearly the germs $a$, $b$ cannot be satellites, if they have different Puiseux exponents. Let  $f_a$, $f_b$ be the functions defining 
 $a$ and $b$, and let $\wt f_a$, $\wt f_b$ be their 
lower $(p,q)$-quasihomogeneous parts. Then up to constant factor, $\wt f_g(z,w)=(w^q-C_gz^p)^{s_g}$, $g=a,b$, 
$s_g\in\nn$, $C_g\in\cc\setminus\{0\}$, see (\ref{pqpart}). Without loss of generality we can and will consider 
that $C_b=1$, rescaling $w$.  The  curves $b$  and $a$ are 
 parametrized respectively by $t\mapsto(t^{qs_b}, t^{ps_b}(1+o(1)))$ and  $\tau\mapsto(\tau^{qs_a}, 
c_a\tau^{ps_a}(1+o(1)))$, $c_a^q=C_a$, see the discussion in 
Example \ref{exqpart}. Therefore, they are satellites, if and only if $C_a=1$. 
This proves Statement 1). The equality $C_a=1$ is equivalent to the statement that the restrictions to 
the line $L=L_{(1,1)}$ (tangent to the curve $\{ y=\zeta^2\}$ at $(1,1)$) of the quasihomogeneous polynomials  
$\wt f_a(\zeta,y)$ and $\wt f_b(\zeta,y)$ have the same roots. The latter roots are exactly the finite limits of  the $\zeta$-coordinates 
of points of the intersections $a\cap L_P$ and $b\cap L_P$ respectively (Proposition \ref{pinters}). 
 In the case of quadratic germs 
the polynomial $\wt f_b|_L=-1+2\zeta-\zeta^2=-(1-\zeta)^2$ has just one, double root 1. If $\wt f_a\neq\wt f_b$, i.e., 
$C_a\neq1$, then the polynomial $\wt f_a|_L=-1+2\zeta-C_a\zeta^2$ does not vanish at 1. This together with 
Statement 1) and the above discussion proves Statement 2).
\end{proof}

Recall that in the chart $\zeta$ the dual billiard involution $\sigma_P$ converges 
to  $\eta_{\rho}$. Therefore, in the chart $u$ it converges to the 
involution $u\mapsto-\frac u{1+\rho u}$, see (\ref{sigmaup}).
Hence, the germ  of the involution $\sigma_P$ at $u=0$ acts as 
\begin{equation}\sigma_P:u\mapsto-u+(\rho+\phi(z_0))u^2+\dots, \ \phi(z_0)\to0, \text{ as } z_0\to0.\label{sigmapu}
\end{equation}
The intersection points from  Proposition \ref{rro} with $u$-coordinates $u_\pm$ will be denoted 
by $A_{0\pm}$. The point  $A_{1+}=\sigma_P(A_{0-})\in L_P$ should lie in the same level curve $\alpha$ of the 
integral, as $A_{0-}$. Therefore, it lies in some irreducible germ $a_1$ of holomorphic curve at $O$, since 
the germ of  $\alpha$  is analytic. 
One has 
\begin{equation} u_{1+}:=u(A_{1+})=-u_-+(\rho+o(1))u_-^2=u_++\Theta u_+^2+o((u_+)^2),\label{usat}\end{equation}
by (\ref{upmas}), (\ref{sneq}), (\ref{sigmapu}), and since $r=\rho_0=2$. In particular, $\zeta(A_{1+})=1+u_{1+}\to1$, as $z_0\to1$, and $u_{1+}\simeq u_+$. 
Hence, $u_{1+}$ has asymptotics (\ref{u+d2}), and $a_1$ is a $\delta$-satellite of the germ $b$ (Proposition \ref{asintsp} and 
(\ref{u+d2})). 
Therefore, $a_1$ intersects $L_P$ at another point $A_{1-}$ with 
$$u_{1-}:=u(A_{1-})=-u_{1+}+(2+\delta)u_{1+}^2,$$ 
by (\ref{upmas}). The point $A_{2+}:=\sigma_P(A_{1-})$ also lies in the intersection $\alpha\cap L_P$, and 
$$u_{2+}:=u(A_{2+})=-u_{1-}+(\rho+o(1))u_{1-}^2$$
$$=u_{1+}(1+\Theta u_{1+}+o(u_{1+}))=u_++2\Theta u_+^2+o(u_+^2).$$
Repeating this procedure  we get a sequence of distinct points $A_{k+}\in \alpha\cap L_P$ 
with coordinates asymptotic to $u_++k\Theta u_+^2+o(u_+^2)$, $k\in\nn$. 
Passing to limit we get  that the level curve $\alpha$, which is a one-dimensional 
 analytic subset in a neighborhood of the point $O$,  has infinite intersection index with the tangent line $L_O$ 
at $O$. This is obviously impossible. The contradiction thus obtained proves Theorem \ref{tsputnik}. 
 \end{proof}

 Theorem \ref{tsputnik} together with Remarks \ref{exsp}, \ref{remsp} imply the statement of Theorem \ref{tufold}.
 \end{proof}

\subsection{Singular quadratic germs. Proof of Theorem \ref{lalt}}

In the proof of Theorem \ref{lalt} we use Theorem \ref{tufold} and the following proposition. To state it, 
let us recall the following definition.

\begin{definition} \cite[definition 3.3]{gs} Let $L\subset\cp^2$ be a line, and let $O\in L$. 
A {\it $(L,O)$-local multigerm (divisor)} is respectively a finite union (linear combination $\sum_jk_jb_j$ with 
$k_j\in\rr\setminus\{0\}$) 
of distinct irreducible germs of analytic curves $b_j$ (called {\it components}) at base points  $B_j\in L$ 
such that each germ at  $B_j\neq O$ is different from the line $L$. (A germ at $O$ can be arbitrary, 
in particular, it may coincide with the line germ $(L,O)$.)  
The {\it $(L,O)$-localization} of an algebraic curve (divisor) in $\cp^2$ is the  $(L,O)$-local 
multigerm (divisor) formed by all its local branches $b_j$ of the above type.
\end{definition}

\begin{proposition} \label{pmgerm} Let $b$ be an irreducible germ of holomorphic curve at $O\in\cc^2$, and let 
$(z,w)$ be coordinates adapted to $b$ in which the corresponding constant $C_b$ from (\ref{pqpart}) is equal to 
one. Let $L_O$ be the projective tangent line to $b$ at $O$. Let $\Gamma$ be a $(L_O,O)$-local multigerm. 
Let $\sigma_P:L_P\to L_P$ be a family of projective involutions, $P\in b\setminus\{O\}$, such that 
the intersections $\Gamma\cap L_P$ are $\sigma_P$-invariant for all $P\in b$ close enough to $O$. 
Let $\sigma_P$ converge to $\er$ in the coordinate $\zeta:=\frac z{z(P)}$ on $L_P$, as $P\to O$. Let the corresponding 
number $\rho$ be greater than the projective Puiseux exponent $r=r_b$ of the germ $b$. Then all the 
germs in 
$\Gamma$ are based at the point $O$, and the $\zeta$-coordinate of each point of  the intersection  
$\Gamma\cap L_P$ has a finite limit, as $P\to O$. 
\end{proposition}
\begin{proof} The proof of Proposition \ref{pmgerm} is analogous to the proof of theorem 4.24 in \cite[p. 1037]{gl2}. 
The intersection points with those germs in 
$\Gamma$ that are based at points different from $O$ (if any) have $\zeta$-coordinates that tend to infinity, since 
their $z$-coordinates tend to either infinity, or non-zero finite limits, as $z(P)\to0$. Suppose the contrary to the statement of the proposition: 
the $\zeta$-coordinate of some  
point of the intersection $\Gamma\cap L_P$ tends to infinity. Its  $\sigma_P$-image 
also lies in $\Gamma\cap L_P$, by invariance, and has $\zeta$-coordinate 
converging to $\theta_\rho:=\frac{\rho-1}\rho=\er(\infty)$, since $\sigma_P(\zeta)\to\er(\zeta)$.  This implies 
that there exists a germ $b_1\subset\Gamma$ based at $O$ whose intersection point with $L_P$ has $\zeta$-coordinate 
converging to $\theta_\rho$. One has $\theta_\rho\in(\theta_r,1)$, $\theta_r=\frac{r-1}r$, since $\rho>r$. Therefore, $\zeta_{1-}:=\theta_\rho$ is 
a root of a polynomial $\mcr_{p,q,C_{b_1}}=(1-r+r\zeta)^q-C_{b_1}\zeta^p$, by Proposition \ref{pinters}. 
Hence, $0<C_{b_1}<1$, and the same polynomial $\mcr_{p,q,C_{b_1}}$ has a unique root $\zeta_{1+}\in(1,+\infty)$, 
due to the following proposition.
\begin{proposition} \label{proot}  Let $p,q\in\nn$, $1\leq q<p$, $r=\frac pq$. The following statements are equivalent:

1) The polynomial $\mcr_{p,q,C}$ has a real root in the interval $(\theta_r,1)$.

2) It has a real root greater than 1.

3) $0<C<1$. 

In this case the  above roots  are unique, and the correspondence between them for all $C\in(0,1)$ is 
a decreasing homeomorphism $(\theta_r,1)\to(1,+\infty)$.
\end{proposition}
\begin{proof} The complex roots of the polynomial $\mcr_{p,q,C}(\zeta)$ are $q$-th powers of roots of a 
polynomial 
$$H_{p,q,c}(\theta)=c\theta^p-r\theta^q+r-1, \ \ \ c^q=C,$$
 since  $\mcr_{p,q,c}(\theta^q)=\prod_{j=0}^{q-1}((1-r+r\theta^q)-ce^{\frac{2\pi j}q}\theta^p)$. 
 The statement of Proposition \ref{proot} for $q$-th powers of roots of the polynomial $H_{p,q,c}$ is 
given by  \cite[proposition 4.25]{gl2}, and it implies Proposition \ref{proot}.
\end{proof}
 The root $\zeta_{1+}$ is the limit of the $\zeta$-coordinate of some 
point of intersection $b_1\cap L_P$ (Proposition \ref{pinters}). Hence,  its $\er$-image, 
which will be denoted by $\zeta_{2-}$, is the limit of the $\zeta$-coordinate of an intersection point of the line $L_P$ 
with a germ $b_2\subset\Gamma$ based at $O$. One has $\theta_r<\zeta_{1-}<\zeta_{2-}<1$, by monotonicity of the 
map $\er|_\rr$. Again $\zeta_{2-}$ is a root of a polynomial $\mcr_{p,q,C_{b_2}}$, $C_{b_2}>1$, 
and the latter polynomial has another root $\zeta_{2+}\in(1,\zeta_{1+})$, as in \cite[proof of theorem 4.24]{gl2}. 
Continuing this procedure we get an infinite decreasing sequence of roots $\zeta_{j+}$, 
all of them being  limits of $\zeta$-coordinates of points of the intersection $\Gamma\cap L_P$. 
Hence, the cardinality of the latter intersection is unbounded, as $P\to O$, while 
the intersection index of the multigerm $\Gamma$ with $L_O$ is finite. 
The contradiction thus obtained proves the proposition.
\end{proof}

\begin{proof} {\bf of Theorem \ref{lalt}.} Quadraticity of the germ $b$ follows from Theorems \ref{quasi-pqr} and 
\ref{classpqr}. If $b$ is regular, then there is nothing to prove. Let $b$ be singular. Let 
$\rho$ denote the residue at $O$ of the dual   billiard structure. Then $b$ is uniformly 
$\delta$-folded for some $\delta>0$, and $\rho=2+\delta>2$, by Theorem \ref{tufold}. Therefore, 
the meromorphic integral $R$ is rational and $b$ lies in an algebraic curve, by Proposition \ref{merrat}. 

Suppose the contrary to the constance statement: $R\not\equiv const$ along the  line $L_O$ tangent to $b$ at $O$, i.e.,, the $z$-axis. 
Fix a point $A\in L_O\setminus\{ O\}$ with finite $z$-coordinate $z_1=z(A)$ that is not an indeterminacy 
point for the integral $R$. 
For every $P\in b\setminus\{ O\}$ the intersection of the line 
$L_P$ with the level curve $\Gamma:=\{ R=R(A)\}$ is $\sigma_P$-invariant. The $(L_O,O)$-localization 
of the algebraic curve $\Gamma$ is a $(L_O,O)$-local multigerm satisfying the conditions of 
Proposition \ref{pmgerm}, by construction. Hence, all its curves are based at one point $O$, by  Proposition \ref{pmgerm}. On the other hand, it contains a germ of analytic curve based at the point $A$, 
by construction. The contradiction thus obtained proves 
that $R|_{L_O}\equiv const$. 

Suppose now the contrary to the last statement of Theorem \ref{lalt}: the punctured line $L_O\setminus\{ O\}$ 
contains a singular point $A$ for the foliation $R=const$. Without loss of generality we consider that $R|_{L_O}\equiv0$. 
Let us consider the germ of the integral $R$ at $A$ 
and write it as the product $w^kf(z,w)$ with $k\in\nn$; $f(z,w)$ being a germ of meromorphic function 
with $f(z,0)\not\equiv0,\infty$. The point $A$ is  singular for the foliation, if and only if at least one of the two 
following statements holds: either $A$ is an indeterminacy point for the function $f$, or $A$ is 
its pole (zero). In both cases at least one of the level curves $\{ R=0\}$ or $\{ R=\infty\}$ contains a local branch 
$a$ based at $A$ that does not lie in the $z$-axis $L_O$.  Let us denote the latter level curve by $\Gamma$. 
Its $(L_O,O)$-localization is a multigerm satisfying the conditions, and hence, the statement of Proposition 
\ref{pmgerm}. Therefore, it consists of germs of curves based at the unique point $O$, while, by assumption, 
some of its germs has base point $A\neq O$. The contradiction thus obtained proves that 
$L_O\setminus\{ O\}$ is a regular leaf of the foliation $R=const$ and proves Theorem \ref{lalt}.
\end{proof}
 
\subsection{Uniqueness of singular point with singular branch. Proof of Theorem \ref{typesing}}
Here we prove Theorem \ref{typesing}. Quadraticity of local branches is already proved (Theorem \ref{lalt}). 
Let us prove uniqueness of point $O\in\gamma$ at which some local branch of the curve $\gamma$ 
is singular. Suppose the contrary: there exist at least two distinct points $O_1, O_2\in\gamma$ with singular local 
branches $b_1$ and $b_2$ respectively. Let $L^1$, $L^2$ denote their projective tangent lines 
at $O_1$, $O_2$. The rational integral is constant along both lines $L^1$ and $L^2$, 
by Theorem \ref{lalt}. One has $L^1\neq L^2$. Indeed, if  $L^1=L^2$, then the punctured line $L^1\setminus\{ O_1\}$ would contain a 
singular point $O_2$ of foliation by level curves of the integral, which is forbidden by Theorem \ref{lalt}. 
Thus,  $L^1$ and $L^2$ intersect at some point $A$ distinct from some of the points $O_j$, say, $O_1$. 
But then the punctured line $L^1\setminus\{ O_1\}$ contains a singular point $A$ of foliation by level curves, -- 
a contradiction to Theorem \ref{lalt}. Theorem \ref{typesing} is proved.

\section{Plane curve invariants. Proof of Theorem \ref{type-conic}}
Here we prove Theorem \ref{type-conic} stating that every irreducible algebraic curve $\gamma\subset\cp^2$ 
satisfying the statements of Theorem \ref{typesing} is a conic. The proof given in Subsection 6.2 is based 
on B\'ezout Theorem applied to the intersection of the curve $\gamma$ with its Hessian curve and 
 Shustin's  formula \cite{sh} for Hessians of  singular points.  The corresponding background material is recalled 
 in Subsection 6.1. 
\subsection{Invariants of plane curve singularities}
Hereby we recall the material from \cite[Chapter III]{BK}, \cite[\S10]{Mi}, \cite{sh}, see also a modern exposition in \cite[Section I.3]{GLS}. This material in a brief form needed here is presented in \cite[subsection 4.1]{gs}. 

Let $\gamma\subset\cp^2$ be a non-linear irreducible 
 algebraic curve.   Let $d$ denote its degree. Let $H_{\gamma}$ denote its Hessian curve: 
 the zero locus of the Hessian determinant of the defining homogeneous polynomial of $\gamma$. 
 It is an algebraic curve of  degree  $3(d-2)$.  
  The set of all singular and inflection points of the curve $\gamma$ 
 coincides with the intersection  $\gamma\cap H_{\gamma}$.   
The intersection index of these curves  is equal to $3d(d-2)$, by B\'ezout Theorem. 
On the other hand, it is equal to the sum of the contributions $h(\gamma,Q)$, which are called 
the {\it Hessians of the germs} $(\gamma,Q)$, through all 
the singular and inflection points $Q$ of the curve $\gamma$:
\begin{equation}3d(d-2)=\sum_{Q\in\gamma}h(\gamma,Q).\label{hesin}\end{equation} 
Let us recall an explicit formula for the Hessians $h(\gamma,Q)$  \cite[formula (2) and theorem 1]{sh}. To do this, let us introduce the following notations. For every local 
branch $b$ of the curve $\gamma$ at $Q$ let $s(b)$ denote its multiplicity: its intersection index with a generic line through $Q$. Let $s^*(b)$ denote the analogous multiplicity of the dual germ. Note that 
$s(b)=q_b$,  $s^*(b)=p_b-q_b,$
where $p_b$ and $q_b$ are the exponents in the parametrization $t\mapsto(t^{q_b},c_bt^{p_b}(1+o(1)))$ of the local branch 
$b$ in adapted 
coordinates. One has 
\begin{equation}s(b)=s^*(b)=q_b \text{ for every quadratic branch } b.\label{forquad}\end{equation}
Let 
$b_{Q1},\dots,b_{Qn(Q)}$ denote the local branches of the curve $\gamma$ at $Q$; here $n(Q)$ denotes their number. 
The above-mentioned formula for $h(\gamma,Q)$ 
from \cite{sh} has the form 
\begin{equation} h(\gamma,Q)=3\kappa(\gamma,Q)+\sum_{j=1}^{n(Q)}(s^*(b_{Qj})-s(b_{Qj})),\label{hessa}\end{equation} 
where $\kappa(\gamma,Q)$ is the $\kappa$-invariant, the class of the singular point. 
Namely, consider the germ of function $f$ defining the germ $(\gamma,Q)$; $(\gamma,Q)=\{ f=0\}$. Fix a line $L$ 
through $Q$ that is transversal to all the local branches of the curve $\gamma$ at $Q$. Fix a small ball  $U=U(Q)$ 
centered at $Q$ and consider a level curve $\gamma_{\var}=\{ f=\var\}\cap U$ with small $\var\neq0$, which is non-singular. 
The number 
$\kappa(Q)=\kappa(\gamma,Q)$ is the number of  points of the curve $\gamma_{\var}$ where its tangent line is parallel to $L$.  It is well-known that 
\begin{equation}\kappa(\gamma,Q)=2\delta(\gamma,Q)+\sum_{j=1}^{n(Q)}(s(b_{Qj})-1),\label{kap}\end{equation}
see, for example, \cite[propositions I.3.35 and I.3.38]{GLS}, where 
 $\delta(\gamma,Q)=\delta(Q)$ is the $\delta$-invariant. Namely, consider the curve $\gamma_{\var}$, which is a Riemann surface whose  boundary is a finite collection of 
closed curves: their number equals to $n(Q)$. Let us take the 2-sphere 
with $n(Q)$  deleted disks. Let us paste it to $\gamma_{\var}$: this yields to a compact surface. By definition, 
its genus is the $\delta$-invariant $\delta(Q)$. 
One has $\delta(Q)\geq0$, and $\delta(Q)=0$ whenever  $Q$ is a non-singular point.   Hironaka's genus formula \cite{hir} implies that 
\begin{equation}\sum_{Q\in\Sing(\gamma)}\delta(\gamma,Q)\leq\frac{(d-1)(d-2)}2.\label{d<d}\end{equation}
Formulas (\ref{hesin}), (\ref{hessa}) and  (\ref{kap}) together imply the formula 
\begin{equation}3d(d-2)=6\sum_Q\delta(\gamma,Q)+3\sum_Q\sum_{j=1}^{n(Q)}(s(b_{Qj})-1)\label{3dd-2}\end{equation}
$$+\sum_Q\sum_{j=1}^{n(Q)}(s^*(b_{Qj})-s(b_{Qj})).$$

\subsection{Proof of Theorem \ref{type-conic}}
All the local branches of the curve $\gamma$ are quadratic.  All of them are regular, except maybe for some branches at a unique singular point $O$ (if any). Therefore, the  third sum in the right-hand side in (\ref{3dd-2}) 
vanishes. All the terms in the second sum vanish except for those corresponding to the singular branches based  
at the point $O$. The first sum is no greater than $\frac{(d-1)(d-2)}2$, by (\ref{d<d}). Therefore, 
\begin{equation}3d(d-2)\leq 3(d-1)(d-2)+\sum_{j=1}^{n(O)}(s(b_{Oj})-1).\label{soj1}\end{equation}
If all the local branches at $O$ are regular, then the latter sum vanishes, and we get 
$3d(d-2)\leq 3(d-1)(d-2)$, hence $d=2$. Let now there exist at least one singular branch, say $b_{Ot}$: $s(b_{Ot})\geq2$. 
The intersection index of the curve $\gamma$ with a line through $O$ 
tangent to $b_{Ot}$  is no less than $2s(b_{Ot})+\sum_{j\neq t}s(b_{Oj})$. The latter 
intersection index should be no greater than $d$, by B\'ezout Theorem. Therefore, 
$$2s(b_{Ot})+\sum_{j\neq t}s(b_{Oj})\leq d, \ \sum_{j=1}^{n(O)}(s(b_{Oj})-1)<d-2,$$
$$3d(d-2)<3(d-1)(d-2)+d-2=3d(d-2).$$
The contradiction thus obtained proves Theorem \ref{type-conic}.

 \section{Classification of complex rationally integrable dual   billiards. Proof of Theorem \ref{tcompl}}
Let $\gamma\subset\cp^2$ be a non-linear irreducible algebraic curve equipped with a rationally integrable singular 
 dual   billiard structure. The curve $\gamma$ is a conic, by Theorems \ref{typesing} and \ref{type-conic}. 
 Thus, for the proof of Theorem \ref{tcompl} it suffices to  classify rationally integrable singular dual   billiard 
 structures on the conic 
 $$\gamma=\{ wt=z^2\}\subset\cp^2_{[z:w:t]}.$$
 To do this, we first classify the a priori possible residue configurations of the corresponding involution family. 
 In Subsection 7.1 we show that the billiard structure in question may have at most four singularities, 
 the corresponding residues lie in $\mcm\setminus\{0\}$ and their sum is equal to 4. This implies that 
 the a priori possible residue configurations are $4$, $(1,1,1,1)$, $(2,1,1)$, $(\rho,4-\rho)$ with $\rho\in\mcm\setminus\{0\}$,  $(\frac43,\frac43,\frac43)$, $(\frac32,\frac32,1)$, $(\frac43, \frac53, 1)$. We prove that 
 each residue configuration is realized by a rationally integrable dual  billiard, and we find the corresponding integrals. 
 We show that the cases of integer residues  correspond to the dual billiard structures of conical pencil type. 
 
\subsection{Residues of singular dual   billiard structures on conic}
\begin{proposition} \label{propsum} 
Let $\gamma\subset\cp^2$ be a regular conic  equipped with a singular holomorphic 
dual   billiard structure with isolated singularities that are its poles of order at most one. 
Then the sum of their residues  is equal to 4.
\end{proposition}
\begin{proof} Let us take an affine chart $\cc^2_{z,w}=\{ t=1\}\subset\cp^2_{[z:w:t]}$ in which $\gamma\cap\cc^2=\{ w=z^2\}$. 
For the $(2,1;2)$-billiard on the latter conic the sum of residues is equal to four: 
  the residues at $0$, $\infty$ are both equal to 2, since the projective symmetry 
  $(z,w)\mapsto(\frac zw, \frac1w)$ of the billiard structure (see the addendum to Theorem \ref{classpqr}) 
 permutes $0$ and $\infty$ and preserves the residues. To treat the general case, we use the following 
 proposition.
 \begin{proposition} \label{proinfi} Let $\gamma$ and the affine chart $\cc^2_{z,w}$ be as above. A  singular holomorphic dual billiard structure 
 on $\gamma$ is meromorphic (i.e., has  order at most one at each singular point), if and only if 
 the corresponding involutions $\sigma_P:L_P\to L_P$, $P=(z_0,z_0^2)\in\cc^2$,  written in the coordinate 
 $u=z-z_0$ on $L_P$, have the form 
$$\sigma_P:u\mapsto-\frac u{1+f(z_0)u}, \ \ f(z) \text{ is a rational function with simple poles,}$$
\begin{equation} f(z)=\frac1z(\la+o(1)), \ \text{ as } z\to\infty; \ \ \la\in\cc.
 \label{resinf}\end{equation}
 The residue at infinity is equal to $4-\la$. The above dual billiard structure has regular point at infinity, if and only if $\la=4$. 
 \end{proposition}
 \begin{proof} A finite singular point of $\sigma_P$ is of order  one, if and only if the corresponding function $f(z)$ has 
 simple pole there (Proposition \ref{propress}).  Let $E=[0:1:0]$ be the infinite point of the conic $\gamma$. 
 Consider the affine coordinates $(\wt z, \wt w)=(\frac zw, \frac1w)$ centered at $E$. 
Let $\rho$ denote the residue at $E$: then in the coordinate $\wt\zeta:=\frac{\wt z}{\wt z(P)}$ on $L_P$ 
 the involution $\sigma_P$ converges to $\eta_\rho(\wt\zeta)$, as $P\to E$. 
 Therefore, in the coordinate $\zeta$ the involution $\sigma_P$ converges to $\eta_{4-\rho}$, see statement (\ref{eta2rho}) and discussion before it. 
 Hence, in the coordinate $\hat u:=\zeta-1$ the involution $\sigma_P$ takes the form $\hat u\mapsto-\frac{\hat u}{1+g(z_0)u}$, $g(z)$ 
 is a rational function, $g(z)\to4-\rho$, as $z\to\infty$. Rescaling to the coordinate $u=z_0\hat u$ yields (\ref{resinf}) with $f(z)=\frac{g(z)}z$, 
 $\la=4-\rho$. The converse is proved by converse argument.  The last statement of Proposition \ref{proinfi} (regularity at $E$) 
 follows from Proposition \ref{propress}. Proposition \ref{proinfi} is proved. 
 \end{proof} 
 
 Consider now an arbitrary dual billiard structure on a conic whose singularities are of order at most one. Let us choose an 
affine chart $\cc^2_{z,w}$ in which $\gamma\cap\cc^2=\{ w=z^2\}$ and so that the above point $E$ at infinity be regular for the 
dual billiard structure. Then the corresponding function $f(z)$ from (\ref{resinf})  is rational with  simple poles, let us denote 
them $a_j$ (Proposition \ref{propress}). Hence, $f(z)=\sum_j\frac{\la_j}{z-a_j}$, $\la_j$ being residues, and thus,  
$\sum_j\la_j=4$, by   the last statement of Proposition \ref{proinfi}. This proves Proposition \ref{propsum}. 
\end{proof}

\begin{proposition} \label{props2} Let $\gamma$ be a regular conic equipped with a rationally integrable 
singular dual   billiard structure. Then each singular point of the structure is its pole of order 1, and 
its residue lies in   $\mcm\setminus\{0\}$.   
 \end{proposition} 
 \begin{proof} Well-definedness of residues follows from integrability and Proposition \ref{prhod}. 
 Their non-vanishing follows from  Proposition \ref{propress}. Let $O$ be a  singular point, and let $\rho$ be 
 the corresponding residue. Then the $(2,1;\rho)$-billiard 
 is quasihomogeneously integrable, by Theorem \ref{quasi-pqr}. 
 Therefore, $\rho\in\mcm$, by Theorem \ref{classpqr}.  Proposition \ref{props2} is proved.
 \end{proof}
 \begin{corollary} \label{cconfig} Let $\gamma$ be a regular conic equipped with a rationally integrable 
singular dual   billiard structure. Then it has at least one and most four singular points, with 
residue collections being of one of the following types:
\begin{equation}4, \ (2,2), \ (1,3), \ (2,1,1), \ (1,1,1,1); \  \ \ (\rho,4-\rho) \text{ with } \rho\in\mcm\setminus\zz;\label{rescoll}\end{equation}
$$ \left(\frac32,\frac32,1\right), \ \left(\frac43,\frac43,\frac43\right), \ \left(\frac43, \frac53,1\right).$$
\end{corollary}
\begin{proof} The residues lie in $\mcm\setminus\{0\}$ (Proposition \ref{props2}), and hence, are greater or 
equal to one.  Their sum  is equal to 4 
(Proposition \ref{propsum}). Therefore, the number of singularities is between one and four. 
The cases of one and two singularities are obviously given by the first, second, third and sixth 
collections in (\ref{rescoll}). The case of three  singularities with natural residues is the collection $(2,1,1)$. 
The case of four singularities is $(1,1,1,1)$. Cases of three singularities with some of 
residues being non-integer correspond to the three last residue collections   in (\ref{rescoll}). Indeed, 
each non-integer number in $\mcm\setminus\{0\}$ takes the form   $2\pm\frac2k$, $k\in\nn_{\geq3}$. 
Therefore, if the number of singularities is three, then 
 non-integer residues are of the type $2-\frac2{k}$,  $k\geq3$. Finally, all possible configurations have one of the two 
 following types: 
 $(2-\frac2{k_1},2-\frac2{k_2},2-\frac2{k_3})$, $(2-\frac2{k_1},2-\frac2{k_2},1)$. For the first type, writing the condition that 
 the sum of residues is equal to 4  yields 
 $$\frac2{k_1}+\frac2{k_2}+\frac2{k_3}=2, \ \ \ \ \ k_1,k_2,k_3\geq3.$$
 Therefore, $k_1=k_2=k_3=3$, and we get the residue collection $(\frac43,\frac43,\frac43)$. 
 For the second type we get $\frac2{k_1}+\frac2{k_2}=1$, $k_1,k_2\in\nn_{\geq3}$. The only solutions of the 
 latter equation are $\{ k_1,k_2\}=\{4,4\}, \{3,6\}$, which correspond to the residue configurations $(\frac32,\frac32,1)$ and 
 $ (\frac43, \frac53,1)$ respectively.  The corollary is proved.
 \end{proof}
 
 \begin{proposition} \label{exun} 
 Let $\gamma\subset\cp^2$ be a regular conic. For any two collections of distinct points 
 $a_1,\dots,a_n\in\gamma$ and non-zero 
numbers $(x_1,\dots,x_n)$ with $\sum_{j=1}^nx_j=4$ there exists a unique  singular holomorphic 
 dual   billiard structure on $\gamma$ with singular points $a_j$ being poles of order  one with residues $x_j$.
 \end{proposition}
 \begin{proof} The proposition follows from (\ref{resinf}) and uniqueness of a rational function $f(z)$ vanishing 
 at infinity as $\frac\la{z}(1+o(1))$ with given $\la$ and  simple poles  with given positions and residues 
 (see the proof of Proposition \ref{propsum}). 
 \end{proof}  
 \subsection{Case of integer  residues: pencil of conics}
 
 \begin{proposition} \label{propencil} 
  Let a singular holomorphic dual billiard structure on  a regular conic $\gamma$ 
 have singularities $a_1,\dots,a_m$, $m\in\{1,2,3,4\}$, with residues $\la_j\in\nn$, $j=1,\dots,m$. 
 Then it is realized by the pencil of conics passing through $a_j$ and having contact with $\gamma$ of order $\la_j$  at $a_j$.
\end{proposition}
 
 \begin{proof}  Case 1): four distinct points $a_1,\dots,a_4$ with residues $\la_j=1$. Consider the pencil of conics passing through them.  
 It defines the projective involutions $\sigma_P: L_P\to L_P$, $P\in\gamma\setminus\{ a_1,\dots,a_4\}$, 
  permuting the intersection points of the lines $L_P$ with each conic of the pencil. This yields a singular holomorphic dual billiard structure on $\gamma$ 
  with singularities at $a_j$. 
 
 {\bf Claim 6.} {\it The  involution family $\sigma_P$ is holomorphic with singularities $a_j$ of order one 
  and residue one.}
  
  \begin{proof} For every $a_j$ each conic $C$ of the pencil, $C\neq\gamma$, intersects the line $L_{a_j}$ 
  transversally at two distinct points: $a_j$ and some point $b_j$. For every $P\in\gamma$ close to $a_j$ 
  the line $L_P$ intersects $C$ at two points $Q(P)$ and $Y(P)$ converging to $a_j$ and $b_j$ respectively, 
  as $P\to a_j$. They are permuted by the involution $\sigma_P$, 
  and their $\zeta$-coordinates tend to $0$ and $\infty$ respectively. Therefore, in the coordinate $\zeta$ 
  the involution $\sigma_P$ converges to $\eta_1(\zeta)=\frac1\zeta$. Hence, $\sigma_P$ has simple pole 
  with residue one at $a_j$. The claim is proved.
  \end{proof}
  
  Claim 6 together with Proposition \ref{exun} imply that the initial dual billiard coincides with the one defined 
  by the above pencil. 
  
  Case 2):  three singular points $a_1$, $a_2$, $a_3$ with residues $1$, $1$, $2$ respectively. 
  Consider the pencil of conics passing through these points and tangent to $\gamma$ at $a_3$. 
  The above  involution family $\sigma_P$  defined by this pencil  has a rational quadratic 
  integral (Example \ref{tabobs}).  Therefore, its singularities $a_1$, $a_2$, $a_3$ are poles  of order one, by Proposition \ref{prhod}. 
Its residues at $a_1$, $a_2$ are equal to 1, see Claim 6 and its proof. Therefore, its residue at the 
third point $a_3$ is equal to $4-2=2$ (Proposition \ref{propsum}). Hence, the initial dual billiard structure 
coincides with the one defined by the pencil, by Propositon \ref{exun}.

The remaining cases of residue configurations $(1,3)$, $4$ are treated analogously. 
 Proposition \ref{propencil} is proved.
  \end{proof}
  
  \subsection{Case of two singularities: a quasihomogeneously 
  integrable $(2,1;\rho)$-billiard}
  
  \begin{proposition} \label{pquas} Every singular holomorphic dual   billiard structure on a regular conic 
  with two singularities of order one is projectively equivalent to a $(2,1;\rho)$-billiard. 
  It is rationally integrable, if and only if the latter billiard is quasihomogeneously integrable; 
  this holds if and only if $\rho\in\mcm$. 
  \end{proposition}
  \begin{proof}
  The first statement of Proposition \ref{pquas}, with $\rho$ being the residue at some  
  singularity, follows from Propositions \ref{propsum} and  \ref{exun}. The condition that $\rho\in\mcm$ is 
   necessary for rational integrability, by Proposition \ref{props2}. Conversely, it $\rho\in\mcm$, then 
   the billiard, which is equivalent to the $(2,1;\rho)$-billiard, is rationally integrable, by Theorem \ref{classpqr}. 
   This proves Proposition \ref{pquas}.
   \end{proof}
  
   \subsection{Integrability of residue configuration $(\frac32,\frac32,1)$}
 
 Take an affine chart $(z,w)$ in which the conic $\gamma$ is given by the equation $w=z^2$,  
the singular points of the billiard structure with residue $\frac32$ are $(0,0)$ and the infinite point, and 
the singular point with residue $1$ is $(1,1)$. The corresponding involution family $\sigma_P:L_P\to L_P$ 
written in the coordinate $u=z-z_0$, $z_0:=z(P)$, takes the form (\ref{sigma2}), by Proposition \ref{propress}, \ref{proinfi}, \ref{exun}:
\begin{equation}u\mapsto -\frac u{1+f(z_0)u}, \ f(z)=\frac3{2z}+\frac1{z-1}=\frac{5z-3}{2z(z-1)}.\label{inv32}
\end{equation}
\begin{lemma} \label{int321} The billiard structure on $\gamma$ defined by the above involution family 
$\sigma_P$ admits 
the rational integral 
\begin{equation}R(z,w)=\frac{(w-z^2)^2}{(w+3z^2)(z-w)(z-1)}.\label{int01}\end{equation}
\end{lemma}

{\bf Motivation of construction of integral.} The singular point at the origin has residue $\frac32$. The corresponding 
$(2,1;\frac32)$-billiard has a quasihomogeneous integral $R_{\frac32}(z,w)=\frac{(w-z^2)^2}{z(w+3z^2)}$, see Theorem \ref{classpqr}.  
Let us try to construct a rational integral  of our non-quasihomogeneous dual billiard in the form $\frac{(w-z^2)^2}{Q(z,w)}$
 so that the lower $(2,1)$-quasihomogeneous  part at $(0,0)$ of the denominator $Q(z,w)$ be equal to the denominator in $R_{\frac32}$ up 
 to constant factor  
   (see the proof of Theorem \ref{quasi-pqr}). Then the zero locus $\{ Q=0\}$ should contain an irreducible quadratic germ of analytic curve at $(0,0)$ 
 having a contact bigger than two with the conic $C:=\{ w+3z^2=0\}$. Let us look for  a polynomial $Q$ of degree four vanishing on  
the conic $C$: $Q(z,w)=(w+3z^2)H(z,w)$. To find the zero locus of the polynomial $H$, we have to find the images of the points of  intersection 
 $L_P\cap C$ under the involutions $\sigma_P$. We show that the latter images lie in the union of lines $\{ w=z\}$ and $\{ z=1\}$. This 
 together with Proposition \ref{invzero} will imply that the function (\ref{int01}) is an integral. 

\begin{proof} {\bf of Lemma \ref{int321}.} 
Fix a $P=(z_0,z_0^2)\in\gamma$. The line $L_P$ intersects the conic $C$ at two points $A$ and $D$ 
 with coordinates 
$z=-z_0$ and $z=\frac13z_0$ respectively, since the $\zeta$-coordinates of the intersection points 
are roots of the polynomial $\mcr_{2,1,-3}(\zeta)=3\zeta^2+2\zeta-1$, see (\ref{prest}); 
its roots are $-1$ and $\frac13$. Let $B$ and $F$ denote respectively the 
intersection points of the line $L_P$ with the lines $\{ z=1\}$ and $\{w=z\}$  respectively.

{\bf Claim 7.} {\it One has } $\sigma_P(D)=B$, $\sigma_P(A)=F$. 

\begin{proof} 
The $u$-coordinates of the points $A$, $D$, $B$, $F$, $u=z-z_0$,   are 
\begin{equation}u(A)=-2z_0, \ u(D)=-\frac23z_0, \ u(B)=1-z_0, \ u(F)=\frac{z_0(1-z_0)}{2z_0-1}.\label{adef}
\end{equation} 
For $A$, $D$, $B$ the formulas are obvious. The $z$-coordinate of the point $F$ is found from the equation  on 
$F$ as the point of intersection of the lines $\{ w=z\}$ and $L_P=\{ w=2zz_0-z_0^2\}$: 
$$w(F)=2z_0z(F)-z_0^2=z(F), \ \ z(F)=\frac{z_0^2}{2z_0-1}.$$
The latter formula implies the last formula  in (\ref{adef}). One has 
$$u(\sigma_P(D))=-\frac{u(D)}{1+f(z_0)u(D)}=\frac{2z_0}{3(1+\frac{5z_0-3}{2z_0(z_0-1)}(-\frac23z_0))}$$
$$ =
-\frac{2z_0(z_0-1)}{2z_0}=1-z_0=u(B),$$
$$u(\sigma_P(A))=-\frac{u(A)}{1+f(z_0)u(A)}=\frac{2z_0}{1-\frac{5z_0-3}{2z_0(z_0-1)}2z_0}=
\frac{2z_0(z_0-1)}{2-4z_0}=u(F).$$
The claim is proved.
\end{proof}

The claim together with the above discussion implies the statement of Lemma \ref{int321}.
\end{proof}

  \subsection{Integrability of residue configuration $(\frac43,\frac43,\frac43)$}
  
  \begin{lemma} \label{lem43} The singular holomorphic dual   billiard structure on a regular conic with 
  three poles of order one and residues equal to $\frac43$ is rationally integrable. 
  In the affine chart $\cc^2_{z,w}$, where the conic is given by the equation $w=z^2$ and the singularities are 
  $$a_j:=(\var^j,\var^{2j}), \ \var=e^{\frac{2\pi i}3}, \ j=0,1,2,$$
     the corresponding involution family $\sigma_P:L_P\to L_P$ 
  written in the coordinate $u:=z-z_0$, $z_0=z(P)$, takes the form (\ref{sigma3}):
   \begin{equation}\sigma_P: u\mapsto-\frac u{1+f(z_0)u}, \ f(z)=\frac{4z^2}{z^3-1}.\label{sigma33}\end{equation}
The rational function given by (\ref{exo2b}): 
  \begin{equation}R(z,w)=\frac{(w-z^2)^3}{(1+w^3-2zw)^2},\label{exo2b2}\end{equation}
  is an integral of the billiard.
  \end{lemma}
 \begin{proof} Take an affine chart as above.
 The involution family $\sigma_P$ given by (\ref{sigma33}) has first order poles at $a_j$ with residues $\frac43$ 
  and is regular at infinity,  by Propositions \ref{propress} and \ref{proinfi}. 
  For the proof of invariance of the restrictions $R|_{L_P}$ under the involutions $\sigma_P$, 
 it suffices to show that the intersection of each line $L_P$ with the zero locus of the denominator,  the cubic  
 $$C:=\{1+w^3-2zw=0\},$$
 is $\sigma_P$-invariant. We will do this in the two following propositions.
 
 \begin{proposition} For every $P\in\gamma\setminus\{ a_0, a_1, a_2\}$ let  $S(P)\in L_P$ denote the fixed point 
 distinct from $P$  of the involution $\sigma_P$. The fixed point family $S(P)$ coincides   
with the triple punctured cubic  $C\setminus\{ a_0, a_1, a_2\}$. In particular, the cubic $C$ is rational.
\end{proposition}
\begin{proof} Solving the fixed point equation $u=-\frac u{1+f(z_0)u}$ in non-zero $u$ yields
$$u(S(P))=-\frac2{f(z_0)}=\frac{1-z_0^3}{2z_0^2}, \ z(S(P))=\frac{1+z_0^3}{2z_0^2},$$ 
$$w(S(P))=w(P)+2z_0u(S(P))=z_0^2+\frac{1-z_0^3}{z_0}=\frac{1}{z_0}.$$
Therefore, the fixed point family $S(P)$ runs along the parametrized rational curve 
\begin{equation}K:=\left(t\mapsto\left(\frac{1+t^3}{2t^2}, \frac1t\right), \ | \ t\in\oc\right).\label{parte}\end{equation}
The curve $K$ obviously satisfies the equation $1+w^3-2zw=0$ of the cubic $C$, and hence, coincides with $C$. 
\end{proof}

\begin{proposition} \label{pro3int}
 For every $P\in\gamma\setminus\{ a_0, a_1, a_2\}$ the intersection $L_P\cap C$ is 
$\sigma_P$-invariant.
\end{proposition}
\begin{proof} One of the points of the intersection $L_P\cap C$ is the fixed point $S(P)$. Let us show that the other 
intersection points are permuted by $\sigma_P$. To do this, let us find  explicitly their $t$-parameters, see 
(\ref{parte}). Along the line $L_P$ one has $w=z_0^2+2z_0(z-z_0)=2z_0z-z_0^2$. Substituting 
$z=\frac{1+t^3}{2t^2}$ and $w=\frac1t$ to the latter equation yields the equation 
$$(t-z_0)(t^2-\frac1{z_0})=0.$$
Its solution $t=z_0$ corresponds to the fixed point $S(P)$. The other  two solutions are $t=\pm\frac1{\sqrt{z_0}}$.  
Here we fix some value of square root and denote it $\sqrt{z_0}$; the other value is $-\sqrt{z_0}$.    The corresponding values $z$ and $u$ are equal respectively to 
$$z_{\pm}=\frac{z_0\pm\frac1{\sqrt{z_0}}}2, \ u_{\pm}=z_{\pm}-z_0=\frac{-z_0\pm\frac1{\sqrt{z_0}}}2.$$
The involution $\sigma_P$ sends the point with the $u$-coordinate  $u_+$ to the point with the $u$-coordinate 
$$-\frac{u_+}{1+f(z_0)u_+}=-\frac{\frac1{\sqrt{z_0}}-z_0}{2(1+2\frac{z_0^{\frac32}}{z_0^3-1}(1-z_0^{\frac32}))}.$$
Writing $z_0^3-1=(z_0^{\frac32}-1)(z_0^{\frac32}+1)$ in the denominator and cancelling the former factor 
yields 
$$-\frac{u_+}{1+f(z_0)u_+}=-\frac{(1-z_0^{\frac32})(1+z_0^{\frac32})}{2\sqrt{z_0}(1+z_0^{\frac32}-2z_0^{\frac32})}
=-\frac{z_0+\frac1{\sqrt{z_0}}}2=u_-.$$
This implies that the involution $\sigma_P$ permutes the intersection points with $u$-coordinates $u_{\pm}$. 
The proposition is proved.
\end{proof} 

Lemma \ref{lem43} follows from Propositions \ref{pro3int} and \ref{invzero}. 
\end{proof}

\subsection{Integrability of the configuration $(\frac43, \frac53, 1)$. End of proof of Theorem \ref{tcompl}}

\begin{lemma} \label{l362} The singular holomorphic dual billiard structure on a regular conic $\gamma$ with three singularities of order one and 
residues $\frac43$, $\frac53$, $1$ is rationally integrable. In the affine chart $\cc^2_{z,w}$ where $\gamma=\{ w=z^2\}$ 
and the corresponding singularities are $(0,0)$, infinity and $(1,1)$ respectively the involutions $\sigma_P:L_P\to L_P$ 
defining the dual billiard structure have the following form in the coordinate $u=z-z_0$, $z_0=z(P)$: 
\begin{equation}\sigma_P:u\mapsto-\frac u{1+f(z_0)u}, \ \ f(z)=\frac4{3z}+\frac1{z-1}=\frac{7z-4}{3z(z-1)}.\label{43fz}\end{equation}
The function 
\begin{equation} R(z,w)=\frac{(w-z^2)^3}{(w+8z^2)(z-1)(w+8z^2+4w^2+5z^2w-14zw-4z^3)}\label{int43}\end{equation}
is an integral of the dual billiard.
\end{lemma}
\begin{proof} Formula (\ref{43fz}) follows  from  Propositions \ref{proinfi} and \ref{exun}. 

{\bf Motivation of the construction of the integral $R$.} The residue of the dual billiard at $(0,0)$ is equal to $\frac43$. 
The  corresponding $(2,1;\frac43)$-billiard has quasihomogeneous integral $R_{\frac43}(z,w)=\frac{(w-z^2)^3}{(w+8z^2)^2}$, 
by Theorem \ref{classpqr}. 
 Let $E$ denote the infinity point of the conic $\gamma$. Its residue is equal to $\frac53$. In the coordinates 
 $(\wt z,\wt w)=(\frac zw, \frac1w)$  the corresponding $(2,1;\frac53)$-billiard 
 has quasihomogeneous integral $R_{\frac53}(\wt z,\wt w)=\frac{(\wt w-\wt z^2)^3}{(\wt w+8\wt z^2)(\wt w+\frac54\wt z^2)}$ 
 (Theorem \ref{classpqr}). Note that the denominators in both $R_{\frac43}$ and $R_{\frac53}$ vanish on the same conic 
 $$C:=\{ w=-8z^2\}=\{\wt w=-8\wt z^2\}.$$ 
We would like to construct an integral of the billiard from the lemma as a ratio 
$R(z,w)=\frac{(w-z^2)^3}{(w+8z^2)Y(z,w)}$. To find the polynomial $Y$, we find the images 
of points of the intersection  $L_P\cap C$ under the involution $\sigma_P$. We show that their families parametrized by $P$ 
form the union of the line $\{ z=1\}$ and a cubic. The latter union will be the zero locus of the polynomial $Y$. 

\begin{proposition} For every $P=(z_0,z_0^2)\in\gamma\setminus\{(0,0),(1,1), E\}$ the intersection $L_P\cap C$ consists of two points 
$A=A(P)$ and $D=D(P)$: $z(A)=\frac{z_0}4$, $z(D)=-\frac{z_0}2$.  One has  
\begin{equation}z(\sigma_P(A))=1, \ \sigma_P(D)=\left(-\frac{z_0(2z_0+1)}{2-5z_0}, \frac{z_0^2(z_0-4)}{2-5z_0}\right).\label{parcurve}\end{equation}
\end{proposition}
\begin{proof} The $\zeta$-coordinates of points of the intersection $L_P\cap C$, $\zeta=\frac z{z_0}$, are roots of the polynomial 
$\mcr_{2,1,-8}(\zeta)=8\zeta^2+2\zeta-1$. Its roots are $\frac14$ and $-\frac12$, and the corresponding intersection points will 
be denoted by $A$ and $D$ respectively. This proves the first statement of the proposition. Let us find their $\sigma_P$-images in the 
coordinate $u=z-z_0$. One has 
$$u(A)=-\frac{3z_0}4, \ u(\sigma_P(A))=-\frac{u(A)}{1+f(z_0)u(A)}=\frac{\frac{3z_0}4}{1-\frac{7z_0-4}{3z_0(z_0-1)}\frac{3z_0}4}=1-z_0,$$ 
$$z(\sigma_P(A))=1, \ \ \ \ \ u(D)=-\frac{3z_0}2,$$ 
$$u(\sigma_P(D))=-\frac{u(D)}{1+f(z_0)u(D)}=\frac{\frac{3z_0}2}{1-\frac{7z_0-4}{3z_0(z_0-1)}\frac{3z_0}2}
=\frac{3z_0(z_0-1)}{2-5z_0},$$ 
$$z(\sigma_P(D))=-\frac{z_0(2z_0+1)}{2-5z_0}, \ 
w(\sigma_P(D))=2z_0z(\sigma_P(D))-z_0^2=\frac{z_0^2(z_0-4)}{2-5z_0}.$$
Proposition \ref{parcurve} is proved.
\end{proof}
\begin{corollary} The families  of images $\sigma_P(A(P))$ and $\sigma_P(D(P))$ are respectively the line $\{ z=1\}$ and  the  rational cubic 
 \begin{equation}S:=\{\left(-\frac{t(2t+1)}{2-5t}, \frac{t^2(t-4)}{2-5t}\right) \ | \ t\in\oc\}\label{defcos}\end{equation}
\end{corollary}
 \begin{proposition} \label{procub} The cubic $S$ is the zero locus of the polynomial
 \begin{equation} K(z,w)=w+8z^2+4w^2+5z^2w-14zw-4z^3\label{s=k=0}\end{equation}
 \end{proposition}
 \begin{proof} One can prove the proposition directly by substituting the parametrization (\ref{defcos}) to (\ref{s=k=0}). But we will 
 give a geometric proof explaining how formula (\ref{s=k=0}) was found. First let us show that 
 \begin{equation} \gamma\cap S=\Sigma:=\{(0,0), (1,1), E\},\label{gsig}\end{equation}
 $$t=0,1,\infty \ \text{ at } (0,0),  (1,1),  E \ \text{ respectively.}$$
 Indeed, for every $P\in\gamma\setminus\Sigma$ the line $L_P$ intersects $\gamma$ 
 only at $P$, and its intersection points $A$ and $D$ with $C$ do not coincide with $P$, since $\gamma$ 
 and $C$ intersect only  at two points: $(0,0)$ and $E$.   Therefore, $\sigma_P(D)\in L_P\setminus\{ P\}$ lies outside $\gamma$. 
On the other hand, as $P$ tends to a point $X\in\Sigma$, one has $\sigma_D(P)\to X$. Indeed, this holds exactly when 
$z_0=z(P)$ tends to some of the points $0$, $1$ or $\infty$, and in this case $\sigma_D(P)\to X$: both latter statements follow from 
(\ref{parcurve}). This proves (\ref{gsig}). 

{\bf Claim 8.} {\it The germs of the  curve $S$ at $(0,0)$ and $E$ are regular and tangent to the conic $C$ and to the conic 
$\{ w=-\frac54z^2\}$ respectively with contact of order at least three. The curve $S$ is bijectively parametrized by the parameter $t$, see (\ref{defcos}), except maybe for possible self-intersections.}

\begin{proof} The coordinates of a point of the curve $S$ with a parameter $t\to0$ are asymptotic to $-\frac t2$ and $-2t^2$ respectively. 
This implies the statement of the claim for the germ at $(0,0)$. The proof for the germ at infinity is analogous. The parametrization 
(\ref{defcos}) is either bijective (up to self-intersections), or a covering of degree at least two. The latter case is clearly impossible, since 
the germ of the curve $S$ at $(0,0)$ is injectively parametrized by a neighborhood of the point $t_0=0$ and no other parameter value is sent 
to $(0,0)$. 
\end{proof} 

{\bf Claim 9.} {\it The germ of the curve $S$ at $(1,1)$ is a cusp.} 

\begin{proof} The germ of the curve $S$ at $(1,1)$ is irreducible, since it is a germ of curve parametrized by the parameter $t$ at the base point $t_0=1$ in the parameter line, see (\ref{gsig}). It is a singular germ, since the derivative of the map (\ref{defcos}) at $t=1$ is zero and by the last statement 
of Claim 8. Therefore, the projective line $L=L_{(1,1)}$ tangent to $S$ at the point $(1,1)$ is tangent to $S$ with contact at least three. 
On the other hand, the tangency order cannot be bigger than three, since $S$ is a cubic. Hence, it is equal to three, and $(1,1)$ is a cusp. 
The claim is proved.
\end{proof}

{\bf Claim 10.} {\it There exists a cubic polynomial vanishing on $S$ of the form}
\begin{equation}K(z,w)=w+8z^2+\alpha w(w+\frac54z^2)+\beta zw+\psi z^3.\label{kprep}\end{equation}

\begin{proof} Let $K(z,w)$ be a cubic polynomial vanishing on $S$. Its homogeneous cubic part should contain no $w^3$ and $w^2z$ 
 terms, since $S$ contains the point $E=[0:1:0]\in\cp^2$ and is tangent to the infinity line there. Therefore, it 
 is a linear combination of monomials from (\ref{kprep}) (and may be $z$).  Its lower $(2,1)$-quasihomogeneous  
part at $(0,0)$ is $w+8z^2$ up to constant factor, since the germ of the curve $S$  at the origin is regular and tangent to the conic $C=\{ w+8z^2=0\}$ 
with contact of order at least three (Claim 8). Hence normalizing $K$ by constant factor, we can and will consider 
that  $K$ is equal to $w+8z^2$ plus a linear combination of monomials  $w^2$, $wz$, $wz^2$, $z^3$. Passing to the affine coordinates $(\wt z,\wt w)=(\frac zw,\frac1w)$ centered at $E$ 
we get that $K(z,w)$ is equal to $\frac1{\wt w^3}H(\wt z,\wt w)$, where $H(\wt z,\wt w)$ is a polynomial vanishing 
on the germ of the curve $S$ at $E$. The latter germ being regular and tangent to the conic $\{ w+\frac54z^2=0\}=\{\wt w+\frac54\wt z^2=0\}$ with 
contact of order at least three (Claim 8), the $(2,1)$-quasihomogeneous part of the polynomial $H(\wt z,\wt w)$ is equal to 
$\wt w+\frac54\wt z^2$ up to constant factor; hence $H(\wt z,\wt w)$ is equal to $\alpha(\wt w+\frac54\wt z^2)+\psi\wt z^3$ plus 
a polynomial of degree at most three whose monomials are divisible by either $\wt w^2$, or  $\wt w\wt z$. One has 
$$\frac1{\wt w^3}(\wt w+\frac54\wt z^2)=w(w+\frac54z^2), \ \frac{\wt z^3}{\wt w^3}=z^3, \ \frac1{\wt w^3}\wt w\wt z=zw.$$
This together with the above discussion implies that the polynomial $K$ has the type (\ref{kprep}). The claim is proved.
\end{proof}

Finding  unknown coefficients $\alpha$, $\beta$, $\psi$ from the linear equation saying that $K(z,w)$ vanishes at $(1,1)$  with its 
first derivatives yields (\ref{s=k=0}). One can also check directly that the polynomial $K$ given by (\ref{s=k=0}) vanishes at $(1,1)$  with 
its  first partial derivatives. 
Proposition \ref{procub} is proved.
\end{proof}
\begin{proposition} \label{b1b2} Set 
$$\Gamma:=C\cup\{ z=1\}\cup S, \ \ C=\{ w+8z^2=0\}, \ S \text{ is given by (\ref{defcos}).}$$ 
For every $P\in\gamma\setminus\Sigma$ the intersection  $L_P\cap \Gamma$  is $\sigma_P$-invariant.
\end{proposition}
\begin{proof} For every $P\in\gamma\setminus\Sigma$ the line $L_P$ intersects $\Gamma$ at six points: $A,D\in C\cap L_P$, 
the point $\sigma_P(A)\in\{ z=1\}\cap L_P$, the point $\sigma_P(D)\in S\cap L_P$ and two more points $B_1, B_2\in S\cap L_P$; 
$B_j=B_j(P)$. It suffices to show that 
\begin{equation} \text{the involution } \sigma_P \text{ permutes  } B_1 \text{ and } B_2 \text{ for every } P\in\gamma\setminus\Sigma.
\label{invperm}\end{equation}
The proof of (\ref{invperm}) will be split into the two following claims. Set 
$$u_j=u(B_j(P))=z(B_j(P))-z(P)=z(B_j(P))-z_0, \ \ j=1,2.$$

{\bf Claim 11.} {\it One has} 
\begin{equation} u_1+u_2=\frac{4-7z_0}{2}, \ \ u_1u_2=\frac{3z_0(z_0-1)}2.\label{besum}\end{equation}

\begin{proof}  One has $z=z_0+u$, $w=z_0^2+2z_0u$ on $L_P$. In the coordinate $u$ on $L_P$ the restriction to $L_P$ of the polynomial 
$K(z,w)$  takes the form
\begin{equation}K|_{L_P}=(10z_0-4)u^3+(41z_0^2-40z_0+8)u^2+\chi u +9z_0^2(z_0-1)^2, \ \ \chi\in\cc.\label{320}\end{equation}
Indeed, the cubic term in $K|_{L_P}$ coincides with that of the sum 
\begin{equation}-4z^3+5z^2w=-4(u+z_0)^3+5(u+z_0)^2(z_0^2+2z_0u).\label{cubterm}\end{equation}
Thus, the coefficient at $u^3$ equals $10z_0-4$. The coefficient at $u^2$ in $K|_{L_P}$ is the sum of similar coefficients  in 
(\ref{cubterm}) and in the expression 
$$8z^2+4w^2-14zw=(8+16z_0^2-28z_0)u^2+\text{lower terms}.$$
The coefficient at $u^2$ in (\ref{cubterm}) is equal to $-12z_0+25z_0^2$. Hence, the coefficient at $u^2$ in $K|_{L_P}$ is equal to 
$41z_0^2-40z_0+8$. The free term of the polynomial $K|_{L_P}$ is equal to its value at the point $P=(z_0,z_0^2)$: 
$$K(z_0,z_0^2)=z_0^2+8z_0^2+4z_0^4+5z_0^4-14z_0^3-4z_0^3=9z_0^2(z_0-1)^2.$$
This proves (\ref{320}). The roots of the restriction $K|_{L_P}$ are $u_1$, $u_2$ and the $u$-coordinate $u_3$ of the point $\sigma_P(D)$: 
$$u_3:=z(\sigma_P(D))-z_0=-\frac{z_0(2z_0+1)}{2-5z_0}-z_0=\frac{3z_0(z_0-1)}{2-5z_0}.$$
Formula (\ref{320}) together with Vieta's formulas imply that 
$$u_1+u_2=\frac{41z_0^2-40z_0+8}{4-10z_0}-u_3=\frac{41z_0^2-40z_0+8-6z_0^2+6z_0}{4-10z_0}=\frac{4-7z_0}2,$$
$$u_1u_2=u_3^{-1}\frac{9z_0^2(z_0-1)^2}{4-10z_0}=\frac{3z_0(z_0-1)}2.$$
This proves (\ref{besum}). 
\end{proof}

Thus, by (\ref{besum}), the numbers $u_1$, $u_2$ are roots of the quadratic polynomial 
$$Q(u):=2u^2+(7z_0-4)u+3z_0(z_0-1).$$

{\bf Claim 12.} {\it One has}
\begin{equation}Q\circ\sigma_P(u)=\frac{Q(u)}{(1+f(z_0)u)^2}, \ f(z)=\frac{7z-4}{3z(z-1)}.\label{pulls}\end{equation}

\begin{proof} Recall that 
$\sigma_P(u)=-\frac u{1+f(z_0)u}.$ 
Therefore, 
$$Q\circ\sigma_P(u)=\frac{2u^2-(1+f(z_0)u)(7z_0-4)u+3z_0(z_0-1)(1+f(z_0)u)^2}{(1+f(z_0)u)^2}.$$
The numerator in the latter ratio is equal to 
$$(2-f(z_0)(7z_0-4)+3z_0(z_0-1)f^2(z_0))u^2$$
$$+(-(7z_0-4)+6z_0(z_0-1)f(z_0))u+3z_0(z_0-1)=Q(u).$$
This proves Claim 12.
\end{proof}

The involution $\sigma_P$ sends the collection  of roots of the polynomial $Q$ to itself: their images are zeros of 
the pullback $Q\circ\sigma_P$, which are roots of $Q$, by Claim 12.  Therefore, $\sigma_P$ permutes the roots: 
otherwise, it would fix three points, two  roots and   $0=u(P)$, which is impossible, since $\sigma_P\neq Id$. 
Thus, $\sigma_P$ permutes the points $B_1,B_2\in S\cap L_P$. This proves Proposition \ref{b1b2}.
\end{proof}

The zero locus of the rational function $R(z,w)$ given by (\ref{int43}) is the conic $\gamma$. Its 
polar locus  is the  curve $\Gamma$ from Proposition \ref{b1b2}. 
For every $P\in\gamma\setminus\Sigma$  the intersections of the latter loci 
 with $L_P$ are respectively the point $P$ and $\Gamma\cap L_P$.  They are $\sigma_P$-invariant, by  Proposition \ref{b1b2}. 
 Therefore, the function $R|_{L_P}$  is also $\sigma_P$-invariant, by Proposition \ref{invzero}. Hence, 
 $R$ is an integral of the dual billiard in question. This proves Lemma \ref{l362}.
 \end{proof}

\begin{proof} {\bf of Theorem \ref{tcompl}.} Let an irreducible germ of analytic curve $\gamma\subset\cp^2$ 
admit a structure of rationally integrable dual   billiard. Then the curve $\gamma$ is a conic, and 
the billiard structure extends to a singular holomorphic  one with poles of order at most one and residues lying in 
$\mcm\setminus\{0\}$, by  
 Proposition \ref{proalg},  Theorems \ref{typesing}, \ref{type-conic} and Proposition \ref{props2}. 
The sum of residues should be equal to four, by Proposition \ref{propsum}. All the  collections of residue values 
lying in $\mcm\setminus\{0\}$ with sum equal to 4 are described above. The corresponding billiard structures are rationally 
integrable with integrals given in this and previous subsections. 
This proves Theorem \ref{tcompl}. 
\end{proof}

\section{Real   integrable   dual billiards. Proof of Theorems  \ref{tgerm}, \ref{tclosed} and the addendums 
to Theorems \ref{tcompl}, \ref{tgerm}}

\subsection{Real germs: proof of Theorem \ref{tgerm} and the addendums to Theorems \ref{tcompl}, \ref{tgerm}}
Let a germ of real $C^4$-smooth curve $\gamma\subset\rr^2$ 
carry a rationally integrable dual   billiard structure. 
Then its complex Zariski closure $\overline\gamma\subset\cp^2_{[z:w:t]}\supset\cc^2_{z,w}=\{ t=1\}$ is an algebraic curve, 
and the billiard structure extends to a rationally integrable singular holomorphic dual   
billiard structure on every its non-linear irreducible component (Proposition \ref{proalg2}). Therefore, each non-linear irreducible component is a conic 
(Theorem \ref{tcompl}). Thus, 
the germ $\gamma$ is a chain of adjacent arcs of a finite collection of non-linear conics and maybe lines. It contains 
at least one conical arc, being non-linear. A conical arc cannot be adjacent to a straightline segment, since $\gamma$ is $C^2$-smooth. 
There are no adjacent arcs of distinct conics:  they would have contact of order at most 
4 (B\'ezout Theorem), and hence, would not paste together in $C^4$-smooth way, while $\gamma$ is $C^4$-smooth. This is the place where we use the condition that 
$\gamma$ is $C^4$-smooth.  Thus, $\gamma$ is a germ of real conic, 
which will be also denoted $\gamma$, and the complexified billiard structure on its complexification is one of those 
given by Theorem \ref{tcompl}. Let us find real forms of the complex dual billiards on conic given by Theorem \ref{tcompl}. 

Case 1):   The complexified dual billiard on the complexified  conic $\gamma$ is given by a complex pencil of conics $\mcc_\la$; $\gamma=\mcc_0$. 

\begin{proposition} \label{repens} The real dual billiard on the real conic 
$\gamma$ is defined by a real pencil of conics whose complexification is the pencil $\mcc_\la$.
\end{proposition}

{\bf Claim 13.} {\it The   pencil $\mcc_\la$ contains a complexified  real conic  $\mcc_{\la_0}\neq\gamma$.} 

\begin{proof} Fix a real point $P_0\in\gamma$ where the dual billiard involution $\sigma_{P_0}:L_{P_0}\to L_{P_0}$ is well-defined. Fix 
a point $E_1\in\rp^2\setminus\gamma$ close to $P_0$ and lying on the concave side from the conic $\gamma$. Consider the right real tangent line 
to $\gamma$ through $E_1$, let $P_1$ be its tangency point with $\gamma$. Set $E_2=\sigma_{P_1}(E_1)$. Take now the right real tangent line 
to $\gamma$ through $E_2$, let $P_2$ be the corresponding tangency point. Similarly we construct $E_3$, $E_4$, $E_5$. If $E_1$ is close 
enough to $P_0$, then the points $E_2,\dots, E_5$, $P_2,\dots,P_4$ are well-defined and close to $P_0$. The five real points $E_1,\dots, E_5$ 
lie in the same complex conic $\mcc_{\la_0}\neq\gamma$, since the complex dual billiard  is defined by the pencil $\mcc_\la$. 
The conic $\mcc_{\la_0}$ is the complexification of a real conic. Indeed, otherwise  $\mcc_{\la_0}$ and its complex conjugate conic would be distinct and would intersect at five distinct points $E_1,\dots,E_5$, which is  impossible. The claim is proved.
\end{proof}

\begin{proof} {\bf of Proposition \ref{repens}.} 
Let $\mcc_{\la_0}$ be the real conic from the claim. The complex pencil $\mcc_\la$ of complex conics is the 
complexification of the real pencil of real conics containing $\gamma$ and $\mcc_{\la_0}$, by construction and 
since a pencil of conics is uniquely determined by its two conics.  
Hence, the involution $\sigma_{P}:L_P\to L_P$ permutes 
the points of intersection of the line $L_P$ with each conic from the  real pencil,  since this is true for the pencil $\mcc_\la$. 
Thus, the real dual billiard on $\gamma$ is defined by a real pencil. This proves Proposition \ref{repens}.
\end{proof}

Case 2a): the billiard structure has two singular points with residues $2\pm\frac2k$, $k\in\nn_{\geq3}$. 
The coordinatewise complex conjugation cannot permute them, since it should preserve the residue. 
Therefore, both singular points lie in the real part of the conic, and the dual   billiard structure 
has the type 2a) from Theorem \ref{tgerm} and has the corresponding integral (\ref{exot1}) or (\ref{exot2}). 

Case 2b): the billiard structure has three singular points with residues $\frac32$, $\frac32$, $1$. 
Then the coordinatewise complex involution fixes the singular point with residue $1$ and may either fix, or permute 
the two other singular points. 

Subcase 2b1): the coordinatewise complex conjugation fixes all the three singular points. Then all of them lie 
in the real conic. Let us choose real homogeneous coordinates $[z:w:t]$  so that the singular points with residue 
$\frac32$ are $[0:0:1]$ and $[0:1:0]$, the singular point with residue $1$ is $[1:1:1]$  and the lines 
$w=0$, $z=0$ are tangent to $\gamma$ at the two former points. Then we get that  
in the affine chart $\cc^2_{z,w}=\{ t=1\}$ the real conic $\gamma$ is given by the equation $w=z^2$, 
the dual   billiard structure has type 2b1) from Theorem \ref{tgerm} and has integral (\ref{exo2bnew}) (Lemma \ref{int321}). 
The fact that $\gamma$ indeed coincides with the conic $\{wt=z^2\}\subset\cp^2_{[z:w:t]}$ follows from 
B\'ezout Theorem and the 
fact that the conics in question are tangent to each other at two points $[0:0:1]$ and $[0:1:0]$ and have 
yet another common point $[1:1:1]$. 

Subcase 2b2): the coordinatewise complex conjugation permutes the singular points with residue $\frac32$. 
Applying a real projective transformation, we can and will consider that the singular point with residue $1$ (which lies 
in the real conic) has coordinates $[0:1:0]$, the line $\{ t=0\}$ is tangent to  $\gamma$ at the latter point, 
 the line $\{ w=0\}$ is tangent to $\gamma$ at the point 
$[0:0:1]$,  the points with residue $\frac32$ have coordinates $[\pm i:-1:1]$. Then $\gamma$ is given by the 
equation $wt=z^2$, as in the above discussion. Passing to the affine chart $\cc^2_{z,w}=\{ t=1\}$ we get that the 
dual   billiard structure has type 2b2) from Theorem \ref{tgerm}. This argument implies that 
the dual   billiard structures 2b1) and 2b2) are complex-projective equivalent and proves 
the corresponding statement of the addendum to Theorem \ref{tcompl}.

Let us calculate the integral of the dual   billiard structure 2b2) in the above affine chart $(z,w)$.  
To do this, we find explicitly the projective equivalence $F:\cp^2\to\cp^2$ 
between the structures 2b1) and 2b2). It should send singular points $[0:0:1]$, $[0:1:0]$, $[1:1:1]$ of the 
structure 2b1) to the singular points of the structure 2b2) with the same residues. Let us construct an $F$ 
sending them to  $[i,-1:1]$, $[-i:-1:1]$, $[0:1:0]$. It should send the projective tangent lines to the conic at the points $[0:0:1]$, $[0:1:0]$ to its tangent lines at their images  $[i,-1:1]$, $[-i:-1:1]$. The two former tangent lines are the $z$-axis and 
the infinity line $\{ t=0\}$; their intersection point is $[1:0:0]$. The two latter tangent lines at $[i,-1:1]$, $[-i:-1:1]$  intersect at the point $(0,1)=[0:1:1]$, by symmetry and since in the affine chart $\cc^2_{z,w}=\{ t=1\}$ they pass 
through $(\pm i,-1)$ and have slopes $\pm2i$. Therefore, $F$ should also send 
$[1:0:0]$ to $[0:1:1]$. Finally, it sends $[1:0:0]$, $[0:1:0]$, $[0:0:1]$ to $[0:1:1]$, $[-i:-1:1]$, $[i,-1:1]$. 
The matrix of the  projective transformation $F$ is uniquely defined up  to scalar factor. 
Its columns are $\la_1(0,1,1)$, $\la_2(-i,-1,1)$,  $\la_3(i,-1,1)$, by the above statement; $\la_j\in\cc^*$. 
Let us choose the normalizing scalar factor so that $\la_1=1$. Then the  coefficients $\la_2$ and $\la_3$ 
are found from the linear equation saying that $F([1:1:1])=[0:1:0]$: $\la_2=\la_3=-\frac12$. We get that 
\begin{equation}F \text{ is given by } M:=\left(\begin{matrix} 0 & \frac i2 & -\frac i2\\
1 & \frac12 & \frac12\\ 1 & -\frac12 & -\frac12\end{matrix}\right), \ \ \ M^{-1}=\left(\begin{matrix} 0 & \frac12 & \frac12\\
-i & \frac12 & -\frac12\\ i & \frac12 & -\frac12\end{matrix}\right).\label{fmatr}\end{equation}
The  transformation $F$ thus constructed preserves the conic $\gamma=\{ wt=z^2\}$. Indeed,  
 its image is a conic tangent to $\gamma$ at two points $[\pm i:-1:1]$ and intersecting $\gamma$ in yet 
another point $[0:1:0]$ (by construction). Hence, it has intersection index at least 5 with $\gamma$ and thus, coincides with $\gamma$, 
by B\'ezout Theorem.  
The map $F$ sends the billiard structure 2b1) to 2b2), by construction. Let us check that it sends the integral $R_{b1}$ 
of the structure 2b1) to the integral $R_{b2}$ of the structure 2b2). Indeed, 
the integrals  written in the homogeneous coordinates $[z:w:t]$ take the form 
$$R_{b1}(z,w,t)=\frac{(wt-z^2)^2}{(wt+3z^2)(z-t)(z-w)},$$ 
$$R_{b2}(z,w,t)=\frac{(wt-z^2)^2}{(z^2+w^2+t^2+wt)(z^2+t^2)}.$$
The variable change given by the inverse matrix in (\ref{fmatr}), 
$$\left(\begin{matrix}\wt z \\ \wt w \\ \wt t\end{matrix}\right)=M^{-1}\left(\begin{matrix} z \\  w \\ t\end{matrix}\right),$$
transforms $R_{b1}(\wt z,\wt w,\wt t)$ to $R_{b2}(z,w,t)$. 

Case 2c): the complexified dual   billiard structure has three singularities with residues equal to $\frac43$. 
The coordinatewise complex conjugation permutes them. Hence, it should fix some of them, being an involution. 
Thus, it either fixes one singularity and permutes the two other ones (Subcase c1)), or fixes all the three singularities (Subcase c2)).

Subcase c1). The singular point fixed by coordinatewise complex conjugation is a real point of the 
conic $\gamma$, and the two other (permuted) singularities are complex-conjugated. Applying a  projective transformation 
with a real matrix, we can and will consider that  $\gamma=\{ wt=z^2\}$, the fixed singularity is the point 
$[1:1:1]$ and the permuted singularities are $[e^{\pm\frac{2\pi i}3}: e^{\mp\frac{2\pi i}3}:1]$. Then 
in the affine chart $\cc^2_{z,w}=\{ t=1\}$ the dual   billiard structure in question takes the form 2c1) 
as in Theorem \ref{tgerm}. Hence, it has first integral $R_{c1}$ given by (\ref{exo2b}), see  Lemma \ref{lem43}. 

Subcase c2). Then all the singularities of the billiard structure 
are real points in $\gamma$. Applying a  real projective transformation, we can and will consider that  $\gamma=\{ wt=z^2\}$ and the singularities are $[0:0:1]$, $[0:1:0]$ and $[1:1:1]$. Then 
in the affine chart $\cc^2_{z,w}=\{ t=1\}$ the dual   billiard   takes the form 2c2). 
This together with Proposition \ref{exun} implies that there exists a complex projective transformation 
$F$ fixing the complexified conic $\gamma$  and sending the dual billiard of type 2c2) to that of type 2c1). This 
proves the corresponding statement of the addendum to Theorem \ref{tcompl}.   Let us show 
that the real dual billiard of type 2c2)  has the integral $R_{c2}$ given by (\ref{exoc2}), and $R_{c2}=R_{c1}\circ F$ 
up to constant factor. To do this, we find $F$. Set  
$$\var:=e^{-\frac{2\pi i}3}.$$
We choose $F$ so that it sends the singular 
points $[0:0:1]$, $[0:1:0]$, $[1:1:1]$ of the structure 2c2) respectively to the singular points 
$[\var: \bar\var:1]$, $[\bar\var: \var:1]$, $[1:1:1]$ of the structure 2c1). 
It should preserve the conic $\gamma$. Hence, it sends its projective tangent lines at $[0:0:1]$, $[0:1:0]$ to 
those at the points $[\var: \bar\var:1]$, $[\bar\var:\var:1]$. Therefore, the intersection point 
$[1:0:0]$ of the two former tangent lines should be sent to the intersection point of the two latter tangent lines.
In the above affine chart $\cc^2_{z,w}$ the $z$-coordinate of the latter intersection point is found 
from the equation
$$\bar\var+2\var(z-\var)=\var+2\bar\var(z-\bar\var):$$
$$z=-\frac12, \ w=-\bar\var-\var=1.$$
Finally, the projective transformation $F$ should send the points $[1:0:0]$, $[0:1:0]$, $[0:0:1]$ to 
$[-\frac12:1:1]$, $[\bar\var:\var:1]$, $[\var:\bar\var:1]$. Hence,  its matrix   (normalized by appropriate scalar factor) 
takes the form  
\begin{equation}\left(\begin{matrix}-\frac12 & \la_2\bar\var & \la_3\var\\
1 & \la_2\var & \la_3\bar\var\\
1 & \la_2 & \la_3\end{matrix}\right), \ \ \la_2,\la_3\in\cc^*.\label{matr111}\end{equation}
The coefficients $\la_2$ and $\la_3$ are found from the following system of equations saying that the transformation 
$F$ should fix the point $[1:1:1]$: 
$$\begin{cases}\la_2(1-\var)+\la_3(1-\bar\var)=0\\
\frac32+\la_2(1-\bar\var)+\la_3(1-\var)=0.\end{cases}$$
We get that $\la_3=\var\la_2$, $\frac32+\la_2(1+\var-2\bar\var)=0$, $\la_2=\frac{\var}2$, $\la_3=\frac{\bar\var}2$, 
and the transformation $F$ is given by the matrix
\begin{equation}M:=\left(\begin{matrix}-\frac12 & \frac12 & \frac12\\
1 & \frac{\bar\var}2 & \frac\var2\\
1 & \frac\var2 & \frac{\bar\var}2\end{matrix}\right)\label{matr111}\end{equation}
Let us now calculate the pullback of the integral $R_{c1}$ under the projective transformation $F$. 
Writing $R_{c1}$ in the homogeneous coordinates $[z:w:t]$, we get 
$$R_{c1}([z:w:t])=\frac{(wt-z^2)^3}{(t^3+w^3-2zwt)^2}.$$
Applying the linear transformation given by the  matrix $M$ to the 
polynomial $Q_1(z,w,t)=wt-z^2$ in the numerator   yields 
$$Q_1\circ M(z,w,t)=\frac1{4}((2z+\bar\var w+\var t)(2z+\var w+\bar\var t)-(w+t-z)^2)=-\frac{3}{4}(wt-z^2).$$
Applying $M$ to the polynomial $Q_2(z,w,t)=t^3+w^3-2zwt$ in the denominator yields 
$$8Q_2\circ M(z,w,t)=(2z+\var w+\bar\var t)^3+
(2z+\bar\var w+\var t)^3$$
$$-2(w+t-z)(2z+\var w+\bar\var t)(2z+\bar\var w+\var t)$$
$$=16z^3+2w^3+2t^3-12z^2w-12z^2t-6zw^2-6zt^2-3w^2t-3wt^2$$
$$+24zwt+2(z-w-t)(4z^2+w^2+t^2-2zw-2zt-wt)$$
$$=24(z^3-z^2w-z^2t)-3(w^2t+wt^2)+30zwt$$
$$=3(8z^3-8z^2w-8z^2t-w^2t-wt^2+10zwt).$$
In the affine chart $\cc^2_{z,w}=\{ t=1\}$ we get 
$Q_1\circ M(z,w,1)=-\frac34(w-z^2)$, 
 $$Q_2\circ M(z,w,1)=\frac38(8z^3-8z^2w-8z^2-w^2-w+10zw).$$
Therefore,  $R_{c1}\circ F=R_{c2}$ up to constant factor. 

Case 2d): the complexified dual billiard on the conic has three singularities with residues $\frac43$, $1$, $\frac53$. 
The complex conjugation, which preserves the dual billiard, should fix them, since their residues are distinct. 
Therefore, applying a real projective transformation, we can and will consider that the underlying real conic  
is the parabola $\{ w=z^2\}$, and the singularities are respectively the points $(0,0)$, $(1,1)$ and its infinite point. 
The involution family defining the dual billiard is of the type 2d), by construction and Proposition \ref{exun}. 
It has integral $R_{2d}$ given by  (\ref{exodd}), due to Lemma \ref{l362}.
Theorem \ref{tgerm} and the addendums to Theorems \ref{tgerm} and \ref{tcompl} are proved.

\subsection{Case of closed curve. Proof of Theorem \ref{tclosed}} Let now $\gamma$ be a $C^4$-smooth 
closed curve equipped with an integrable dual   billiard structure. 
 The involutions $\sigma_P$ can be defined by just one 
convex closed invariant curve, and they depend continuously  
on $P\in\gamma$.  Let $R$ be a non-trivial rational first integral of the foliation by invariant curves. For every 
$P\in\gamma$ the restriction $R|_{L_P}$ is $\sigma_P$-invariant, since this holds in a neighborhood of the point $P$ in $L_P$ 
(by definition), and by analyticity. Therefore, $\gamma$ is a conic, by Theorem \ref{tgerm}, and it 
contains no singularity of the dual billiard. Hence, the dual billiard is given by a pencil of 
conics containing $\gamma$, since all the other rationally integrable dual billiards on conic listed in Theorem \ref{tgerm} have real 
singularities. Those conics of the pencil that are close enough to $\gamma$ and lie on its concave side
 are disjoint and form a foliation 
of a topological annulus adjacent to $\gamma$. Indeed, otherwise the pencil would consists of 
conics intersecting at some point $P_0\in\gamma$. But then $P_0$ would  be a singular point of the  dual billiard, 
see the proof of  Proposition \ref{propencil}. The contradiction thus obtained proves Theorem \ref{tclosed}. 

The basic set of the corresponding pencil lies in $\cp^2\setminus\rp^2$ and is described by the following obvious proposition. 

\begin{proposition} Let $\gamma\subset\rp^2$ be a regular conic equipped with a dual billiard structure given by a real pencil of conics. 
Let the basic set of the pencil contain no real points. Then 
it   consists of either four distinct points, or two distinct 
points. In the latter case the regular complexified conics of the pencil are tangent to each other at the two points of the basic set. 
\end{proposition}

\section{Integrable projective billiards. Proof of Theorem \ref{tgermproj} and 
its addendum}

Here we prove Theorem \ref{tgermproj} classifying rationally $0$-homogeneously integrable projective billiards and its addendum 
providing formulas for integrals (in Subsection 9.4). To do this, in Subsection 9.1 
we  prove Proposition \ref{ratmom} stating that such a billiard admits an integral that is a $0$-homogeneous rational function in the moment vector. 
Afterwards in Subsection 9.2  we prove Proposition \ref{procriter} stating 
that rational $0$-homogeneously integrability of a projective billiard is equivalent to rational 
integrability of its dual billiard. In Subsection 9.3 we prove Proposition \ref{procons}.

\subsection{The moment map and normalization of integral. Proof of Proposition \ref{ratmom}}

Recall that we identify the ambient Euclidean plane $\rr^2_{x_1,x_2}$ of a projective billiard with the plane 
$\{ x_3=1\}\subset\rr^3_{x_1,x_2,x_3}$, and we denote $r=(x_1,x_2,1)$. 
The geodesic flow has an universal invariant: the moment vector 
$$M:=[r,v]=(-v_2, v_1,\Delta(x,v)), \ \ \Delta(x,v):=x_1v_2-x_2v_1,$$
which separates any two orbits of the geodesic flow \cite{bolotin2}. This  implies that, 
{\it every integral of the projective billiard is a reflection-invariant function of $M$ and vice versa,}  as in \cite{bolotin2}. 

Consider  now a $C^2$-germ of planar curve equipped with a transversal line field,  a connected domain $U$ adjacent to  it and the 
projective billiard in $U$. (Or  a (global) projective billiard in some connected domain $U$ in $\rr^2$.) Let it have a first  integral $R(x,v)$ 
that is a rational $0$-homogeneous function in $v$ of degree uniformly bounded by some constant $d$. Let $W\subset\rr^3_{M_1,M_2,M_3}$ 
denote the image of the moment map $T\rr^2|_U\to\rr^3$, $(x,v)\mapsto[r,v]$, $x=(x_1,x_2)$. 

\begin{proposition} \label{prat}   Let us represent the above integral $R$ as a function of the moment $M$. The function $R$ is 
$0$-homogeneous in $M$: $R(\la M)=R(M)$ for every $\la\in\rr$; thus, it is well-defined on the tautological projection image 
$\mathbb P(W)=\pi(W\setminus\{0\})$. The image $\mathbb P(W)$ is a union of projective lines along which the function $R$ is rational 
of degree at most $d$. The family of the latter lines forms an open subset in the space $\rp^{2*}$ of lines. 
  \end{proposition}
   \begin{proof} As $x=(x_1,x_2)$ is fixed, the restricted moment map $v\mapsto M=[r,v]$ is a linear isomorphism of the tangent plane 
   $T_x\rr^2$ and the plane $r^\perp$ orthogonal to $r$. Therefore, the restriction to $r^\perp$ of the function $R$ is 
   rational $0$-homogeneous of degree no greater than $d$. This proves $0$-homogeneity of the function $R(M)$ and  
   well-definedness  of the function $R$ on 
   $\mathbb P(W)$. Let $\ell(x)\subset\rp^2$ denote the projective line 
   that is the projectivization of the  subspace $r^\perp$.  One has $\mathbb P(W)=\cup_{x\in U}\ell(x)$. 
   The function $R|_{\ell(x)}$ is rational for every $x\in U$, by rationality on $r^\perp$. 
   The map $\rr^2\to\rp^{2*}$, $x\mapsto\ell(x)$ is a  diffeomorphism onto the open subset of those projective lines 
   that do not pass through the origin in the affine chart $\rr^2_{x_1,x_2}=\{ x_3=1\}$. Hence, it maps $U$ onto an open subset in $\rp^{2*}$. This  proves the proposition.
   \end{proof}
   
 As is shown below, the statement of Proposition \ref{ratmom} is implied by Proposition \ref{prat} and the next proposition. 
\begin{proposition} \label{2rat} Let $d\in\nn$. Let a function $f(z,w)$ be defined on a 
neighborhood of the origin in $\rr^2_{z,w}$. Let 
it be rational  in each  variable, and let its degree in the variable $w$ be no greater than $d$. 
Then it is a rational  function of two variables. 
\end{proposition}
 \begin{proof} Let us write 
 $$f(z,w)=\frac{a_0(z)+a_1(z)w+\dots+a_d(z)w^d}{b_0(z)+b_1(z)w+\dots+b_d(z)w^d}.$$
 Fix $2d+1$ distinct points $w_0,\dots,w_{2d}$ close to zero. The functions $R_j(z):=f(z,w_j)$ 
 are rational. The system of $2d+1$ equations  
$$ f(z,w_j)=\frac{a_0(z)+a_1(z)w_j+\dots+a_d(z)w_j^d}{b_0(z)+b_1(z)w_j+\dots+b_d(z)w_j^d}=R_j(z)$$
in $2d+2$ unknown coefficients $a_s(z)$, $b_s(z)$ can be rewritten as a system of 
$2d+1$ linear equations on them (multiplying by denominator). For every $z$ it has a unique solution up to 
constant factor depending on $z$, since two rational functions in $w$ of degree at most $d$ cannot 
coincide at $2d+1$ distinct points. This follows from the fact that their difference, which is a rational 
function in $w$ of degree at most $2d$, cannot have more than $2d$ zeros. The  solution 
$(a_0(z),\dots,a_d(z),b_0(z),\dots,b_d(z))$ of the above linear systems 
can be normalized by constant factor so that its  components 
be expressed as rational functions of the parameters  $w_j$ and $R_j(z)$ of the  system. Therefore, 
$a_s(z)$ and $b_s(z)$ are rational functions in $z$. This proves the proposition.
\end{proof}

Let $V\subset\rp^{2*}$ denote the open set of lines from the last statement of Proposition \ref{prat}. Fix two distinct lines $\La_1,\La_2\in V$ and two distinct points $y_j\in\La_j$, $j=1,2$. 
Consider two pencils $\mcp_j$ of lines  through $y_j$. The function $R$ is rational along each line in $\mcp_j$ close to $\La_j$. 
Choosing affine chart  $\rr^2_{z,w}\subset\rp^2$ so that $y_1$, $y_2$ be the intersection points of the infinity line with the coordinate axes 
we get that $R$ is locally a rational function in each separate variable $z$, $w$. Therefore, it is locally rational in two variables, by 
Proposition \ref{2rat}. Hence, it is globally rational on all of $\mathbb P(W)$, by connectivity of the domain $U$, and hence, of the open subset 
$\mathbb P(W)$. Therefore, $R(M)$ is a  $0$-homogeneous 
rational function in $M$. The first part of  
Proposition \ref{ratmom} is proved. Let us prove its second part: independence of integrability on choise of side. Let 
a $C^2$-smooth germ of curve $C$ equipped with a transversal line field define a rationally $0$-homogeneously integrable 
projective billiard on one side from $C$. Then it admits an integral that is a rational $0$-homogeneous 
function $R(M)$ of the moment vector $M$ 
(the first part of Proposition \ref{ratmom}). The moment vector (and hence, the  integral) extends as a constant function 
along  straight lines crossing $C$ (treated as orbits of geodesic flow)  
from one side  of the curve $C$ to the other side. 
Invariance of the integral $R(M)$ under the billiard flow is equivalent to its reflection invariance. But reflection invariance depends only 
on the transversal line field and not on the choice of side. Therefore, if $R$ is an integral on one side, it will be automatically an integral on the 
other side. Proposition \ref{ratmom} is proved.

\subsection{Integrability and duality. Proof of Proposition \ref{procriter}}

The proof of  Proposition  \ref{procriter} is analogous to the arguments from \cite{bolotin2, bm, gl2}. 
 On the ambient projective plane $\rp^2_{[x_1:x_2:x_3]}\supset\rr^2$ we deal with the projective duality $\rp^{2*}\to\rp^2$ 
given by the orthogonal polarity. We use the following
 \begin{remark} \cite{bolotin2, bm, gl2}. For every $r=(x_1,x_2,1)\in\rr^3$ and $v\in T_{(x_1,x_2)}\rr^2$ consider the two-dimensional vector subspace in $\rr^3$ 
 generated by $r$ and $v$ (punctured at the origin). Let $L(r,v)\subset\rp^2$ denote the corresponding 
 projective line (its projectivization). 
The composition of the moment map $(r,v)\mapsto M=[r,v]$ and the tautological projection 
 $\rr^3_{M_1,M_2,M_3}\setminus\{0\}\to\rp^2_{[M_1,M_2,M_3]}$ sends each pair $(r,v)$ to the 
 point $L^*(r,v)$ dual to $L(r,v)$. 
 \end{remark}
 Consider a projective billiard on a curve $C\subset\rr^2_{x_1,x_2}$. Its dual  curve is identified with a curve  
 $\gamma=C^*\subset\rp^2_{[M_1:M_2:M_3]}$, see the above remark. Let 
 the dual billiard on $\gamma$ have a rational integral. It can be written as a $0$-homogeneous rational function $R(M_1,M_2,M_3)$. 
  The corresponding function 
 $R([r,v])$ is a rational $0$-homogeneous integral of the projective billiard. Indeed, its invariance under reflections acting on $v\in T_Q\rr^2$, 
 $Q\in C$,  follows from the above remark and the fact that duality conjugates 
 the billiard reflection acting on lines through $Q$ 
 to the dual billiard involution acting on the dual line $Q^*$.
 Conversely, let the projective billiard have a rational $0$-homogeneous integral. Then it can be written as 
 $R[r,v]$, where $R(M)$ is a rational $0$-homogeneous function,  by Proposition \ref{ratmom}. The function  
 $R(M)$ is an integral of the dual billiard, since $R[r,v]$ is an integral of the projective billiard and by the 
above conjugacy. Proposition \ref{procriter} is proved.

\subsection{Space form billiards on conics. Proof of Proposition \ref{procons}}

Let a projective billiard on a finitely punctured conic $C$ be a space form billiard with matrix $A$. 
In the case, when $A=\diag(1,1,0)$, the billiard is Euclidean, and each conic confocal to $C$ is a caustic. Analogous statement 
holds in the case of non-zero constant curvature, when $A=\diag(1,1,\pm1)$.  This implies the second statement of Proposition \ref{procons}.

Let us prove the converse. Let a transversal line field $\mcn$ on a punctured conic $C$ define a projective billiard having a complex conical caustic 
$S$. Let us show that it is projectively equivalent to a space form billiard with matrix  $\diag(1,1,-1)$. 
Let $\mcd\subset C$ denote the finite set of those points $Q\in C$ for which the line $L_Q$ tangent to $C$  at $Q$ is also tangent to $S$ 
at some point. 
For every $Q\in C^o:=C\setminus(\mcd\cup S)$ the line $\mcn(Q)$ is well-defined by harmonicity condition on the tuple of four distinct 
lines through $Q$:  
 $L_Q$, $\mcn(Q)$ and the complex lines $\La_1$, $\La_2$ through $Q$ tangent to $S$. It says that  
there exists a projective involution of the space $\cp^1$ of complex lines through $Q$ that fixes $L_Q$, $\mcn(Q)$ and permutes 
$\La_1$, $\La_2$.
Let $E_j=E_j(Q)$, $j=1,2$,  denote the  tangency points of the  lines $\La_j$ with $S$. 
Fix  coordinates $(x_1,x_2,x_3)$ on $\rr^3$ 
(homogeneous coordinates on $\rp^2\supset C$) in which $S=\{<Ax,x>=0\}$, $A=\diag(1,1,-1)$. Let us show that 
the projective billiard on $C^o$ is the space form billiard with the matrix $A$: for every $Q\in C^o$ 
the two-dimensional subspaces $H_T(Q), H_\mcn(Q)\subset\rr^3$ 
projected to the lines $L_Q$ and $\mcn(Q)$ respectively are orthogonal in the scalar product $<Ax,x>$.

 Fix a point $B\in\mcn(Q)\cap S$.  The four points $E_1$, $E_2$, $Q$, $B$ are distinct, and no three of them are collinear, since $Q\in C^o$.
There exists 
a  projective involution $\mathbf I:\cp^2\to\cp^2$ fixing the points of the line $QB$ and permuting $E_1$, $E_2$ (and hence, $\La_1$, $\La_2$). It  
fixes $\mcn(Q)$, and hence, 
$L_Q$, by  harmonicity. It preserves  $S$: the conic  
$\mathbf I(S)$ is  tangent to $S$ at  $E_1$ and $E_2$ and intersects $S$ at  $B\neq E_{1,2}$;  hence, 
 $\mathbf I(S)=S$. Thus, $\mathbf I$ is  the  projectivization of a non-trivial linear involution $\rr^3\to\rr^3$ preserving 
  the quadratic form $<Ax,x>$ and transversal two-dimensional subspaces $H_T(Q)$, $H_\mcn(Q)$ and acting trivially on $H_\mcn(Q)$. 
  This  implies orthogonality of the latter subspaces in the  scalar product $<Ax,x>$. Proposition \ref{procons} is proved.

\subsection{Proof of Theorem  \ref{tgermproj} and its addendum}

\begin{proof} {\bf of Theorem \ref{tgermproj}.} 
Let a nonlinear germ of $C^4$-smooth curve $C\subset\rr^2$ carry a transversal line field  $\mcn$ 
defining a 0-homogeneously 
rationally integrable projective billiard. Then the dual   
billiard  on the dual curve $\gamma=C^*$  is rationally integrable (Proposition \ref{procriter}). 
Let $C'\subset C$ denote the complement of the curve $C$ to the set of its inflection points, i.e., points 
where the geodesic curvature vanishes. (A priori the set of inflection points may 
contain a straightline interval.) The dual to $C'$ is a union of $C^4$-smooth arcs of the curve $\gamma$.  The latter 
arcs are conics, by Theorem \ref{tgerm}. Hence, $C'$ is a union of conical arcs. The curve $C$ being $C^4$-smooth, 
the boundary points of the set $C'$ are not inflection points, and adjacent conical arcs paste  $C^4$-smoothly. 
This implies that  $C=C'$ is a conic. 

The rationally $0$-homogeneously integrable projective billiards on a (punctured) conic $C$ are exactly those dual to the rationally integrable 
dual billiards on (punctured) conic $\gamma$ (Proposition \ref{procriter}). Thus, it suffices to  find the projective billiards dual to all the integrable  dual billiards in Theorem \ref{tgerm}. In each of these projective billiards the transversal line field $\mcn$ is defined on 
 $$C^o=C\setminus(\text{at most four points}).$$   
Case 1): the dual   billiard structure on a (punctured) conic $\gamma$ is given by a pencil of conics. 
Then the complexified conic dual to any regular conic from the pencil 
is a complex caustic of the projective billiard on $C$. This together with Proposition \ref{procons} implies that the projective billiard is a space form 
billiard, whose space form matrix can be chosen $\diag(1,1,-1)$.

To treate the other cases, let us introduce the next notations. For  $Q\in C^o$ set 
\begin{equation} P=L_Q^*:=\text{ the point  dual  to the  line  } L_Q; \ \ P\in \gamma=C^*;
\label{whp}\end{equation}
$$\wt P=\mcn^*(Q):= \text{ the point dual to  the projective line tangent to } \mcn(Q);$$
$$\wt P \text{ lies in the line } Q^*=L_{P} \text{ tangent to } \gamma \text{ at } P.$$
 \begin{proposition} Consider the dual billiard on $\gamma$ for the projective billiard defined by the line field $\mcn$. 
 For every $Q\in C^o$ the point $\wt P$ is the unique fixed point distinct from $P$ of the 
 dual billiard involution $\sigma_{P}:Q^*\to Q^*$.
 \end{proposition}
 The proposition follows from definition. 
 
 In what follows for every rationally integrable dual billiard from  Theorem \ref{tgerm}, cases 2a)--2d), we find the above   fixed points 
$\wt P$  of the corresponding involutions. Their dual lines $\mcn(Q)=\wt P^*$ form the line field 
 defining the corresponding projective billiard.  To do this, 
 we work in homogeneous coordinates $[z:w:t]$ in the ambient projective plane $\rp^2\supset\gamma$ in which 
 $$\gamma=\{ wt-z^2=0\}; \ \ \gamma=\{ w=z^2\} \text{ in the affine chart } \rr^2_{z,w}=\{ t=1\}.$$
The curve $C$ is projective dual to $\gamma$ with respect to the duality  $\rp^{2*}\to\rp^2_{[z:w:t]}$ given by the 
 orthogonal polarity. We will work with the curve $C$ in the new homogeneous coordinates $[x_1,x_2,x_3]$ given by   
the projectivization $[F]:\rp^2_{[z:w:t]}\to\rp^2_{[x_1,x_2,x_3]}$ of the linear map 
 \begin{equation}F:(z,w,t)\mapsto(x_1,x_2,x_3):=(\frac z2, t, w).\label{dualf}\end{equation} 
 For every point $Q\in C$ let $P\in\gamma$ be the corresponding point in (\ref{whp}).  Set  
 $$ z_0:=z(P).$$ 
 
 {\bf Claim 14.} {\it In the coordinates $[x_1:x_2:x_3]$ given by (\ref{dualf}) one has}
  \begin{equation} Q=[-z_0:z_0^2:1], \ \ C=\{ x_2x_3=x_1^2\}; \ \ C\cap\{ x_3=1\}=\{ x_2=x_1^2\}.\label{gcon}\end{equation} 
  
  \begin{proof}   The projective tangent line $Q^*$ to 
 $\gamma$ at the point $P$ and its orthogonal-polar-dual point $Q\in\rp^2_{[z:w:t]}$ are given by the equations 
 $$Q^*=\{-2z_0z+w+z_0^2t=0\}, \ Q=[-2z_0:1:z_0^2]\in C.$$
 In the coordinates $[x_1:x_2:x_3]$ one has $Q=[-z_0:z_0^2:1]$.
  \end{proof}
 
{\bf Claim 15.} {\it Let $Q=(x_1,x_2)\in C$ in the affine chart $\rr^2_{x_1,x_2}=\{ x_3=1\}$. 
Consider a $Q$-parametrized family   of points $B(Q)\in Q^*$. Let $z(B(Q))=g(z_0)$; 
$g(z_0)$ is a function of $z_0$. For every $Q\in C$ the dual to the point $B(Q)$ 
  is the line through $Q$ directed by the vector $(\dot x_1,\dot x_2)=(1,-2g(-x_1))$ at $Q$.}
  
  \begin{proof} In the affine chart $\rr^2_{z,w}=\{ t=1\}\subset\rp^2_{[z:w:t]}$ one has 
  $$w(B(Q))=2z(B(Q))z(P)-z^2(P)=2g(z_0)z_0-z_0^2.$$
Thus, $[z:w:t](B(Q))=[g(z_0):2g(z_0)z_0-z_0^2: 1]$. Hence, the dual line 
  $B^*(Q)\subset\rp^2_{[z:w:t]}$ is given by the equation $g(z_0)z+(2g(z_0)z_0-z_0^2)w+t=0$. Writing it in the coordinates 
  $(x_1,x_2,x_3)$, see (\ref{dualf}), yields 
  $2g(z_0)x_1+(2g(z_0)z_0-z_0^2)x_3+x_2=0$. Thus, in the affine chart $\rr^2_{x_1,x_2}$ the latter line is directed by the 
  vector $(1, -2g(z_0))=(1,-2g(-x_1(Q))$, by (\ref{gcon}). 
\end{proof}

 Case 2a): the dual   billiard structure on $\gamma$ is given by the family of involutions 
 $\sigma_{P}:L_{P}\to L_{P}$ taking the form 
 $$\sigma_{P}:\zeta\mapsto\frac{(\rho-1)\zeta-(\rho-2)}{\rho\zeta-(\rho-1)}, \ \zeta=\frac z{z_0}, \ 
 \rho=2-\frac2{2N+1}, \text{ or } \rho=2-\frac1{N+1}.$$
 The fixed point $\wt P\in L_{P}$ of the involution $\sigma_{P}$ has $\zeta$-coordinate $\frac{\rho-2}\rho$, 
 hence 
 $$z(\wt P)=g(z_0), \ g(\theta)=\frac{\rho-2}\rho \theta.$$
 Therefore, the dual line $\wt P^*$ is directed by the vector $(1,-2g(-x_1(Q)))=(1,\frac{2(\rho-2)}\rho x_1(Q))$ at $Q$. 
  Thus, the line field $\mcn$ defining the projective billiard on $C$ is directed by 
 the vector field $(\rho,2(\rho-2)x_1)$ on $C$. 
 The latter field  is tangent to the level curves of the quadratic polynomial
$\mcq_{\rho}(x_1,x_2):=\rho x_2-(\rho-2)x_1^2$. Thus, it has type 2a) from Theorem \ref{tgermproj}.
 
 Cases 2b), 2c), 2d): in the coordinate $u:=z-z_0$, $z_0=z(P)$, 
 the involutions $\sigma_{P}:Q^*\to Q^*$ take the form 
$$ \sigma_{P}: u\mapsto-\frac u{1+f(z_0)u}$$
$$ f=f_{b1}(z):=\frac{5z-3}{2z(z-1)} \text{ (type 2b1))},  \text{ or } \ 
   f=f_{b2}(z):=\frac{3z}{z^2+1} \text{ (type 2b2)),}$$
$$f=f_{c1}(z):=\frac{4z^2}{z^3-1} \text{ (type 2c1))}, \  \text{ or } \ 
   f=f_{c2}(z):=\frac{8z-4}{3z(z-1)} \text{ (type 2c2))},$$
   $$f=f_d(z):=\frac{7z-4}{3z(z-1)}.$$

The $u$- and $z$-coordinates of the fixed point $\wt P$ of the involution $\sigma_P$ are 
   $$u(\wt P)=-\frac2{f(z_0)}, \ z(\wt P)=z_0-\frac2{f(z_0)}.$$

   Subcase 2b1). One has 
   $$z(\wt P)=z_0-\frac{4z_0(z_0-1)}{5z_0-3}=g(z_0), \ \ g(\theta)=\frac{\theta(\theta+1)}{5\theta-3}.$$
  Therefore, the dual line $\wt P^*$ is directed by the vector 
  $$(1,-2g(-x_1))=(1,\frac{2x_1(x_1-1)}{5x_1+3})=(1,\frac{2(x_2-x_1)}{5x_1+3}), \ x_j=x_j(Q).$$
  Here we have substituted $x_1^2=x_2$, since $Q\in C$.  Hence, the line field $\mcn$ is directed by 
the vector field $(5x_1+3, 2(x_2-x_1))$ and thus, has type 2b1).
   
   Subcase 2b2). One has
   $$z(\wt P)=z_0-\frac{2(z_0^2+1)}{3z_0}=g(z_0), \ \ g(\theta)=\frac{\theta^2-2}{3\theta},$$
and the line field $\mcn$ on $C$ is directed by the vector field 
$(1,-2g(-x_1))=(1,\frac{2(x_1^2-2)}{3x_1})$. Or equivalently, by the 
 vector field $(3x_1,2x_2-4)$, since  $x_1^2=x_2$ on $C$. Thus, it has type 2b2).
 
 Subcase 2c1). One has 
 $$z(\wt P)=z_0+\frac{1-z_0^3}{2z_0^2}=g(z_0), \ \ g(\theta)=\frac{\theta^3+1}{2\theta^2},$$
 the  field  $\mcn$ is directed by the vector field $(1,-2g(-x_1))=(1,\frac{x_1^3-1}{x_1^2})=(1,\frac{x_1x_2-1}{x_2})$ on $C$. 
 Or equivalently, by  $(x_2, x_1x_2-1)$. We get type 2c1).

Subcase 2c2). One has 
$$z(\wt P)=z_0-\frac{3z_0(z_0-1)}{4z_0-2}=g(z_0), \ \ g(\theta)=\frac{\theta(\theta+1)}{4\theta-2},$$
the line field $\mcn$ is directed by the vector field $(1,\frac{x_1(x_1-1)}{2x_1+1})=(1,\frac{x_2-x_1}{2x_1+1})$ on $C$. 
Or equivalently, by the field  $(2x_1+1, x_2-x_1)$. Hence, it has type 2c2). 

Subcase 2d). One has 
$$z(\wt P)=z_0-\frac{6z_0(z_0-1)}{7z_0-4}=g(z_0), \ \ g(\theta)=\frac{\theta(\theta+2)}{7\theta-4},$$
the line field $\mcn$ is directed by the vector field $(1,\frac{2x_1(x_1-2)}{7x_1+4})=(1,\frac{2x_2-4x_1}{7x_1+4})$. 
Or equivalently, by the field  $(7x_1+4,2x_2-4x_1)$. Hence, it has type 2d).
This proves Theorem \ref{tgermproj}.
\end{proof}

\begin{proof} {\bf of  the addendum to Theorem \ref{tgermproj}.} Consider the real conic $C=\{ x_2x_3=x_1^2\}\subset\rp^2_{[x_1:x_2:x_3]}$ 
equipped with a projective billiard structure from  Theorem \ref{tgermproj}. 
Let $F:\rr^3\to\rr^3$ be the transformation from (\ref{dualf}):  
$$(z,w,t):=F^{-1}(x_1,x_2,x_3)=(2x_1,x_3,x_2).$$
Let $[F]^{-1}:\rp^2\to\rp^2$ denote the projectivization of the transformation $F^{-1}$. 
Recall that  the orthogonal-polar-dual  to the conic $[F]^{-1}(C)$ is  
the conic $\gamma=C^*=\{ wt=z^2\}$, see Claim 14 above, and the post-composition of $[F]^{-1}$ with the 
duality sends  the projective billiard on $C$  to 
the corresponding dual   billiard   on $\gamma$ given by Theorem \ref{tgerm}. See the above proof of 
Theorem \ref{tgermproj}. 

\begin{proposition} \label{printeg} Let $R$ be a rational integral of the dual billiard on $\gamma$ written as a 
$0$-homogeneous rational function $R(z,w,t)$. Then the function 
\begin{equation}\wt R(x,v):=R(v_2,-2\Delta, -2v_1), \ \ \Delta:=x_1v_2-x_2v_1,\label{rxv}\end{equation}
is a $0$-homogeneous rational integral of the projective billiard on $C$.
\end{proposition}
\begin{proof} In the affine charts $\rr^2_{x_1,x_2}=\{ x_3=1\}\subset\rp^2$, $\rr^2_{z,w}=\{ t=1\}\subset\rp^2$ 
in the source and  image the map $[F]^{-1}$ and its differential take the form 
$$[F]^{-1}:(x_1,x_2)\mapsto(z,w):=\left(\frac{2x_1}{x_2},\frac1{x_2}\right),$$
$$d[F]^{-1}(x_1,x_2)(v)=\hat v:=\left(-\frac{2\Delta}{x_2^2}, -\frac{v_2}{x_2^2}\right), \ \Delta:=x_1v_2-x_2v_1.$$
Set $r:=(z,w,1)=(\frac{2x_1}{x_2},\frac1{x_2},1)$, and let us identify $\hat v=(\hat v_1,\hat v_2)$ with  
$$\hat v=(\hat v_1,\hat v_2,0)=\left(-\frac{2\Delta}{x_2^2}, -\frac{v_2}{x_2^2},0\right)\in\rr^3_{z,w,t}.$$
 The function $R([r,\hat v])$ is an integral of the $[F]^{-1}$-pushforward 
of the projective billiard on $C$, which is a projective billiard on $[F]^{-1}(C)$; see  
Proposition \ref{procriter}. Therefore, $R([r,\hat v])$ written as a function of $x=(x_1,x_2)$ and $v=(v_1,v_2)$ 
 is an integral of the projective billiard on $C$. One has 
$$[r,\hat v]=\frac1{x_2^2}(v_2,-2\Delta,-2v_1).$$
Hence, $R([r,\hat v])$  takes the form (\ref{rxv}), by $0$-homogeneity. Proposition \ref{printeg} is proved.
\end{proof}

 In what follows we calculate the integral (\ref{rxv}) explicitly  for  the integrals $R$  listed in the addendum to Theorem \ref{tgerm}.

  Case 1): the dual billiard structure on $\gamma$ is given by a pencil of conics containing $\gamma$. 
  Then it admits a quadratic rational integral $R$, which is a ratio of two quadratic forms in $(M_1,M_2,M_3)$. 
  The corresponding integral (\ref{rxv}) is  a ratio of two quadratic forms in the vector $(v_2,-2\Delta,-2v_1)$. 
  
  Case 2a1):  $\rho=2-\frac2{2N+1}$, the integral $R(z,w)$ written in the affine chart $\rr^2_{z,w}=\{ t=1\}$ 
   has type (\ref{exot1}). In the homogeneous coordinates 
  $[z:w:t]$ it takes the form
  $$R(z,w,t)=\frac{(wt-z^2)^{2N+1}}{t^2\prod_{j=1}^N(wt-c_jz^2)^2}, \ \ 
  c_j=-\frac{4j(2N+1-j)}{(2N+1-2j)^2}.$$
  Substituting 
  \begin{equation}(z,w,t)=(v_2,-2\Delta, -2v_1), \ \Delta=x_1v_2-x_2v_1,\label{substor}\end{equation} 
   to $R$, see (\ref{rxv}), and multiplying by 4 yields   integral (\ref{r2a1v}): 
  $$\Psi=\Psi_{2a1}(x_1,x_2,v_1,v_2):=\frac{(4v_1\Delta-v_2^2)^{2N+1}}{v_1^2\prod_{j=1}^N(4v_1\Delta-c_jv_2^2)^2}.$$

Case 2a2): $\rho=2-\frac1{N+1}$, the integral $R(z,w)$ has type (\ref{exot2}), and in 
the homogeneous coordinates
$$R(z,w,t)=\frac{(wt-z^2)^{N+1}}{zt\prod_{j=1}^N(wt-c_jz^2)}, \ \  c_j=-\frac{j(2N+2-j)}{(N+1-j)^2}.$$
Substitution (\ref{substor}) and multiplication by $-2$ yield (\ref{r2a2v}): 
$$\Psi=\Psi_{2a2}(x_1,x_2,v_1,v_2)=\frac{(4v_1\Delta-v_2^2)^{N+1}}{v_1v_2\prod_{j=1}^N(4v_1\Delta-c_jv_2^2)}.$$

Case 2b1): $R(z,w)$ has type (\ref{exo2bnew}), and in the homogeneous coordinates 
$$R(z,w,t)=\frac{(wt-z^2)^2}{(wt+3z^2)(z-t)(z-w)}.$$
Substitution (\ref{substor})   yields (\ref{r2b1v}):
 $$\Psi=\Psi_{2b1}(x_1,x_2,v_1,v_2)=\frac{(4v_1\Delta-v_2^2)^2}{(4v_1\Delta+3v_2^2)(2v_1+v_2)(2\Delta+v_2)}.$$
 
 Case 2b2): $R(z,w)$ has type (\ref{exo2bnew2}), and in the homogeneous coordinates 
$$R(z,w,t)=\frac{(wt-z^2)^2}{(z^2+w^2+wt+t^2)(z^2+t^2)}.$$
Substitution (\ref{substor}) yields (\ref{r3b2v}):
$$\Psi=\Psi_{2b2}(x_1,x_2,v_1,v_2)=\frac{(4v_1\Delta-v_2^2)^2}{(v_2^2+4\Delta^2+
4v_1\Delta+4v_1^2)(v_2^2+4v_1^2)}.$$

Case 2c1): $R(z,w)$ has type (\ref{exo2b}), and in the homogeneous coordinates 
$$R(z,w,t)=\frac{(wt-z^2)^3}{(t^3+w^3-2zwt)^2}.$$
 Substitution (\ref{substor})  and multiplication by $64$ yield (\ref{r2c1v}): 
 $$\Psi=\Psi_{2c1}(x_1,x_2,v_1,v_2)=\frac{(4v_1\Delta-v_2^2)^3}{(v_1^3+\Delta^3+
 v_1v_2\Delta)^2}$$
 
 Case 2c2): $R(z,w)$ has type (\ref{exoc2}), and in the homogeneous coordinates 
$$R(z,w,t)=\frac{(wt-z^2)^3}{(8z^3-8z^2w-8z^2t-w^2t-wt^2+10zwt)^2}.$$
Substituting (\ref{substor}) and multiplying by 64 yields (\ref{r2c2v}):
$$\Psi=\Psi_{2c2}(x_1,x_2,v_1,v_2)=\frac{(4v_1\Delta-v_2^2)^3}{(v_2^3+2v_2^2v_1+(v_1^2+2v_2^2+5v_1v_2)\Delta+v_1\Delta^2)^2}.$$

Case 2d): $R(z,w)$ is as in (\ref{exodd}), and in the homogeneous coordinates
$$R(z,w,t)=\frac{(wt-z^2)^3}{(wt+8z^2)(z-t)(wt^2+8z^2t+4w^2t+5wz^2-14zwt-4z^3)}.$$
Substituting (\ref{substor}) and multiplying by $-8$ yields (\ref{r2dv}): 
$$\Psi=\Psi_{2d}(x_1,x_2,v_1,v_2)$$
$$=\frac{(4v_1\Delta-v_2^2)^3}{(v_1\Delta+2v_2^2)(2v_1+v_2)(8v_1v_2^2+2v_2^3+(4v_1^2+5v_2^2+28v_1v_2)\Delta+16v_1\Delta^2)}.$$
The addendum to Theorem \ref{tgermproj} is proved.
\end{proof}

\section{Billiards with complex algebraic caustics. Proof of Theorems \ref{talg1} and \ref{talg2}}
\subsection{Case of  Euclidean billiard. Proof of Theorem \ref{talg1}}

Let $C\subset\rr^2_{x_1,x_2}=\{ x_3=1\}\subset\rr^3_{x_1,x_2,x_3}$ be a $C^2$-smooth connected curve. We identify its ambient  plane with the 
 affine chart  $\{ x_3=1\}\subset\rp^2_{[x_1:x_2:x_3]}$. 
Let $\gamma=C^*\subset\rp^2_{[x_1:x_2:x_3]}$ be its orthogonal-polar dual curve. 
Consider the usual billiard on  $C$. Its dual billiard on $\gamma$ is given by {\it Bialy--Mironov  
angular symmetries} $\sigma_P:L_P\to L_P$, $P\in\gamma$, defined as follows: $\sigma_P(P)=P$; $\sigma_P$ permutes  points 
 $a^*, b^*\in L_P$, if and only if  the lines $Oa^*$, $Ob^*$  are symmetric with respect to the line $OP$. 
See Fig. 5 below.  Here $O=(0,0)\in\rr^2$. 
\begin{remark}
The Bialy -- Mironov angular billiard was used in the solution of Bolotin's polynomial version of Birkhoff Conjecture 
\cite{bm, bm2, gl2}. 
\end{remark}
 \begin{figure}[ht]
  \begin{center}
   \epsfig{file=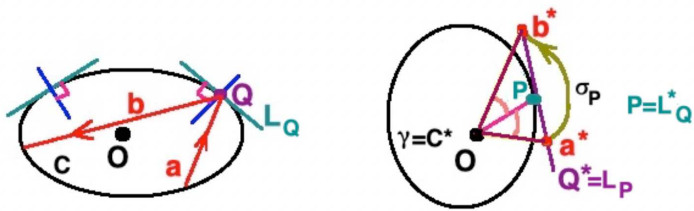, width=25em}
   \vspace{-0.3cm}
   \caption{Euclidean  billiard and its dual: Bialy -- Mironov angular billiard}
        \label{fig:4}
  \end{center}
\end{figure}
\begin{proposition} \label{pinvlp} Let $C$, $\gamma$, $\sigma_P$ be as above. 
 Let  the billiard in $C$ have a complex algebraic 
caustic $S$. Let $S^*$ be its complex  projective dual, and let 
$H(x_1,x_2,x_3)$ be its defining homogeneous polynomial: $S^*=\{ H=0\}$, and $H$ has the minimal possible 
degree. Set 
$$d:=\deg H, \ \  R(x_1,x_2,x_3):=\frac{H^2(x_1,x_2,x_3)}{(x_1^2+x_2^2)^{d}}.$$
The function $R$ 
is an integral of the Bialy--Mironov angual billiard on $\gamma$.
\end{proposition}
\begin{proof} Set 
$$\ii:=\{ x_1^2+x_2^2=0\}\subset\cp^2_{[x_1:x_2:x_3]}.$$
For every $P\in\gamma$ the complexification of the angular symmetry $\sigma_P:L_P\to L_P$ is the projective involution of 
 the complexified  line $L_P$ that permutes its intersection points with $\ii$, see \cite{bm}, \cite[proposition 2.18]{gl2}. 
 Thus, it leaves invariant  polar and zero loci  $L_P\cap \ii$, $L_P\cap S^*$  of the rational function $R|_{L_P}$. 
One has $S^*\not\subset\ii$, since  a caustic  contains no straight line. 
This together with Proposition \ref{invzero} applied to the involution $\sigma_P$ implies non-constance and $\sigma_P$-invariance of the restriction $R|_{L_P}$ 
and proves Proposition \ref{pinvlp}.
\end{proof} 

Thus, the dual billiard structure on $\gamma$ is rationally integrable. Therefore, $\gamma$ (and hence, 
$C$) lies in  a conic. In more detail, 
if $C$ were $C^4$-smooth, then this would follow from Theorem \ref{tgerm}. Let us treat the case, when $C$ 
is $C^2$-smooth.  The polar locus of the integral $R$ lies in $\ii$, and $R|_{L_P}$ is invariant under the angular symmetry. 
Therefore, the billiard  on $C$ is polynomially integrable, 
 see  the discussion on p. 1004 in \cite{gl2}, and hence, $C$ is a conic, by 
 \cite[theorem 1.6]{gl2}. Here is a more detailed explanation. The complex Zariski closure of the curve $\gamma$ is 
 an algebraic curve (Proposition \ref{proalg2}). The family $\sigma_P$ extends to a singular dual billiard structure on each its 
  non-linear  irreducible component, with integral $R$ having polar locus  in $\ii$. Hence, each  component is 
  a conic, by   \cite[theorem 1.25]{gl2}. Thus, 
  $C$ is a union of conical arcs. Different conical arcs (if any) should be confocal, see the discussion in \cite[subsection 6.2]{gl2}. 
  Any two intersecting confocal conics are orthogonal. This together with $C^2$-smoothness of the curve $C$ 
  implies that  $C$ lies in a conic; see \cite[subsection 6.3]{gl2} for more details.  
 Theorem \ref{talg1} is proved.
 \subsection{Case of projective billiard. Proof of Theorem \ref{talg2}}
Let $S_1$ and $S_2$ be two complex algebraic caustics. Let $S_1^*$, $S_2^*$ be their dual curves. Let $d_j$ denote the 
degrees of the curves $S_j^*$, and let $\mcp_j$ be their defining polynomials of degrees $d_j$.  
The dual curve $\gamma=C^*$ equipped with the corresponding dual billiard structure has a non-constant rational integral 
$$R:=\frac{\mcp_1^{2d_2}}{\mcp_2^{2d_1}},$$
as in the above proof of Proposition \ref{pinvlp}. Therefore, $\gamma$, and hence, $C$ is a conic, by Theorem \ref{tgerm}. 
Theorem \ref{talg2} is proved.

\section{Acknowledgements}

I am grateful to Sergei Tabachnikov for  introducing me to projective billiards, statement of conjecture and helpful discussions. 
I wish to thank  Sergei Bolotin,  Mikhail Bialy, Andrey Mironov and Eugenii Shustin for helpful discussions.

\end{document}